\documentclass{daj}

\usepackage{amssymb,verbatim}
\usepackage{amsmath,amsfonts}
\usepackage{amsthm}
\usepackage{url}
\usepackage{graphicx} 
\usepackage{enumitem} 
\usepackage{hyperref} 
\hypersetup{colorlinks} 
\usepackage{bbm} 
\usepackage{mathrsfs} 
\usepackage{cancel} 

\setlist[enumerate]{
	label=\textnormal{({\roman*})},
	ref={\roman*}}

\makeatletter
\def\th@plain{%
	\thm@notefont{}
	\itshape 
}
\def\th@definition{%
	\thm@notefont{}
	\normalfont 
}
\makeatother

\newtheorem{thm}{Theorem}[section]
\newtheorem{lemma}[thm]{Lemma}

\newtheorem{cor}[thm]{Corollary}
\newtheorem{prop}[thm]{Proposition}
\newtheorem{conj}[thm]{Conjecture}
\newtheorem{conv}[thm]{Convention}

\newtheorem{sublemma}[thm]{Sublemma}

\newtheorem{rem}[thm]{Remark}

\numberwithin{figure}{section}
\numberwithin{equation}{section}

\def\emp{\nothing}
\def\sq{\square}

\def\zz{\mathbb Z}
\def\nn{\mathbb N}

\def\rr{\mathbb R}

\def\pp{\mathbb P}

\def\sm{\smallsetminus}

\def\Om{\Omega}

\def\la{\lambda}
\def\ga{\gamma}
\def\si{\sigma}
\def\de{\delta}
\def\ep{\ve}
\def\al{\alpha}
\def\be{\beta}
\def\om{\omega}

\def\ve{\varepsilon}

\def\vp{\varphi}

\def\cL{\mathcal L}

\def\CB{\mathcal B}

\def\ssu{\subset}

\def\<{\langle}
\def\>{\rangle}

\def\DET{\text{{\rm DET}}}

\def\rM{ {\text {\rm M} } }

\def\vt{\ze}

\def\St{{\text {\rm Stab} } }

\def\0{{\mathbf 0}}

\def\nothing{\varnothing}

\def\.{\hskip.06cm}
\def\ts{\hskip.03cm}

\def\by{{\textbf{y}}}

\def\bX{\textbf{\textit{X}}}
\def\bZ{\textbf{\textit{Z}}}
\def\bcx{\textbf{\textit{x}}}

\def\La{\Lambda}

\def\ze{{\zeta}}

\newcommand{\wi}{\mathrm{width}}

\newcommand{\SSYT}{\operatorname{SSYT}}
\newcommand{\SYT}{\operatorname{{\rm SYT}}}

\def\.{\hskip.06cm}
\def\ts{\hskip.03cm}

\def\nin{\noindent}

\def\qp{\textrm{q}}
\def\eG{\Phi}
\def\SP{{\textsc{\#P}}}

\def\Pb{{\text{\bf P}}}
\def\Path{{\text{\rm Path}}}

\DeclareMathOperator{\Ac}{\mathcal{A}} 
\DeclareMathOperator{\Bnsim}{\ \overset{\CB}{\nsim} \ } 
\DeclareMathOperator{\Bsim}{\ \overset{\CB}{\sim} \ } 
\DeclareMathOperator{\eBnsim}{\ {\nsim} \ } 
\DeclareMathOperator{\eBsim}{\ {\sim} \ } 
\DeclareMathOperator{\Cc}{\mathcal{C}} 
\DeclareMathOperator{\Dc}{\mathcal{D}} 
\DeclareMathOperator{\Eb}{\mathbb{E}} 
\DeclareMathOperator{\Ec}{\mathrm{FT}} 
\DeclareMathOperator{\FT}{\mathrm{FT}} 
\DeclareMathOperator{\ED}{\mathrm{ED}} 
\DeclareMathOperator{\Gc}{\mathcal{G}} 
\DeclareMathOperator{\Pbl}{\Pb_{\lambda}} 
\DeclareMathOperator{\Pblm}{\Pb_{\lambda/\mu}} 
\DeclareMathOperator{\Rb}{\mathbb{R}} 
\DeclareMathOperator{\Zb}{\mathbb{Z}} 

\dajAUTHORdetails{%
	title = {Sorting Probability for Large Young Diagrams}, 
	author = {Swee Hong Chan, Igor Pak, and Greta Panova},
	plaintextauthor = {Swee Hong Chan, Igor Pak, and Greta Panova},
	keywords = {1/3--2/3 conjecture, hook-length formula, linear extension, Schur function, sorting probability, standard Young tableau},
}   

\dajEDITORdetails{%
	year={2021},
	number={24},
	received={27 July 2020},   
	published={30 November 2021},  
	doi={10.19086/da.30071},       
}   

\begin{document}
\begin{frontmatter}[classification=text]
	
	\title{Sorting Probability for Large Young Diagrams} 
	
	\author[swee]{Swee Hong Chan}
	\author[igor]{Igor Pak}
	\author[greta]{Greta Panova}

	\begin{abstract}
		For a finite poset $P=(X,\prec)$, let $\cL_P$ denote the set
		of linear extensions of~$P$. The \emph{sorting probability} \ts
		$\de(P)$ \ts is defined as
		$$
		\de(P) \. := \. \min_{x,y\in X} \. \bigl| \Pb\ts[L(x)\leq L(y) ] \ts - \ts \Pb\ts[L(y)\leq L(x) ] \bigr|\.,
		$$
		where $L \in \cL_P$ is a uniform linear extension of~$P$.  We give asymptotic
		upper bounds on sorting probabilities for posets associated with large
		Young diagrams and large skew Young diagrams, with bounded number
		of rows.
	\end{abstract}
\end{frontmatter}

%
%
%
%
%

\section{Introduction} \label{s:intro}
\emph{Random linear extensions} of finite posets occupy an unusual
place in combinatorial probability by being remarkably interesting
with numerous applications, and at the same time by being unwieldy
and lacking general structure.
One reason for this lies in the broad nature of posets, when some
special cases are highly structured, extremely elegant and well
studied, while there is no universal notion of ``large poset'' or
``random poset'' in the opposite extreme.
As a consequence, the results in the area tend to range widely
across the \emph{generality spectrum}:
from weaker results for large classes of posets to stronger results
for smaller classes of posets.
	
In this framework, the famous \ts \emph{$\frac13$ -- $\frac23$ \ts Conjecture}~\ref{c:1323}
is very surprising in both the scope and precision, as it bounds the
\emph{sorting probability} \ts $\de(P)\le \frac13$ \ts for  \emph{all} finite posets~$P$.
There are numerous partial results on the conjecture, as well as the
Kahn--Saks general upper bound \ts $\de(P)\le \frac{5}{11}$.
At the same time, the asymptotic analysis of $\de(P)$
remains out of reach even for the most classical examples.
In this paper we obtain sharp asymptotic upper bounds on $\de(P)$
for large Young diagrams and large skew Young diagrams.  These are the
first asymptotic results of this type, as we are moving down the
generality spectrum.
	
	\smallskip
	
	\subsection{Sorting probability} \label{ss:intro-def}
	Let $P=(X,\prec)$ be a finite poset with $n=|X|$ elements.
	A \emph{linear extension}~$L$ of $P$ is an order preserving bijection \ts
	$L: X\to [n]=\{1,\ldots,n\}$, so that $x\prec y$ implies $L(x)<L(y)$ for all $x,y\in X$.
	The set of linear extensions is denoted $\cL(P)$, and $e(P)=|\cL(P)|$ is the \emph{number
		of linear extensions} of~$P$.
	
	The \emph{sorting probability of two elements} $x,y\in X$,
	$x\ne y$, is defined as
\begin{equation}\label{eq:sort-xy}
\de(P; \ts x,y) \, := \,\. \Bigl|\Pb\bigl[L(x) < L(y)\bigr] \. - \. \Pb\bigl[L(y) <L(y)\bigr]\Bigr|\.,
\end{equation}
	where the probability is over uniform linear extensions  $L\in \cL(P)$.
	This is a measure of how independent random linear extensions on
	elements $x$ and~$y$ are.  The \emph{sorting probability}\footnote{There
		seem to be multiple conflicting notations for variations of the
		sorting probability used in the literature. Notably, in~\cite{BFT,Sah}
		the notation $\de(P)$ means what we denote by \ts $\frac12\bigl(1-\de(P)\bigr)$.
		We hope this will not lead to confusion.} of~$P$ is defined as:
	\begin{equation}\label{eq:sort-prob-def}
	\de(P) \, := \, \min_{x,y\in X, \, x\ne y} \. \de(P; \ts x,y).
	\end{equation}
	Clearly, $\de(P)=1$ when $P$ is a chain, since all pairs of elements are comparable,
	so \ts $\de(P;\ts x,y)=1$ \ts for all $x,y\in X$.
	The idea of the sorting probability~$\de(P)$ is to measure how close to $1/2$
	can one get the probabilities in~\eqref{eq:sort-xy}.
	
	\begin{conj}[{\rm The \ts $\frac13$ -- $\frac23$ Conjecture}] \label{c:1323}
		For every finite poset \ts $P=(X,\prec)$ \ts that is not a chain,
		we have \ts $\de(P)\le \frac13$.
	\end{conj}
	
	This celebrated conjecture was initially motivated by applications to
	sorting under partial information, but quickly became a challenging
	problem of independent interest, and  inspired a great deal of work,
	including our investigation.  To quote~\cite{BFT}, this ``remains one
	of the most intriguing problems in the combinatorial theory of posets.''
	We discuss the history and previous results on the conjecture later in
	the section, after we present our main results (see also~$\S$\ref{ss:finrem-1323}).
	
	\smallskip
	
	\subsection{Main results}\label{ss:intro-main}
	Let \ts $\la=(\la_1,\ldots,\la_d)\vdash n$ \ts be an integer partition
	with at most $d$ parts.  We use $\ell(\la)$ to denote the number of parts
	and $|\la|$ the size of the partition.
	Denote by $P_\la$ the poset associated with~$\la$,
	with elements squares of the Young diagram, and the order defined by
	\ts $(i,j) \preccurlyeq (i',j')$ \ts if and only if \ts $i\le i'$ and $j\le j'$.
	The linear extensions \ts $L \in \cL(P_\la)$ \ts are exactly the
	\emph{standard Young tableaux} of shape~$\la$, see Figure~\ref{f:SYT}.
	
	\begin{figure}[hbt]
		\centering
		\includegraphics[width=12.3cm]{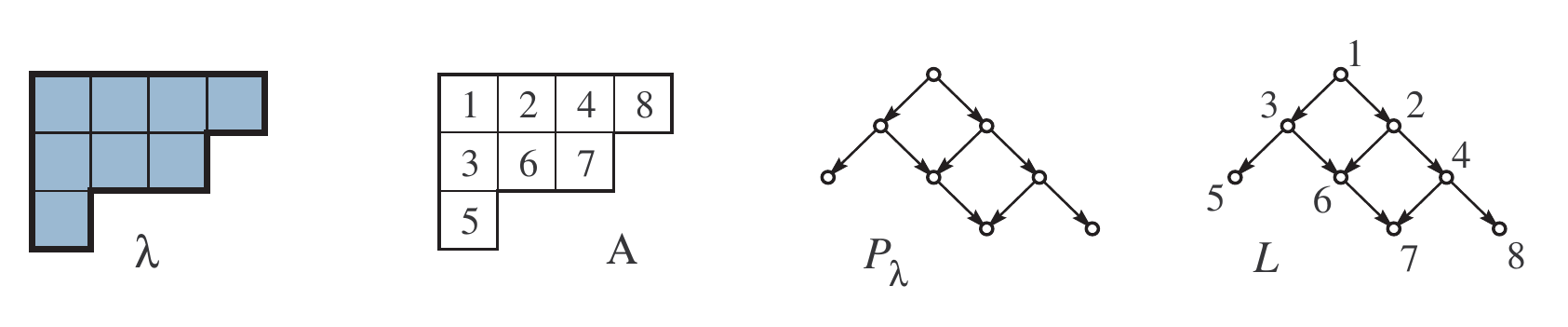}
		\vskip-.2cm
		\caption{{\small Young diagram $\la=(4,3,1)$, standard Young tableau $A\in \SYT(\la)$,
				poset $P_\la$, and the corresponding linear extension $L \in \cL(P_\la)$.  }}
		\label{f:SYT}
	\end{figure}
	
	We state our results, roughly, from less general to more general.
	Let \ts $\al=(\al_1,\ldots,\al_d)\in \rr_+^d$, $\al_1\ge \ldots \ge \al_d \ge 0$,
	and \ts $|\al|=1$, where \ts $|\al|:=\al_1+\ldots + \al_d$.  Such $\al$ are called \emph{Thoma sequences}.
	Define a \emph{Thoma--Vershik--Kerov} (TVK) $\al$-\emph{shape} \. $\la \simeq \al \ts n$, to be
	partition \ts $\la=(\la_1,\ldots,\la_d)$, with \ts $\la_i = \lfloor\al_i n\rfloor$,
	for all \ts $1\le i\le d$.  Note that \ts $|\la|=n-O(1)$ \ts in this case.
	
	\begin{thm} \label{t:TVK}
		Fix~$d\ge 2$. For every Thoma sequence \ts $\al\in \rr^d_{>0}$,
		there is universal constant \ts $C_\al$, s.t.
		$$
		\de\bigl(P_\la\bigr) \. \le \. \frac{C_{\al}}{\sqrt{n}}\,,
		$$
		where \ts $\la\simeq \al \ts n$ \ts is a TVK \ts $\al$-shape.
	\end{thm}
	
	
	\smallskip
	
	We say that a partition $\la\vdash n$ is \ts \emph{$\ep$-thick}, if
	the smallest part \ts $\la_d \ge \ep \ts n$, where $d=\ell(\la)$.
	
	\begin{thm} \label{t:thick}
		Fix~$d\ge 2$. For every $\ep>0$, there is a universal constant  $C_{d,\ep}$,
		such that for every $\ep$-thick partition $\la\vdash n$ with $\ell(\la)=d$
		parts, we have:
		$$
		\de\bigl(P_\la\bigr) \. \le \. \frac{C_{d,\ep}}{\sqrt{n}}\,.
		$$
	\end{thm}
	
	Clearly, every TVK $\al$-shape is $\ep$-thick when
	\ts \ts $0<\ep < \al_d$, and $n$ is
	large enough.  Thus, Theorem~\ref{t:thick} can be viewed as an advanced
	generalization of Theorem~\ref{t:TVK}.
	
	\smallskip
	
	Let \ts $\la=(\la_1,\ldots,\la_d)$, \ts $\mu=(\mu_1,\ldots,\mu_d)$ \ts
	be two partitions with at most $d$ parts, and such that \ts
	$|\la/\mu|:= |\la|-|\mu|=n$.  We write \ts $\mu\ssu \la$,
	if $\mu_i \le \la_i$ for all $1\le i \le d$, and refer to $\la/\mu$
	as \emph{skew partition} (see Figure~\ref{f:SYT-skew}).
	Since poset $P_\mu$ is a subposet of $P_\la$, poset \ts
	$P_{\la/\mu}$ \ts is defined as their difference.
	
	\begin{figure}[hbt]
		\centering
		\includegraphics[width=7.8cm]{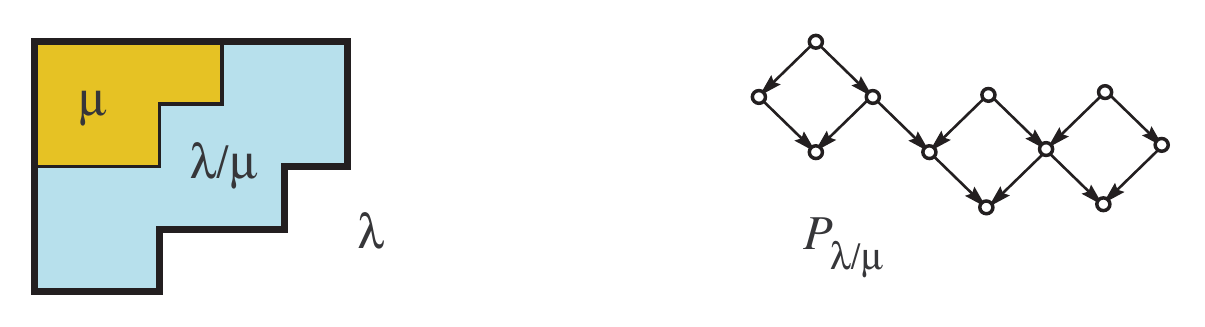}
		\vskip-.2cm
		\caption{{\small
				Skew Young diagram $\la/\mu$ and poset $P_{\la/\mu}$, where $\la=(5,5,4,2)$ and $\mu=(3,2,0,0)$. }}
		\label{f:SYT-skew}
	\end{figure}
	
	Let \ts $\al=(\al_1,\ldots,\al_d)$, $\be = (\be_1,\ldots,\be_d) \in \rr_+^d$,
	$\al_1\ge \al_2\ge \ldots \ge \al_d$, $\be_1\ge \ldots \ge \be_d$, $\be_i \le \al_i$
	for all $1\le i \le d$, and $|\al| - |\be| =1$. Such $(\al,\be)$ are
	called \emph{Thoma pairs}.
	Define a \emph{TVK} $(\al,\be)$-\emph{shape} to be the skew
	partition \ts $\la/\mu$, where \ts $\la \simeq \al \ts n$ \ts
	and \ts $\mu \simeq \be \ts n$. Note that \ts $|\la/\mu|=n+O(1)$ \ts in this case.
	
	\begin{thm} \label{t:TVK-skew}
		Fix~$d\ge 2$. For every Thoma pair \ts $(\al,\be)$, with \ts
        $\al \in \rr^d_{>0}$, \ts $\be\in \rr^{d}_+$,
		there is a universal constant  $C_{\al,\be}$, s.t.
		$$
		\de\bigl(P_{\la/\mu}\bigr) \. \le \. \frac{C_{\al,\be}}{\sqrt{n}}\,,
		$$
		where \ts $\la/\mu$ is a TVK \ts $(\al,\be)$-shape, i.e.\ \ts
		$\la\simeq \al\ts n$, \ts $\mu \simeq \be \ts n$.
	\end{thm}
	
	When \ts $\be=(0,\ldots,0)$, we obtain Theorem~\ref{t:TVK} as a special case.
	We can now state our main result, the analogue of Theorem~\ref{t:thick} for
	skew shapes.
	
	\smallskip
	
	We say that a partition $\la$ is \ts \emph{$\ep$-smooth}, if $\la$ is \ts $\ep$-thick,
	and  \ts $\la_{i}-\la_{i+1} \ge \ep \ts n$, for all $1 \le i < d$.  For brevity,
	we say that a skew partition~$\la/\mu$ is \emph{$\ep$-smooth} if $\la$ is  \emph{$\ep$-smooth}.
	Note that, despite the notation, this condition does not impose any restriction on $\mu$.
	
	\smallskip
	
	\begin{thm}[{\rm Main theorem}] \label{t:main}
		Fix~$d\ge 2$. For every $\ep>0$, there is a universal constant  $C_{d,\ep}$,
		such that for every $\ep$-smooth skew partition $\la/\mu\vdash n$,
		with $\ell(\la)=d$, we have:
		$$
		\de\bigl(P_{\la/\mu}\bigr) \. \le \. \frac{C_{d,\ep}}{\sqrt{n}}\,.
		$$
	\end{thm}
	
	\smallskip
	
	In the TVK case, when $\al_1>\ldots > \al_d > 0$, we obtain
	Theorem~\ref{t:TVK-skew}. However, when the inequalities are non-strict,
	there is no such implication. Similarly, Theorem~\ref{t:main} generalizes
	Theorem~\ref{t:thick} for $\mu=\emp$, and $\la$ is $\ep$-smooth.
	
	\smallskip
	
	The results are proved by using random walks estimates and the
    technique Morales and the last two authors recently developed in
    a series of papers~\cite{MPP1}--\cite{MPP4} on the
	\emph{Naruse hook-length formula} (NHLF).  Roughly,
	in order to estimate the sorting probabilities \ts
	$\de\bigl(P_\la\ts; \. x, \. y\bigr)$, we need very careful bounds
	on the number of standard Young tableaux \ts
	$f(\la/\nu):=|\SYT(\la/\nu)|=e\bigl(P_{\la/\nu}\bigr)$
	for the typical $\nu \ssu \la$ obtained after removing $x$ and/or $y$ from~$\la$.
	The NHLF gives a useful technical tool, which combined with various
	asymptotic estimates implies the result.  We postpone further discussion
	of our results until after a brief literature review.
	
	\smallskip
	
	\subsection{Prior work on sorting probability} \label{ss:intro-sort}
	The \ts $\frac13$ -- $\frac23$ \ts Conjecture~\ref{c:1323} was proposed independently by
	Kislitsyn~\cite{Kis} and Fredman~\cite{Fre} in the context of sorting
	under partial information.  The name is motivated by the following
	attractive equivalent formulation.  In notation of~\eqref{eq:sort-prob-def},
	for every $P=(X,\prec)$ that is not a chain, there exist elements
	$x,y\in X$, such that
	\begin{equation}\label{eq:1323-standard}
	\frac13 \, \le \,
	\Pb\bigl[L(x) < L(y)\bigr]  \, \le \, \frac23\..
	\end{equation}
	
	A major breakthrough was made by Kahn and Saks~\cite{KS}, who
	proved~\eqref{eq:1323-standard} with slightly weaker constants \ts
	$\frac3{11}-\frac8{11}$\ts. In our notation, they showed that
	\ts $\de(P)\le \frac5{11}\approx 0.4545$ \ts for all finite~$P$. A much
	simplified geometric proof (with a slightly weaker bound) was given
	later in~\cite{KL}.  By utilizing technical combinatorial
	tools, the Kahn--Saks bound was slightly improved in~\cite{BFT} to \ts
	$\de(P)\le \frac{1}{\sqrt{5}}\approx 0.4472$, where it currently stands.  
For more on the history
	and various related results, we refer the reader to a dated
	but very useful survey~\cite{Bri99}.
	
While the conjecture does not ask for an efficient algorithm for
finding the desired elements $x,y\in X$, a nearly optimal sorting algorithm
using \ts $O\bigl(\log e(P)\bigr)$ \ts comparisons was found in~\cite{KK}.
See also~\cite{C+} for a simpler version.
	
	Note that the bound \ts $\de(P)\le \frac13$ \ts in the conjecture
	is tight for a 3-element poset that is a union of a $2$-chain and
	a single element.  The effort to establish the conjecture and
	improve the constants remains very active.  First, Linial~\cite{Lin84}
	proved that \ts $\de(P)\le \frac13$ \ts for posets of width~$2$,
	where $\wi(P)$ is the size of the maximal antichain in~$P$.
	In this class, Aigner showed that the tight bound \ts
	$\de(P)=\frac13$ \ts can come only from decomposable posets,
	and Sah~\cite{Sah} recently improved the bound to a slightly lower
	bound \ts $\de(P)< 0.3225$ in the indecomposable case (see also~\cite{Chen}).
	
	Conjecture~\ref{c:1323} was further established for several other classes
	of posets, including semiorders~\cite{Bri89}, $N$-free posets~\cite{Zag1},
	height~2 posets~\cite{TGF}, and posets whose cover graph is a
	forest~\cite{Zag2}. For posets with a nontrivial automorphism
	the conjecture was proved by Pouzet, see~\cite{GHP}, and
	a stronger bound \ts $\de(P)<1-\frac2e \approx 0.2642$ \ts was shown by
	Saks~\cite{Saks}.  Closer to the subject of this paper, Olson
	and Sagan~\cite{OS} recently applied Linial's approach to establish
	Conjecture~\ref{c:1323} for all Young diagrams and skew Young
	diagrams.
	
	There are very few results proving that \ts $\de(P_n)\to 0$ \ts 
as \ts $n\to\infty$ \ts for a sequence \ts $\{P_n\}$ \ts of posets on 
$n$ elements.  Some of them are motivated by the following interesting 
conjecture of Kahn and Sacks~\cite{KS}.
	
	\begin{conj}[{\rm Kahn--Saks}] \label{c:width-KS}
		Let $\eta(d)$ denotes the supremum of $\de(P)$ over all finite posets~$P$
		of width~$d$.  Then \ts $\eta(d)\to 0$ \ts as \ts $d\to \infty$.
	\end{conj}
	
	The most notable result in this direction is due to Koml\'{o}s~\cite{Kom},
	who proved it for height~$2$ posets, as well as posets with \ts $n/f(n)$
	\ts minimal elements, for some undetermined, but possibly very
	slowly growing function \ts $f(n) =\om(1)$. Similarly, Korshunov~\cite{Kor} proved
	that Conjecture~\ref{c:width-KS} holds for \emph{random posets}, which are
	known to have height~$3$ w.h.p.~\cite{KR}.  Note that these are the
	opposite extremes to our setting, as we consider posets $P_{\la/\mu}$ with
	width \ts $d = O(1)$ \ts and height \ts $\Theta(n)$, see 
    also~$\S$\ref{ss:finrem-1323}. 
	
	Before we conclude, let us note that for general posets,
	counting the number $e(P)$ of linear extensions, as well as
	computing the sorting probability $\de(P)$, is $\SP$-complete~\cite{BW}.
	Thus, there is little hope of getting good asymptotic bounds
	on~$\de(P)$, except possibly for one of several notions of
	``random poset''~\cite{Bri93} and ``large poset''~\cite{Jan}.
	In fact, the same complexity results hold for counting linear extensions
	of general $2$-dimensional posets, as well as for posets of
	height~$2$; both results are recently proved in~\cite{DP}.
	This makes (skew) Young diagrams refreshingly accessible
	in comparison.
	
	\smallskip
	
	\subsection{Prior work on asymptotics for standard Young tableaux}
	\label{ss:intro-SYT}
	The combinatorics of standard Young tableaux is a classical subject, but
	until relatively recently, much of the work was on exact counting rather
	than on asymptotics and probabilistic aspects.
	
	The \emph{hook-length formula} (HLF) gives an explicit product formula
	for \ts $e(P_\la)=|\SYT(\la)|$, see e.g.~\cite{Sta99}.  In the
	\emph{stable limit shape}, the Young diagram $\la$ scaled by \ts
	$\frac{1}{\sqrt{n}}$ \ts in both directions $\to \pi$, a curve of area~1.
	Then the HLF gives a tight asymptotic bound for \ts
	$e(P_\la)$ \ts via \emph{hook integral}~\cite{VK}
	(see also~\cite{MPP4}).   \emph{Feit's determinant formula}
	is an exact formula for $f(\la/\mu)$, which can also be
	derived from the \emph{Jacobi--Trudi identity} for skew shapes,
	see e.g.~\cite{Sta99}.  Unfortunately, its determinantal nature
	makes finding exact asymptotics exceedingly difficult, see e.g.\ \cite{BR,MPP4}.
	
	For large skew shapes, Okounkov--Olshanski~\cite{OO} and
	Stanley~\cite{Sta03} computed the asymptotics of $f(\la/\mu)$
	for fixed~$\mu$, as $|\la|\to \infty$. Both papers rely on
	the \emph{factorial Schur functions} introduced by Macdonald
	in~\cite[$\S$6]{Mac}.  The \emph{Naruse hook-length formula} (NHLF)
	was introduced by Hiroshi Naruse in a talk in~2014, and given multiple
	proofs and generalizations in \cite{MPP1,MPP2}.  While the
	formula itself is algebro-geometric in nature, coming from
	the equivariant cohomology of the Grassmannian,
	some of the proofs are direct and combinatorial, using
	factorial Schur functions and explicit bijections~\cite{Kon,MPP1,MPP2} 
(see also~\cite{Pak} for an overview).
	
	In~\cite{MPP4}, Morales--Pak--Panova used the NHLF and the hook integral
	approach to prove an exact asymptotic formula for $f(\la/\mu)$ when
	$\la/\mu$ have a TVK $(\al,\be)$-shape.  In~\cite{MPT}, based on
	a bijection with lozenge tilings given in~\cite{MPP3} and the variational
	principle in~\cite{CKP}, Morales--Pak--Tassy proved an asymptotic
	formula for $f(\la/\mu)$ when both $\la$ and $\mu$ have a stable limit shape.
	In a parallel investigation, Pittel--Romik \cite{PR} found limit curves for
	the shape of random Young tableaux of a rectangle.  Most recently,
	Sun~\cite{Sun} established existence of such limit curves for general
	skew stable limit shapes.
	
	\smallskip
	
	\subsection{Some examples} \label{ss:intro-examples}
	The main difficulty in estimating the sorting probability
	is finding the ``right'' \emph{sorting elements}~$x,y\in X$, such that,
	even when suboptimal, still give a good bound for \ts $\de(P; \ts x,y)$.
	To better understand this issue, let us illustrate the
	sorting probability in some simple examples.
	
	First, take
	$\la=(n,1)$ and $\mu=(1)$. Then poset $P_{\la/\mu}$ consists
	of two chains, of length $1$ and~$(n-1)$.  There is an easy optimal
	pair of elements $x=(1,\lfloor \frac{n+1}{2}\rfloor)$ and $y=(2,1)$.
	Then \ts $\de(P_{\la/\mu})=0$ \ts for even $n$, and \ts
	$\de(P_{\la/\mu})=\frac{1}{n}$ \ts for odd~$n$. \ts 
	Similarly, let  $\la=(n,2)$ and $\mu=(2)$.  The poset $P_{\la/\mu}$ again
	consists of two chains, of length $2$ and~$(n-2)$.  In this case, the
	$x$ as above give suboptimal \ts $\de(P_{\la/\mu}; \. x,y)\sim \frac{1}{2}$.
	Perhaps counterintuitively, the optimal sorting elements are \ts $y=(2,1)$ \ts
	and \ts $x=(1,m)$, where \ts $m= n(1-\frac{1}{\sqrt{2}}) + O(1)$.
	We have \ts $\de(P_{\la/\mu};\. x,y)=\Theta(\frac{1}{n})$ \ts bound
	in this case.  We generalize this example in~$\S$\ref{ss:warmup-upper}.
	
	Now let $\la=(m,m)$, $\mu=\emp$, $n=2m$.  We have
	$f(\la/\mu)= |\SYT(m,m)| = \frac{1}{m+1}\binom{2m}{m}$,
	the \emph{Catalan number}. 
	One can check in this case that
	\ts $\de(P_{\la/\mu}; \. x,y) = \Omega(1)$ for $y=(2,1)$
	and every $x=(1,i)$.  In fact, the bounds that work in this
	case are given by \ts $x=\bigl(1,\frac{m}{2}+k\bigr)$ \ts and \ts
    $y=\bigl(2,\frac{m}{2}-k\bigr)$, for some \ts
	$k=\Theta(\sqrt{m})$.  We prove in~\cite{CPP} that \ts
    $\de(P_{\la/\mu}) = O\bigl(n^{-5/4}\bigr)$ \ts by a direct asymptotic
    argument.
    A weaker \ts $O\bigl(\frac{1}{\sqrt{n}}\bigr)$ \ts  bound
    can be proved via standard bijection from standard
	Young tableaux $A\in \SYT(m,m)$ and Dyck paths \ts $(0,0) \to (m,m)$,
	which in the limit \ts $m\to \infty$ \ts converge to the Brownian excursion
	(see e.g.~\cite{Pit}).  This is the motivational example for this paper.
	
	\smallskip
	
	\subsection{Our work in context} \label{ss:intro-context}
	The differences between various approaches can now be explained
	in the way the authors look for the sorting elements.  In~\cite{Lin84},
	Linial takes $P$ of width two, breaks it into two chains, takes~$x$
	to be the minimal element in one of them and looks for $y$ in
	another chain.  As the previous examples show, this approach can
	never give \ts $\de(P)=o(1)$ \ts for general Young tableaux even
	with two rows. This approach has been influential, and was later
	refined and applied in a more general settings, see e.g.~\cite{Bri89,Zag1}.

	In~\cite{KS} and followup papers~\cite{BFT,KL,Kom,Zag1}, a more
	complicated pigeonhole principle is used, at the end of which
	there is no clear picture of what sorting elements are chosen.
    In fact, the geometric approach in~\cite{KS,KL} can never
     give \ts $\de(P) = o(1)$, as they also point out, cf.~\cite{Saks}.
	The paper most relevant to our paper is \cite{OS}, where the authors look
	for elements $x,y$ on the boundary \ts $\partial \la$, and apply
	the pigeonhole principle, Linial-style.  Already in the Catalan
    example this approach cannot be used to prove that
    \ts $\de(P; \ts x,y) = o(1)$.
	
	Now, following~\cite{PR,Sun}, let \ts $\la\vdash n$ \ts be the stable
	limit shape. It is natural to take $x$ and~$y$ from the same limit curve \ts
	$C_\la(\al):=\partial \{(i,j)\in \la, \. A(i,j)\le \al \ts n\}$,
	where \ts $0<\al <1$, and \ts $A\in \SYT(\la)$ \ts is a
    uniform standard Young tableau of shape~$\la$.
	An example of these limit curves is given in Figure~\ref{f:romik}.
	Since the curves $C_\la(\al)$ have \ts $\Theta(\sqrt{n})$ \ts elements,
	and all $(i,j)\in C_\la(\al)$ can be permuted nearly independently,
	this could in principle give a small sorting probability.
	Making this precise would be both interesting and challenging, but this
	approach fails in our case, since we have \ts $d=O(1)$ \ts rows.
	It does have a few heuristic implications.
	
	\begin{figure}[hbt]
		\centering
		\includegraphics[width=4.4cm]{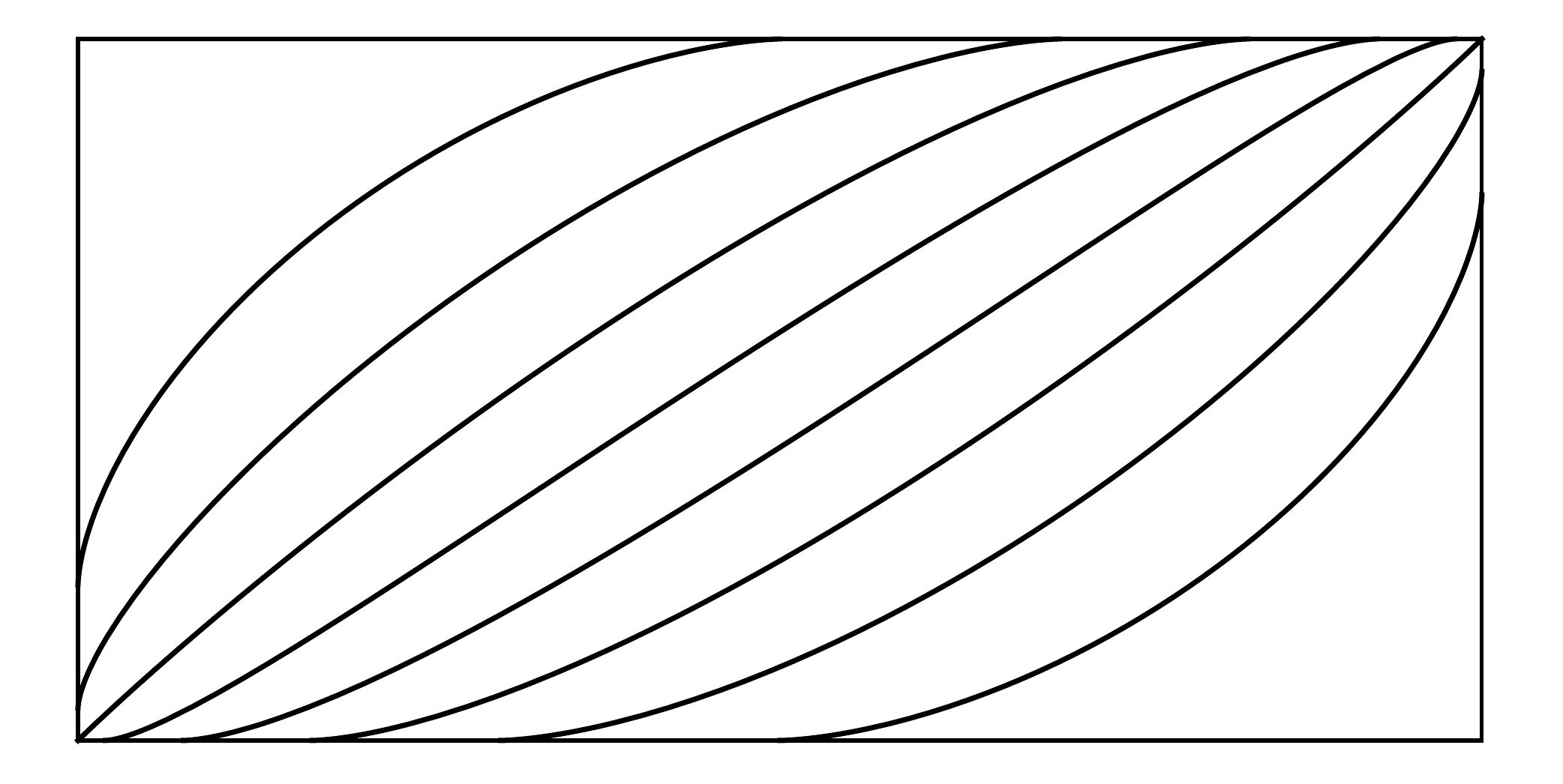}
		\vskip-.4cm
		\caption{\footnotesize The limit curves in a $d\times 2d$ rectangle (created by Dan Romik, April 2020). }
		\label{f:romik}
	\end{figure}
	
	On the one hand, there are likely to be many good sorting pairs of
	elements $x=(i,j)$, $y=(i',j')$, for all $i<j$.
	On the other hand, in general, the limit curves do not have
	a closed-form formula of any kind, and arise as the solution of
	a variational problem~\cite{Sun}. The same holds for the
	asymptotics of \ts $f(\la/\mu)$~\cite{MPT}. As a consequence,
	we are essentially forced to make an indirect argument,
    which proves the result without explicitly
	specifying the exact location of $x,y$ in~$\la$.
	
	
	Our approach is based on a combination of tools and ideas
	from algebraic combinatorics and discrete probability.
	The general philosophy is somewhat similar to the pigeonhole
	principle of Linial~\cite{Lin84}, in the sense that we find a
	sorting pair \ts $x=(1,a)$ \ts and \ts $y=(2,b)$ \ts by searching
	over suitable choices of~$a,b$. As in the Catalan case,
	we start with extreme cases $a=\la_1$, $b=\mu_2+1$, and
	decrease $(a-b)$ until the sorting probabilities
	of~$x$ and~$y$ becomes small. The main difficulty,
	of course, is estimating these sorting probabilities.
	
	In fact, by analogy with the Catalan example, one can interpret
	random standard Young tableaux as random walks from \ts
	$(0,\ldots,0)$ \ts to \ts $(\la_1,\ldots,\la_d)$, which are confined
	to a certain simplex region in $\nn^d$ defined by combinatorial
	constraints.
	The sorting probability  \ts $\de(P_{\la/\mu}; \. x,y)$ \ts can then
	be interpreted as the probability the walk passes below versus
	above of certain codimension-2 subspace. These probabilities
	are then bounded by comparing the simplex-confined lattice walk with
	the usual (unconstrained) lattice walk. This comparison is based
	on delicate estimates which largely rely on the Schur functions
	technology combined with the NHLF.  This technical part occupies
	much of the paper.
	
	\smallskip
	
	\subsection{Structure of the paper} \label{ss:intro-structure}
	We begin by reviewing standard definitions and notation in
	Section~\ref{s:def}, where we also include a number of basic
	results in Algebraic Combinatorics and Discrete Probability.
	In the Warmup Section~\ref{s:warmup} we prove the \ts $\frac13$ -- $\frac23$ \ts
	Conjecture~\ref{c:1323} for all Young diagrams.  This is a known
	result, but the proof we give is new and the tools are
	a precursor of the proof of the Main Theorem~\ref{t:main}.
    We also show how these tools easily give an upper bound
    on the sorting probability $\de(P_\la)$, for $n-\la_1=o(n)$,
    where \ts $n=|\la|$.  In fact, this short section has both
    the style and the flavor of the rest of the paper,
    cf.~$\S$\ref{ss:outline-tune}.
    	
	In Section~\ref{s:outline}, we give key new definitions which allow
	us to state the Main Lemma~\ref{l:main}, and two bounds Lemmas~\ref{l:asy-smooth} and~\ref{l:asy-TVK} on the
	number \ts $f(\la/\mu)$ \ts of standard Young tableaux
	of shape~$\la/\mu$.
	The proofs of these lemmas occupy much of the paper.  The
    technical outline of these proofs is the given
	in~$\S$\ref{ss:outline-roadmap}, so below we only give
	the structure of the paper in the broadest terms.
	
	First, in Sections~\ref{s:paths}--\ref{s:sorting-lattice-path}, we develop
	the technology of lattice path probabilities and their estimates,
	which culminates with the proof of Main Lemma~\ref{l:main}
	in Section~\ref{s:sorting-lattice-path}.  Then, in
	Section~\ref{s:upper-bounds-SYTs}, we develop the technology
	of Young tableaux estimates, which allows us to prove
	Theorem~\ref{t:thick} in Section~\ref{s:straight-shaped Young posets}.
	We then prove Lemma~\ref{l:asy-smooth} and Main Theorem~\ref{t:main}
	in Section~\ref{s:main-thm}. Finally, Lemma~\ref{l:asy-TVK}
	and Theorem~\ref{t:TVK-skew} are proved in Section~\ref{s:TVK}.
	
	We conclude with Section~\ref{s:conj}, where we state several
	conjectures and open problems motivated by our results.
	We present final remarks in Section~\ref{s:finrem}.
	
	\bigskip
	
	
	\section{Definitions, notation and background results}
	\label{s:def}

	\subsection{Standard conventions}  \label{ss:def-standard}
	We fix the number of rows \ts $d \geq 2$ \ts throughout the paper.
	We consider only posets $P=(X,\prec)$ corresponding
	to partitions $\la \vdash n$, or skew partitions \ts $\la/\mu\vdash n$.
	Unless stated otherwise, we have \ts $|X|=n$.
	
	We use \ts $[n]=\{1,\ldots,n\}$, \ts $\nn = \{0,1,2,\ldots\}$, \ts $\zz_+ = \{1,2,\ldots\}$,
	\ts $\rr_+=\{x\ge 0\}$, and \ts $\rr_{>0}=\{x> 0\}$.  We  denote by \ts $\pp_d \ssu\nn^d$ \ts the set of partitions \ts
	$(\la_1,\ldots, \la_d)$, where \ts $\la_1\ge \ldots \ge \la_d\ge 0$, and \ts $\la_i \in \nn$.
	We write \ts $(a_1,\ldots,a_d) \trianglerighteq  (b_1,\ldots,b_d)$, when \ts
	$a_1\ge b_1$, \ts $a_1+a_2\ge b_1+b_2$, \ts \ldots \ts, and \ts
	$a_1+\ldots+a_d=b_1+\ldots+b_d$.
	
	\smallskip
	
	\subsection{Standard Young tableaux}  \label{ss:def-SYT}
	We adopt standard notation in the area.  See e.g.~\cite{Mac95,Sag,Sta99} for
	these results and further references.
	
	Let \ts $\la=(\la_1,\ldots,\la_d)\vdash n$, \ts $\la_1 \ge \ldots \ge \la_d \ge 0$,
	be an \emph{integer partition} of~$n$. Here \ts $n = |\la|:=\la_1+\ldots+\la_d$ \ts
	denotes the size of~$\la$, and $\ell(\la)\le d$ is the number of parts of~$\la$.
	We use $\la'$ to denote a conjugate partition whose parts are the column lengths
	of the diagram~$\la$.
	
	A \emph{skew partition} $\la/\mu$ is
	a pair of partitions $\la=(\la_1,\ldots,\la_d)$, $\mu=(\mu_1,\ldots,\mu_d)$, such
	that $\mu_i\le \la_i$.  In the vector notation above, $\la, \mu\in \pp_d$,
	and $\la-\mu\in \nn^d$.
	The \emph{empty partition} is $\mu=(0,\ldots,0)$, which we
	also denote~$\emp$, e.g.\ $\la/\emp=\la$. The size $|\la/\mu|:=|\la|-|\mu|$;
	we write $\la/\mu\vdash n$ for $|\la/\mu|=n$.
	
	A \emph{Young diagram} (shape), which we also denote by~$\la$, is a set of squares
	$(i,j)\in \nn^2$, such that \ts $1\le i\le d$, and \ts  $1\le j \le \la_i$.
	Similarly, a \emph{skew Young diagram}, which we also denote by~$\la/\mu$, is a set of squares
	$(i,j)\in \nn^2$, such that \ts $1\le i\le d$ \ts and \ts  $\mu_{i}< j \le \la_i$.
	It can in principle have empty rows or be disconnected, although such cases are less
	interesting.   We adopt the \emph{English notation}, where~$i$ increases downwards,
	and~$j$ from left to right, as in Figure~\ref{f:SYT}.
	
	A \emph{standard Young tableau} of shape~$\la/\mu$ is a bijection
	\ts $A: \la/\mu \to [n]$, which increases in rows and columns,
	see Figure~\ref{f:SYT}.  We use $\SYT(\la)$ to denote the set of
	standard Young tableaux of shape~$\la/\mu$.  As in the introduction,
	we use $P_{\la/\mu}=(\la/\mu,\prec)$ to denote the poset on the set of squares
	of~$\la/\mu$, with the partial order defined by \ts $(i,j) \preccurlyeq (i',j')$ \ts
	if and only if \ts $i\le i'$ and $j\le j'$.
	This is a standard definition of a \emph{$2$-dimensional poset}
	associated with a set of points in the plane, see e.g.~\cite{Tro}.
	
	Recall that the linear extensions $\cL(P_{\la/\mu})$ are in natural bijection
	with the set \ts $\SYT(\la)$ \ts of standard Young tableaux.
	Whenever clear, we will use the latter from this point on.  Denote by $\Pblm$
	the uniform probability measure on $\SYT(\la/\mu)$.
	To simplify and unify the notation, from now on we use
	$$
	f(\la/\mu) \, := \, \bigl|\SYT(\la/\mu)\bigr| \. = \. e\bigl(P_{\la/\mu}\bigr) \. = \. \bigl|\cL\bigl(P_{\la/\mu}\bigr)\bigr|.
	$$
	%
	%
	For straight shapes $\la\vdash n$, we have the \emph{Frobenius formula}:
	\begin{equation}\label{eq:SYT-Frob}
	f(\la) \, = \, \frac{n!}{\la_1! \. \cdots\ts \la_d!} \.\prod_{1 \leqslant i < j \leqslant d} \frac{\la_i-\la_j+j-i}{\la_i+j-i}\,,
	\end{equation}
	see e.g.~\cite{FRT} (cf.~\cite[Ex.~1.1]{Mac} and~\cite[Lemma 7.21.1]{Sta99}).
	
	\smallskip
	
	\subsection{Schur polynomials}\label{ss:def-Schur}
	A \emph{semistandard Young tableau} of shape~$\la$ is an map \ts $A: \la \to \zz_+$,
	such that $A$ is weakly increasing in rows and strictly increasing in columns.  We write
	$\SSYT(\la, d)$ for the set of such tableaux with all entries~$\le d$.
	The \emph{Schur polynomial} is a symmetric
	polynomial defined as
	\begin{equation}\label{eq:Schur-def}
	s_{\mu}(x_1,\ldots, x_d) \, :=  \, \det \bigl(x_{j}^{m_i}\bigr)_{i,j=1}^d
	\,\prod_{1\leq i <j \leq d} (x_i-x_j)^{-1}\.,
	\end{equation}
	where \ts $m_i=m_i(\mu):= \mu_i +d-i$.  We call \ts
	$(m_1,\ldots,m_d)= \mu+ (d-1,\ldots,1,0)$ \ts the
	\emph{shifted partition}~$\mu$.
	
	The combinatorics of Schur functions is given by
	\begin{equation}\label{eq:Schur polynomial}
	s_{\lambda}(x_1,\ldots, x_d)\, := \, \sum_{A \in \SSYT(\la, d)} \, \prod_{(i,j)\in \la} \. x_{A(i,j)}
	\, = \, \sum_{A \in \SSYT(\la, d)} \, \prod_{i=1}^d \. (x_i)^{t_i(A)}\ts,
	\end{equation}
	where
	\begin{equation}\label{eq:def-ti}
	t_i(A) \, :=  \, \bigl| \bigl\{(j,k) \in \la/\mu \, \mid \, A(j,k)=i  \bigr\}  \bigr|\., \qquad 1\le i \le d\ts.
	\end{equation}
	The product formula below is classical and follows from~\eqref{eq:Schur-def}
	and~\eqref{eq:Schur polynomial}:
	\begin{equation}\label{eq:HCF}
	s_{\mu}(1,\ldots,1) \, =  \, \bigl|\SSYT(\la, d)\bigr| \, =  \,\prod_{1 \leqslant i < j \leqslant d} \frac{m_i-m_j}{j-i}\..
	\end{equation}
	
	\smallskip
	
	\subsection{Hook-length formulas}\label{ss:def-HLF}
	The \emph{hook-length} of square \ts $(i,j) \in \la$ \ts is defined
	as
	\begin{equation}\label{eq:hook-def}
	h_{\la}(i,j) \. := \. \la_i-j + \la_j' - i+1\ts.
	\end{equation}
	The \emph{hook-length formula} (HLF)~\cite{FRT} (see also~\cite{Sag,Sta99}), is a product
	formula for the number of standard Young tableaux of straight shape:
	\begin{equation}\label{eq:HLF}
	f(\la) \, = \, n!\. \prod_{(i,j)\in \la} \. \frac{1}{h_\la(i,j)}\..
	\end{equation}

	For skew Young diagrams, the number $f(\lambda/\mu)$ can be determined by the
	\emph{Naruse hook-length formula} (NHLF), see~\cite{MPP1, MPP2}.
	Let $D\ssu \la $ be a subset of squares with the same number of squares
	in each diagonal as~$\mu$.  A subset~$D$ is called an \emph{excited diagram}
	if and only if the relation~$\preccurlyeq$ on squares of~$\mu$ holds for
	the corresponding squares in~$D$.  Denote by \ts $\ED(\lambda/\mu)$ \ts
	the set excited diagram of shape~$\la/\mu$. As shown in~\cite{MPP1},
	all $D\in \ED(\lambda/\mu)$ can be obtained from~$\mu$ by a sequence of
	\emph{excited moves}: $(i,j)\to (i+1,j+1)$, for some $(i,j)\in D$,
	s.t. $(i+1,j), \ts (i,j+1) \notin D$.
	
	\begin{thm}[{\rm NHLF~\cite{MPP1}}]\label{t:NHLF-standard}
		For all \ts $\la/\mu\vdash n$, we have:
		\begin{equation}\label{eq:Naruse HLF}
		f(\lambda/\mu) \, = \,  \, n!\. \sum_{D \in \ts \ED(\lambda/\mu)}
		\prod_{(i,j)\in \la\sm D} \. \frac{1}{h_\la(i,j)}\..
		\end{equation}
	\end{thm}
	
	When $\mu=\emp$, we obtain the~HLF~\eqref{eq:HLF}.  The next result
	is a consequence of the~NHLF. Define
	\begin{equation}\label{eq:F-def}
	F(\lambda/\mu) \, := \, n!\. \prod_{(i,j)\in \la/\mu} \. \frac{1}{h_\la(i,j)}\,.
	\end{equation}
	
	\begin{thm}[{\cite[Thm~3.3]{MPP4}}]\label{t:NHLF-asy}
		Let $\la/\mu\vdash n$, $\ell(\la)\le d$.  Then
		$$
		F(\lambda/\mu) \, \le \, f(\la/\mu)  \, \le \, \bigl|\ED(\la/\mu)\bigr| \ts \cdot \ts F(\lambda/\mu)\ts.
		$$
	\end{thm}

	In an effort to
	quantify excited diagrams, we follow an equivalent
	definition given in~\cite[$\S$3.3]{MPP1}.  A \emph{flagged tableau}
	of shape~$\lambda/\mu$ is a tableaux \ts $T\in \SSYT(\mu)$, such that
	\begin{equation}\label{eq:definition flagged tableau}
	j\ts +\ts T(i,j) \ts - \ts i \. \leq \. \lambda_{T(i,j)}\,, \quad  \text{for all} \ \. (i,j) \in \mu\ts.
	\end{equation}
	The corresponding excited diagram is obtained by moving $(i,j)$ for \ts $T(i,j)-i$ \ts
	steps down the southeast diagonal.  The above inequality is a constraint that \ts
	$D\ssu \lambda$.  We denote by \ts $\FT(\la/\mu)$ \ts
	the set of flagged tableaux of shape~$\la/\mu$, so \ts
	$|\FT(\la/\mu)| = |\ED(\la/\mu)|$.
	
	\begin{thm}[{\rm Flagged NHLF~\cite{MPP1}}]\label{t:NHLF}
		For all \ts $\la/\mu\vdash n$, we have:
		\begin{equation}\label{eq:Naruse HLF flagged}
		f(\lambda/\mu) \, = \,  \, n!\. \Biggl[\prod_{(i,j)\in \la} \. \frac{1}{h_\la(i,j)}\Biggr] \.
		\sum_{T \in \FT(\lambda/\mu)} \ \prod_{(i,j) \in \mu} h_\la\bigl(T(i,j),j+T(i,j)-i\bigr).
		\end{equation}
	\end{thm}
	
	\smallskip
	
	\subsection{Bounds on binomial coefficients}\label{ss:def-bin}
	Recall an effective version of the \emph{Stirling formula}:
	\begin{equation}\label{eq:Stirling-formula}
	\sqrt{2\pi}\. n^{n+\frac12}\ts e^{-n} \. \le \. n! \. \le \. n^{n+\frac12}\ts e^{1-n}\ts.
	\end{equation}
	This implies the following standard result:
	
	\begin{prop}\label{p:Stirling's formula}
		Let $a,b$ be integers  such that $a>b >0$. Then
		\begin{align*}
		\frac{\sqrt{2\pi}}{e^2} \. \sqrt{ \frac{a}{b(a-b)} } \. \exp\bigl( a \ts H(b/a)\bigr)
		\,  \leq \, \binom{a}{b} \, \leq \, \frac{e}{2\pi}\. \sqrt{ \frac{a}{b(a-b)} } \. \exp\bigl( a \ts H(b/a)\bigr)\ts,
		\end{align*}
		where \ts $H(r) \ts := \ts -r\ts \log r \ts -\ts (1-r) \ts \log(1-r)$ \ts is the \ts \emph{binary entropy function}.
	\end{prop}
	
	\smallskip
	
	\subsection{Concentration inequalities}\label{ss:def-Hoe}
	Consider a simple random walk \ts $\bX=\bigl(X_t\bigr)_{t \geq 0}$ \ts on \ts $\rr^d$,
	with steps $V=\{v_1,\ldots,v_k\}\ssu \rr^d$ and probability distribution $Q$ on~$[k]$:
	\begin{equation}\label{eq:Hoe-steps}
	X_0\. = \. O, \ \ X_{t+1} \. = \. X_t \. + \. v_i\., \quad \text{where \, $1\le i \le k$ \, \.
		is chosen with probability \. $q_i:=Q(i)$.}
	\end{equation}
	We will use the following concentration inequality that applies in much more general situation.
	
	\begin{thm}[{\rm Hoeffding's inequality~\cite{Hoe63}}]\label{t:Hoeffding's inequality}
		Let $\bX=(X_t)_{t \geq 0}$ be a random walk on $\Rb^d$ with steps $V$ such that $\|v_i\|\le 1$.
		Then, for every $t\geq 1$ and $c>0$ ,
		$$\Pb\bigl[ \ts\|X_t - \Eb[X_t] \ts \| \. \geq \. c\bigr]  \ \leq \  2 \exp \left( -2c^2/t\right).
		$$
	\end{thm}
	
	In~$\S$\ref{ss:most-triplets-admissible}, we will use Hoeffding's inequality for the  set of steps \ts $E=\{e_1,\ldots,e_d\}$ \ts which forms the standard basis
	in~$\rr^d$, and a certain non-uniform distribution~$Q$ on~$[d]$.
	
	\bigskip

\section{Warmup}\label{s:warmup}
	
In this short section we give a new proof and an extension
of the \ts $\frac13$--$\frac23$ \ts Conjecture~\ref{c:1323}
for Young diagrams. We apply these to give an upper bound
for the sorting probability for general Young diagrams.

\smallskip

\subsection{General Young diagrams}\label{ss:warmup-proof}
The first part of the following theorem is the result by Olson and
Sagan~\cite{OS}.  Below, we present a completely different proof 
of the result.  In fact, our sorting pairs of elements are in 
a different location when compared to~\cite{OS}. 

\smallskip
	
\begin{thm}\label{t:SYT-OS}
For every $\la\vdash n$, we have \ts $\de(P_\la)\le \frac13$. Moreover,
$\de(P_\la; \ts x, y) \ts \le \ts \frac13$ \ts for some \ts $x=(1,k)\in \la$
	\ts and \ts $y=(\ell,1)\in \la$.
\end{thm}
	
\smallskip

As suggested by the second part of the Theorem, we need to estimate
sorting probabilities for pairs of elements in the first row and the first column.
	
	\smallskip
	
	\begin{lemma}\label{l:rel_probs}
		Let \ts $\la\vdash n$, and \ts $\Pbl$ \ts denote the probability over uniform
		standard Young tableaux \ts $A\in \SYT(\la)$.  Denote
		$$\aligned
		q_i \, & : = \, \Pbl \bigl[\ts A(i,1)\. < \. A(1,2)\. <\. A(i+1,1)\ts\bigr]\ts, \quad
		1\le i \le \ell-1\., \quad \text{and} \\
		q_\ell \, & : = \, \Pbl \bigl[\ts A(\ell,1)\. < \. A(1,2)\ts\bigr]\ts,
		\endaligned
		$$
		where \ts $\ell=\ell(\la)$ \ts is the length of the first column. Then \ts
		$q_1 \ge \ldots \ge q_\ell$,
		and \ts $q_1+ \ldots + q_\ell=1$.
	\end{lemma}
	
	\smallskip

We present two proofs of the lemma: the traditional Young tableaux proof and the proof
via the Naruse hook-length formula (Theorem~\ref{t:NHLF}).  The former proof is simpler 
while the latter is amenable for generalizations and asymptotic analysis.  We recommend 
the reader study both proofs. 

\smallskip

	\begin{proof}[First proof of Lemma~\ref{l:rel_probs}] 
		Since \ts $A(1,1)=1< A(1,2)$, the number \ts $A(1,2)$ \ts must fall in exactly
		one of the intervals in the lemma.  Thus, we have \ts $q_1+ \ldots + q_\ell=1$.
		
		Let \ts $A\in \SYT(\la)$ \ts be a standard Young tableau, such that \ts
		$A(k,1)<A(1,2)<A(k+1,1)$, for some \ts $1\le k < \ell$.
		Then \ts $A(1,1) =1$, \ldots, $A(k,1)=k$,
		and $A(1,2)=k+1$. The number of such tableaux~$A$ is then equal to $f(\la/\mu^k)$,
		where \ts $\mu^k = (2,1^{k-1})\vdash k+1$.  In the notation of the lemma, we have:
\begin{equation}\label{eq:qk}
		q_k \, = \, \frac{f(\la/\mu^k)}{f(\la)}\..
\end{equation}		

Clearly \. $\mu^k \subset \mu^{k+1}$, and so \. $\la/\mu^{k+1} \subset \la/\mu^k$. Then $f(\la/\mu^{k+1})$ is equal to the number of 
tableaux $A\in \SYT\bigl(\la/\mu^k\bigr)$ \. with \. $A(k,1)=1$.  Therefore, \. $f(\la/\mu^{k+1}) \leq f(\la/\mu^k)$ \. 
and \. $q_{k+1} \leq q_k$.
\end{proof}

\smallskip

	\begin{proof}[Second proof of Lemma~\ref{l:rel_probs}] 		
We follow the first proof until~\eqref{eq:qk}.  At this point, recall the Naruse 
hook-length formula~\eqref{eq:Naruse HLF}:
		$$
		f(\la/\mu) \, = \, (n-|\mu|)! \,  \prod_{(i,j) \in \la} \. \frac{1}{h_\la(i,j)} \,
		\sum_{D\in \ED(\la/\mu)} \. \prod_{(i,j) \in D} \. h_\la(i,j)\ts.
		$$
		Combined with the hook-length formula~\eqref{eq:HLF}, we have:
		\begin{equation}\label{eq:qk-formula}
		q_k \, = \, \frac{f(\la/\mu)}{f(\la)}  \, = \, \frac{(n-k-1)!}{n!}
		\,\sum_{D\in \ED(\la/\mu)} \. \prod_{(i,j) \in D} \. h_\la(i,j)\ts.
		\end{equation}
		
		Now, let \ts $\nu:=(2,1^{k})\vdash k+2$.  We similarly have:
		\begin{equation}\label{eq:qk+1-formula}
		q_{k+1} \, = \, \frac{f(\la/\nu)}{f(\la)} \, = \,  \frac{(n-k-2)!}{n!}
		\,\sum_{D'\in \ED(\la/\nu)} \. \prod_{(i,j) \in D'} \. h_\la(i,j)\ts.
		\end{equation}
		Observe that excited diagrams $D'\in \ED(\la/\nu)$ \ts are characterized
		by the locations of the squares \ts $x_c\in D'$ \ts in the diagonal
		\ts $\{i-j=c\}$, where \ts $-1 \le c \le k$ (see Figure~\ref{f:OS}).
		
		\begin{figure}[hbt]
			\centering
			\includegraphics[width=11.2cm]{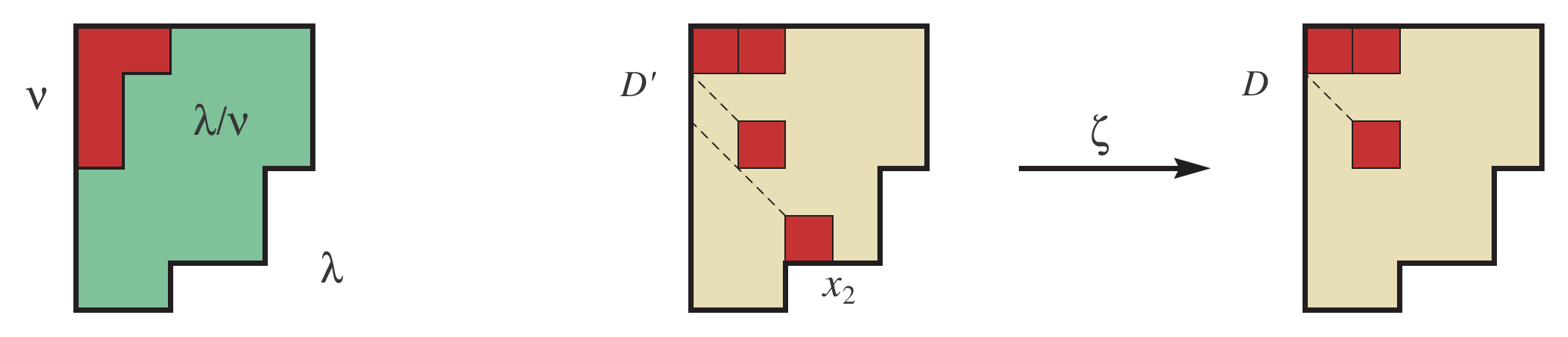}
			\vskip-.1cm
			\caption{{\small
					Skew Young diagram $\la/\nu$, where \ts $\la=(5,5,5,4,4,2)$ \ts and \ts
                    $\nu=(2,1,1)$. \ts
					Map $\vt: D'\to D$, where \ts $D'\in \ED(\la/\nu)$, \ts $D\in \ED(\la/\mu)$,
				    and \ts $D'\sm D = x_2=(5,3)$. }}
			\label{f:OS}
		\end{figure}
		
		Consider a map \ts $\vt: \ED(\la/\nu) \to \ED(\la/\mu)$,
		\ts $\vt(D')=D$, where $D$ is obtained from $D'$ by removing
		the square~$x_k$. From above and by definition of excited diagrams,
		map~$\vt$ is well defined.  This gives:
		\begin{equation}\label{eq:ED-sum-formula}
		\sum_{D'\in \ED(\la/\nu)} \. \prod_{(i,j) \in D'} \. h_\la(i,j) \, =
		\, \sum_{D\in \ED(\la/\mu)}  \prod_{(i,j) \in D} \. h_\la(i,j) \,
        \sum_{\substack{(i,j)\in\la, \, i-j=k \\ D\cup(i,j)\.\in \.\ED(\la/\nu)}}
		\, h_\la(i,j)
		\end{equation}
		The sum on the right is at most
		\begin{equation}\label{eq:hook-sum-formula}
		h_\la(k+1,1) \. + \. h_\la(k+2,2) \. + \. \ldots \, = \,
		\la_{k+1} \. + \. \ldots \. + \. \la_\ell
		\, \le \, n \. - \. |\mu| \, = \, n-k-1\ts.
		\end{equation}
Combining these equations together, we obtain:
$$
		\aligned
		q_{k+1} \ & =_{\eqref{eq:qk-formula}} \qquad \frac{(n-k-2)!}{n!}
		\,\sum_{D'\in \ED(\la/\nu)} \. \prod_{(i,j) \in D'} \. h_\la(i,j)  \\
		& \le_{\eqref{eq:ED-sum-formula}\,\eqref{eq:hook-sum-formula}} \  \frac{(n-k-2)!}{n!}
		\,\sum_{D\in \ED(\la/\mu)} \,\.\ts \prod_{(i,j) \in D} \. h_\la(i,j) \. (n-k-1)
		\ =_{\eqref{eq:qk+1-formula}} \  q_k\,,
		\endaligned
$$
		as desired.
	\end{proof}
	
	\smallskip
	
	\begin{proof}[Proof of Theorem~\ref{t:SYT-OS}]
		Without loss of generality, we can assume that
		$$
		p_1\, := \, \Pbl\bigl[A(1,2) \. < \. A(2,1)\bigr] \, \le \, \frac12\.,
		$$
		since we can conjugate diagram~$\la$, otherwise.  If \ts $p_1\ge \frac13$,
		this implies \ts
		$\de(P_\la; \ts x, y)\le \frac13$ \ts for  \ts $x=(1,2)$ and $y=(2,1)$,
		and proves the theorem.
		
		Suppose now that \ts $p\le \frac13$. By the lemma, we have:
		$$\frac13 \. \ge \. p_1 \. = \. q_1 \. \ge \. q_2 \. \ge \. \ldots \. \ge \. q_\ell\ts.
		$$
		Observe that
		$$
		p_k \, := \, \Pbl\bigl[A(1,2) \. < \. A(k+1,1)\bigr] \, = \, q_1 \. + \. \ldots \. + \. q_k\..
		$$
		Since \ts $q_k \le \frac13$ \ts and \ts $q_1+ \ldots + q_\ell=1$ \ts by the lemma,
		this implies that at least one of these probabilities \ts $p_k \in \bigl[\frac13, \frac23\bigr]$.
		Therefore, the sorting probability \ts $\de(P_\la; \ts x, y)\le \frac13$ \. for \ts $x=(1,2)$ \ts and \ts $y=(k+1,1)$, as desired.
	\end{proof}

\medskip

\subsection{General upper bounds}  \label{ss:warmup-upper}
%
For a partition \ts $\la \vdash n$ \ts define the \emph{imbalance} \ts $\qp(\la)$ \ts as follows:
\begin{equation}\label{eq:imbalance}
\qp(\la) \. := \. \frac{1}{n(n-1)} \. \sum_{i\le j} \. h_\la(i,i) \ts h_\la(j,j+1)\ts.
\end{equation}
Note that
\begin{equation}\label{eq:imbalance-max}
\sum_{i\le j} \. h_\la(i,i) \ts h_\la(j,j+1) \,\le \, \sum_{i} \. h_\la(i,i) \,\sum_{j} h_\la(j,j+1)
\, \le \, n\ts (n-1)\ts,
\end{equation}
so \ts $0\le \qp(\la) \le 1$. \ts  The following result is a generalization of Theorem~\ref{t:SYT-OS}.

\begin{thm}\label{SYT-ext}
For every $\la\vdash n$, we have:
{\rm
$$\de(P_\la) \, \le \,
\min\big\{\.\qp(\la), \. 1-\qp(\la), \. |1-2 \ts\ts \qp(\la)|\.\big\}.
$$
}
\end{thm}

\begin{proof}
In the notation of the proof above, let \ts $k=1$, \ts $\mu=(2)$, and observe that
excited diagrams \ts $D\in \ED(\la/\mu)$ \ts consist of two squares: $(i,i)$
and $(j,j+1) \in \la$, s.t.\ $1\le i \le j$.  Therefore,
$$
p_1\. = \. q_1 \. = \. \frac{f(\la/\mu)}{f(\la)} \. = \. \frac{1}{n(n-1)} \.
\sum_{D\in \ED(\la/\mu)}  \. \prod_{(i,j) \in D} \. h_\la(i,j) \, =_{\eqref{eq:imbalance}} \, \qp(\la)\ts.
$$
There are three possibilities. First, if \ts $q_1\le \frac{1}{3}$, then
the sorting probability \ts $\de(P_\la) \le q_k\le q_1$. Similarly, if
\ts $q_1\ge \frac{2}{3}$, by using \ts $\qp(\la')=1-q_1$,
we have \ts $\de(P_\la) \le 1-\qp(\la)$. Finally, if \ts
$\frac{1}{3} \le q_1  \le \frac{2}{3}$, we have \ts
$\de(P_\la) \le |1-2\ts q_1|$ \ts by definition of $p_1=q_1$.
This implies the result. \end{proof}

\begin{lemma}
Let \ts $\la\vdash n$, and~$m=n-\la_1$.  Then:
$$
\de(P_\la) \, \le \, \frac{m\ts n \. + \. (m-1)(m-2)}{n(n-1)}\..
$$
\end{lemma}

\begin{proof}
We apply Theorem~\ref{SYT-ext} to the conjugate partition~$\la'$.
We have \ts $h_{\la'}(1,1) \leq n$, and
$$\sum_{j\geq 1} \ts h_{\la'}(j,j+1) \, = \,
\sum_{j\geq 1} \ts h_\la(j+1,j) \, = \, m.
$$
Thus, the first term $i=1$ of the summation~\eqref{eq:imbalance}
for the imbalance $\qp(\la')$, is at most \ts $m\ts n$. The remaining
terms with \ts $i\geq 2$ \ts are equal to \ts $\qp(\tau')$,
where \ts $\tau=(\la_2-1,\la_3-1,\ldots)$ \ts of size $\leq m-1$.  We conclude:
$$
\qp(\la') \, \le \, \frac{1}{n(n-1)} \. \bigl(m\ts n \. + \. \qp(\tau')\bigr)
\, \le_{\eqref{eq:imbalance-max}} \, \frac{1}{n(n-1)} \. \bigl(m\ts n \. + \. (m-1)(m-2)\bigr),
$$
as desired. \end{proof}

\begin{cor}\label{c:warmup-upper}
Let \ts $\la\vdash n$, \ts $m=n-\la_1$, and suppose \ts $m=o(n)$.
Then \ts $\de(P_\la) \ts = \ts O\left(\frac{m}{n}\right)$.
\end{cor}

\smallskip
We refer to Section~\ref{s:conj} for further discussion of
general upper bounds.

\bigskip
	
\section{Proof outline}\label{s:outline}
	
	We begin with a number of technical definitions which we present without
	any motivation.  They allow us to state three key lemmas: Main Lemma~\ref{l:main},
	and two asymptotic upper bound Lemmas~\ref{l:asy-smooth} and~\ref{l:asy-TVK}.
	These lemmas follow with a roadmap to the proofs of all theorems in the
	introduction.
	
	\smallskip
	
\subsection{The balance function}  \label{ss:outline-balance}  Define
\begin{align}\label{eq:G-def}
	\eG(\lambda/\mu) \, := \, \prod_{1 \leqslant i < j \leqslant d} \. \min
	\left\{\mu_i-\mu_j+j-i, \frac{\lambda_i+d-i}{\lambda_i-\lambda_j+j-i} \right\}. 
\end{align}
We refer to $\eG(\la/\mu)$ as the \emph{balance function} (not to be confused 
with the \emph{balance constant}).  We will need the following simple estimate:
	
	\smallskip
	
	\begin{prop} For all $\la\vdash N$, we have:
		$$
		1 \. \le \. \eG(\lambda/\mu) \. \le \. (d\ts N)^{\frac{d(d-1)}{2}}.
		$$
	\end{prop}
	
	\begin{proof} The first inequality follows from
		\begin{align}
		& \lambda_i+d-i  \, \geq   \,  \lambda_i-\lambda_j+j-i,  \ \ \text{and}
		\label{eq:lambdai-lambdaj is greater than 1}\\
		& \mu_i-\mu_j+j-i \, \geq   \, 1.\label{eq:mui-muj is greater than 1}
		\end{align}
		for all \ts $1 \leqslant i < j \leqslant d$.
		The second inequality follows from:
		\begin{align}
		&  \lambda_i+d-i \, \leq  \, d\ts N,
		\ \ \text{and}  \label{eq:lambdai-lambdaj is less than n}\\
		& \mu_i-\mu_j+j-i \, \leq  \, d\ts N,\label{eq:mui-muj is less than n}
		\end{align}
		for all \ts $1 \leqslant i < j \leqslant d$.
	\end{proof}
	
	\smallskip
	
	\subsection{Definition of $\ep$-admissible pairs}\label{ss:outline-pairs-def}
	Fix $\ep >0$. 
	%
	We say that \ts $(\lambda,\mu)$ \ts is an \emph{$\ep$-admissible pair of partitions}
	if $\mu \ssu \la$, \ts $\la, \mu \in \pp_d$, and
	\begin{equation}\label{eq:def-admiss-pair}
	\lambda_i -\mu_i \. \geq \. \ep \ts |\la|\ts, \quad \text{for all \ $1\le i \le d$.}
	\end{equation}
	Denote by \ts $\La(n,d,\ep)$ \ts the set of $\ep$-admissible pairs of
	partitions $(\la,\mu)$, such that \ts $\la/\mu\vdash n$, and \ts $\la,\mu\in\pp_d$.
	
	\smallskip
	
	\begin{prop} \label{p:admis-1d}
		Let $(\lambda,\mu)\in \La(n,d,\ep)$  be an $\ep$-admissible pair, \ts $\la, \mu \in \pp_d$. Then \ts $\ep \leq  1/d$.
	\end{prop}
	
	\begin{proof}  We have:
		$$ \ep \, \leq \, \frac{1}{d}  \, \sum_{i=1}^d  \. \frac{\lambda_i-\mu_i}{|\lambda|}  \, =  \,
		\frac{|\lambda|-|\mu|}{d\ts |\lambda|} \, \leq \, \frac{1}{d}\,,
		$$
		as desired.
	\end{proof}
	
	\smallskip
	
	\subsection{Definition of $\ep$-admissible triplets}\label{ss:outline-triplets-def}
	Fix $\ep >0$.  Let \ts $\la, \ga, \mu \in \pp_d$, such that \ts
	$\mu \subseteq \gamma \subseteq \lambda$, and \ts $\la/\mu\vdash n$.
	We say that a triplet $(\la,\ga,\mu)$ is \emph{$\ve$-separated}, if
	\begin{equation}\label{eq:separated-def}
	\gamma_i - \mu_i \geq \frac{\ep^3\ts |\lambda|}{2}\,, \quad  \lambda_i - \gamma_i \geq \frac{\ep^3\ts |\lambda|}{2}\,, \quad
	\text{for all \ \. $1\le i \le d$.}
	\end{equation}
	In other words, condition~\eqref{eq:separated-def} means that the partition $\gamma$
	is bounded away from both $\mu$ and~$\lambda$.
	
	We say that $(\la,\ga,\mu)$ is \emph{progressive}, if
	\begin{equation}\label{eq:progressive-def}
	\bigl\| \gamma \. - \. (1-p)\mu \. - \. p \lambda \bigr\| \, \leq \,  n^{\frac{3}{4}},
	\end{equation}
	where $\| \cdot \|$ denote the $\ell_\infty$-distance in $\Rb^d$, and \ts
	$p:=p(\lambda,\gamma,\mu) \in [0,1]$ \ts is given by
	\begin{equation}\label{eq:definition p}
	p \, := \, \frac1{n}\.\bigl(|\gamma|-|\mu|\bigr).
	\end{equation}
	In other words, condition~\eqref{eq:progressive-def} means that~$\gamma$
	is close to the weighted average of $\mu$ and~$\lambda$.
	
	Finally, we say that \ts $(\lambda,\gamma,\mu)$ \ts is an
	\emph{$\ep$-admissible triplet of partitions}, if \ts
	$\mu \subseteq \gamma \subseteq \lambda$,
	the pair \ts $(\lambda,\mu)$ is $\ep$-admissible, and
	the triplet \ts $(\lambda,\gamma,\mu)$ \ts is both $\ve$-separated
	and progressive.  We use \ts $\Om(n,d,\ep)$ \ts to denote the set
	of $\ep$-admissible triplets.
	
	\smallskip
	
	\subsection{Definition of solid triplets}\label{ss:outline-solid-triplets-def}
	Let \ts $(\lambda,\gamma,\mu) \in \Om(n,d,\ep)$ \ts be an
	$\ep$-admissible triplet defined above.  We say that a triplet
	\ts $(\la,\ga,\mu)$ \ts is \emph{solid} if the following inequalities hold:
	\begin{equation}\label{eq:if conjecture is true}
	\frac{f(\gamma/\mu)}{F(\gamma/\mu)}   \, \leq \, {C} \cdot  \eG(\gamma/\mu)\., \quad
	\frac{f(\lambda/\gamma)}{F(\lambda/\gamma)}   \, \leq \, {C} \cdot \eG(\lambda/\gamma)
	\quad \text{ and } \quad
	\frac{f(\lambda/\mu)}{F(\lambda/\mu)}   \, \geq \, \frac{1}{C} \.\cdot \. \eG(\lambda/\mu),
	\end{equation}
	where~$\ts \eG(\cdot)$ \ts is the balance function defined in~\eqref{eq:G-def}.
	We refer to $C$ as the \emph{solid constant} of the triplet.
	
	\smallskip
	
	\subsection{Sorting probability of solid pairs} \label{ss:outline-solid-def}
	Let $(\lambda,\mu)\in \La(n,d,\ep)$ be an $\ep$-admissible pair.
	We say that a pair \ts $(\la,\mu)$ \ts is \emph{solid}, if there is a
	constant \ts $C_{\lambda,\mu}>0$, such that every $\ep$-admissible triplet
	$(\lambda,\gamma,\mu) \in \Om(n,d,\ep)$ is solid with the
	solid constant~$C_{\lambda,\mu}$.
	
	\smallskip
	
	\begin{lemma}[{\rm Main Lemma}]
		\label{l:main}
		Fix $d\geq 2$ and $\ep >0$.  Let $(\lambda,\mu)\in \La(n,d,\ep)$.
		Suppose further, that \ts $(\la,\mu)$ \ts is a solid pair, with a
		solid constant $C=C_{\la,\mu}>0$.  Then:
		\begin{equation}\label{eq:sorting probability skew-shaped theorem}
		\delta(P_{\lambda/\mu})  \, \leq \, C_{d,\ep} \.  \frac{C^3+1}{\sqrt{n}}\.,
		\end{equation}
		where \ts $C_{d,\ep}>0$ \ts is an absolute constant.
	\end{lemma}
	
	\smallskip
	
	Main Lemma~\ref{l:main} \ts is proved in Section~\ref{s:sorting-lattice-path}.

	\subsection{Asymptotics of  $f(\lambda/\mu)$}\label{ss:outline-asy}
	The key to proving the theorems in the introduction is proving that \ts $f(\lambda/\mu)$ \ts
	is equal, up to a multiplicative constant, to the product of \ts $F(\lambda/\mu)$ \ts
	and the balance function \ts $\eG(\lambda/\mu)$.  Here is the precise statement
	of the reduction.
	
	\begin{lemma}[{\rm Smooth asymptotics}]\label{l:asy-smooth}
		Fix $d\geq 2$ and $\ve >0$.  Let $\lambda/\mu$ be a skew partition,
		such that $\la$ is $\ve$-smooth, and \ts $\la, \mu \in \pp_d$.
		Then there exists an absolute constant \ts $C_{d,\ep} >0$, such that
		$$
		\frac{1}{C_{d,\ep}} \ \eG(\lambda/\mu) \, \leq \,
		\frac{f(\lambda/\mu)}{F(\lambda/\mu)}   \, \leq {C_{d,\ep}} \, \eG(\lambda/\mu)\ts.
		$$
	\end{lemma}
	
	This is the version we need for the proof of the Main Theorem~\ref{t:main}.
	For Theorem~\ref{t:TVK-skew}, we need the following similar result.
	
	\begin{lemma}[{\rm TVK asymptotics}]\label{l:asy-TVK}
		Fix~$d\ge 1$. Let \ts $(\al,\be)$, \ts $\al,\be\in \rr^{d}_+$,
		be a Thoma pair.  Then there is universal constant \ts $C_{\al,\be}>0$,
		such that
		$$
		\frac{1}{C_{\alpha,\beta}} \ \eG(\lambda/\mu) \, \leq \,
		\frac{f(\lambda/\mu)}{F(\lambda/\mu)} \, \leq \, {C_{\alpha,\beta}} \, \eG(\lambda/\mu).
		$$
		where $\la/\mu$ is a TVK $(\al,\be)$-shape, i.e.\ \ts $\la\simeq \al n$, \ts
		$\mu \simeq \be n$.
	\end{lemma}
	
	\smallskip
	
	\subsection{Roadmap for the rest of the paper}\label{ss:outline-roadmap}
	The next three sections are dedicated to the proof of the Main Lemma~\ref{l:main}.
	First, in Section~\ref{s:paths}, we relate sorting probabilities with the estimates
	on the number \ts $f(\la/\mu)$ \ts of standard Young tableaux, which we then
	compare with a  certain lattice random walk in $\rr^d$.  The main result of this
	section is Lemma~\ref{l:most-triplets-admissible}, which proves that the probability
	of having \emph{any} non-$\ep$-admissible triplets is exponentially small.  In the
	following, completely independent Section~\ref{s:technical}, we obtain various Young
	tableaux estimates.  Here the main result is Lemma~\ref{l:pmf of Zt upper bound}
	which gives an upper bound on the number of standard Young tableaux which contain
	a given $\ep$-admissible triplet.  This is the only result which will be used
	later on.  Finally, in a short Section~\ref{s:sorting-lattice-path}, we combine
	Lemma~\ref{l:most-triplets-admissible} and Lemma~\ref{l:pmf of Zt upper bound}
	to prove the Main Lemma~\ref{l:main}.
	
	We restart anew our analysis of the number \ts $f(\la/\mu)$ \ts in
	Section~\ref{s:upper-bounds-SYTs}, this time with a different purpose
	of comparing it to the product \ts $\eG(\la/\mu) \ts F(\la/\mu)$.
	The main results of this section are Lemma~\ref{l:NHLF  upper bounded by schur polynomial}
	and Corollary~\ref{c:interval} which give upper and lower bounds.
	In Section~\ref{s:straight-shaped Young posets}, we prove
	conceptually simpler estimates required for Theorem~\ref{t:thick}.
	This section is both a culmination of earlier results, and a training bound for
	the next two sections.
	
	In Section~\ref{s:straight-shaped Young posets}, we use results from
	Section~\ref{s:upper-bounds-SYTs} to prove Lemma~\ref{l:asy-smooth}.
	We then combine it with the Main Lemma~\ref{l:main} to prove
	Theorem~\ref{t:main} in a short Section~\ref{s:main-thm}.
    Similarly, in a much longer and more technical Section~\ref{s:TVK},
    we first prove Lemma~\ref{l:asy-TVK}, which is then combined with the
    Main Lemma~\ref{l:main} to prove Theorem~\ref{t:TVK-skew}.
	
	\smallskip

	\subsection{A tale of two styles}\label{ss:outline-tune}
	The underlying logic of the paper is rather convoluted and somewhat
	buried in the avalanche of technical estimates, so let us clarify it a bit.
	There are really two things going on at the same time.  On a higher
	level, we develop various probabilistic tools to obtain
	the desired estimates.  While largely elementary from a technical point
	of view, these tools seem to be necessary. They are also unavoidably
	tedious largely because we are starting from scratch in the absence
	of such approach in the existing literature on the subject.
	
	On a lower level, our probabilistic calculations employ a variety
	of highly technical estimates on a host of Young tableau parameters.
	Some of the tools involved, such as NHLF~\eqref{eq:Naruse HLF},
	are relatively recent and come from a long series of works in
	Algebraic Combinatorics, including some by the last two authors.
	While we make our presentation largely self-contained and clarify the
    NHLF in the Warmup Section~\ref{s:warmup}, this technology remains
	difficult and yet to be fully understood on a conceptual level.
	
To make a musical comparison, we have a guitar duo with a new
accessible melody played on a lead guitar, paired with a famously
difficult theme on a rhythm guitar.  The result may appear
cacophonous at first, but we hope the reader can persevere,
become oblivious to the noise, and learn to appreciate the tune.

	\bigskip
	
	\section{Standard Young tableaux as lattice paths}\label{s:paths}
	
	We interpret the standard Young tableaux $A\in \SYT(\la/\mu)$ as lattice paths
	within a simplex in $\nn^d$.  We compare them to unconstrained lattice
	paths to estimate the sorting probabilities.
	
	\smallskip
	
	\subsection{Setup}\label{ss:paths-setup}
	Let $\lambda/\mu\vdash n$, and let \ts $L\in \SYT(\lambda/\mu)$ \ts
	be a uniform random standard Young tableau. Denote by \ts $\bZ = (Z_0,Z_1,\ldots,Z_n)$
	\ts the sequence of  \ts $Z_0=\mu$, \ts $Z_n=\la$, and \ts
	$Z_t = \{(i,j) \. | \. L(i,j) \le t\}$ \ts is a partition.
	Denote by $\Path(\lambda/\mu)$ the set of all such lattice paths $\bZ: \mu\to\la$.
	Note that $\Path(\lambda/\mu)$ is in bijection with $\SYT(\lambda/\mu)$.
	
	We write \ts $\bZ$ \ts as a sequence of vectors \. 
$\bigl(Z_t(1),\ldots,Z_t(d)\bigr)_{0 \leq t \leq n} \in \pp_d$.
	From this point on, we think of \ts $Z_t\in \pp_d$ \ts as a random vector,
	and the sequence \ts $(Z_0,Z_1,\ldots,Z_n)$ \ts as a random lattice path \ts
	$\bZ: \mu \to \la$ \ts  in~$\pp_d$.  We refer to $\bZ$ as
	\emph{tableau random walk}. \ts
	Recall that \ts $\Pblm$ \ts denotes the probability over uniform standard
	Young tableaux \ts $A\in \SYT(\la/\mu)$. By a mild abuse of notation,
	we refer to tableau random walks \ts $\bZ$ \ts as being
	sampled from $\Pblm$.
	
	
	Below we give an upper bound for the sorting probability \ts
	$\de(P_{\lambda/\mu})$ \ts in terms of the probability of the lattice path
	\ts $(Z_t)_{t \geq 0}$ visiting a particular codimension 2 hyperplane
	in~$\rr^d$.
	
	\smallskip
	
	\subsection{Sorting probability via tableau random walks}\label{ss:paths-anti}
	Let $(a,b)$ be two integers, such that \ts
	$\mu_1 < a\le \la_1$ \ts and \ts $\mu_2 < b \le \lambda_2$.
	Consider the event
	$$
	\Ac(a,b)\, := \, \bigl\{ \ts (Z_t)_{0\le t \le n} \, \mid \,
	Z_t(1)=a, \. Z_t(2)=b,   \, \text{ for some } \, t\geq 0\ts \bigr\}.
	$$
	In other words, $\Ac(a,b)$ is the event that the tableau random walk \ts $\bZ=(Z_0,\ldots,Z_n)$ \ts
	intersects the hyperplane in $\Rb^d$ given by \ts $\bigl\{ (x_1,\ldots, x_d) \. \mid \, x_1=a, \. x_2 =b  \bigr\}$.
	
	\begin{lemma}\label{l:quantitative bound for Linial}
		Let \ts $\lambda, \mu\in \pp_d$, \ts $\lambda/\mu\vdash n$, and  let \ts
		$a \in \nn$, s.t. $\mu_1\le a \le \lambda_1$.
		Define
		\begin{equation}\label{eq:epsilon(a)}
		\vp(a) \. := \. \max_{\mu_2 < k \leq \lambda_2 }    \Pblm \big[\Ac(a,k) \big].
		\end{equation}
		Then there exists $b\in \nn$, such that \ts $\mu_2 \le b \le a$,
		$$
		\Bigl|  \Pblm \bigl[L(1,a) < L(2,b)  \bigr] \, - \, \frac{1}{2}\Bigr|  \,\. \leq \, \vp(a).
		$$
		In particular, we have
		\[ \delta(P_{\lambda/\mu}) \, \leq \,  2 \ts\vp(a).  \]
	\end{lemma}
	\begin{proof}
		Observe that \ts $L(1,a) < L(2,b)$ \ts in the language of paths means \ts
		$Z_t(1)=a$ \ts and \ts $Z_t(2)<b$, for some \ts $0\le t\le n$.
		By taking the probabilities of both events, we then have
		$$
		\Pblm \bigl[L(1,a) < L(2,b) \bigr]
		\, = \, \Pblm \bigl[ Z_t(1)=a, \ Z_t(2)<b  \ \text{ for some }t\geq 0 \bigr]
		\, = \, \sum_{k=\mu_2}^{b-1}\Pblm \bigl[ \Ac(a,k) \bigr].
		$$
		Denote by $W(a,b)$ the sum on the right.  It then suffices to show that \ts
		$W(a,b) \in \left[\frac{1}{2}-\vp, \frac{1}{2} + \vp \right]$
		\ts for some $b \in [\mu_2,\lambda_2]$ \ts and \ts $\vp>0$.
		
		Note that, when $b=\mu_2$, the sum has zero summands, so \ts $W(a,b)=0$.
		On the other hand, when $b=a$, we have \ts $W(a,b)=1$.   As the sum is
		nondecreasing, there exists an integer \ts $b' \in [\mu_2,a)$,
		such that \ts $W(a,b') < \frac{1}{2}$, while \ts $W(a,b'+1) \ge \frac{1}{2}$.
		This completes our proof.
	\end{proof}
	
	\smallskip
	
	\subsection{Conditioned lattice random walks are tableau random walks}\label{ss:paths-LRW}
	Fix $\la/\mu\vdash n$, where $\la,\mu\in \pp_d$ as above.
	Recall the notation in~$\S$\ref{ss:def-Hoe}.   Denote by \ts
	$E=\{e_1,\ldots,e_d\}$ \ts the standard basis in~$\rr^d$.
	
	\smallskip
	
	Define the \emph{lattice random walk} \ts $\bX=(X_0,\ldots,X_n)$ \ts
	on $\nn^d$, as follows:
	\begin{equation}\label{eq:LRW-steps}
	X_0\. = \. \mu, \ \ X_{t+1} \. = \. X_t \. + \. e_i\., \quad \text{where \ $i \in [d]$ \
		is chosen with probability} \ \ q_i := \. \frac{1}{n}\ts\bigl(\lambda_i-\mu_i\bigr).
	\end{equation}
	Denote by
	\begin{equation}\label{eq:event-Cc}
	\Cc \. :=  \bigl\{ X_n = \la, \, \bX \in \pp_d \bigr\}
	\end{equation}
	the event that $\bX\in \Path(\la/\mu)$.
	
	\begin{prop} \label{p:cond-prob}
		$$
		\Pb\bigl[\ts \bX \mid \ts \Cc \. \bigr] \, = \,
		\frac{1}{f(\la/\mu)}\..
		$$
	\end{prop}
	
	The proposition is saying that the lattice random walk~$\bX$ conditioned
	to~$\Cc$ coincides with the tableau random walk~$\bZ$ defined above.
	
	\begin{proof}
		Suppose \ts
		$(X_0,\ldots,X_{n}) \in \Path(\lambda/\mu)$.
		Then \ts $\bX$ takes \ts $(\lambda_i-\mu_i)$ \ts steps~$e_i$.
		Therefore,
		$$
		\Pb\bigl[\ts \bX \mid \ts \Cc \. \bigr] \, = \,
		\Pb \bigl[ \ts \bX \, \bigl| \bigr. \, X_n = \la, \, \bX \in \pp_d\bigr]
		\, \propto \, \prod_{i=1}^d \. (q_i)^{\lambda_i-\mu_i} \, = \, 
\prod_{i=1}^d \left(\frac{\lambda_i-\mu_i}{n}\right)^{\lambda_i-\mu_i}.
		$$
		In other words, conditioned on~$\Cc$, the random walk~$\bX$ is uniform
		in \ts $\Path(\lambda/\mu)$.  Since \ts $f(\la/\mu)=\big|\Path(\lambda/\mu)\big|$ \ts 
by definition, we obtain the result.\end{proof}
	
\smallskip

	The reason for the non-uniform choice of distribution~$Q$ given above
	will become clear in the next subsection.  For now, let us mention that
	this distribution is chosen so that \ts $\Eb[X_{n}]= \lambda$.
	This is to ensure that the probability \ts $\Pb \bigl[\Cc]$ \ts
	decays polynomially rather than exponentially, i.e., so that the paths
	in \ts $\Path(\la/\mu)$ \ts are living in the typical regime and
	not the large deviation regime.
	
	\smallskip
	
	\subsection{Polynomial decay}\label{ss:paths-poly-decay}
	Let \ts $\bX = (X_t)_{1\le t\le n}$ \ts be the random walk on $\Zb^d$ defined above.
	It follows from Proposition~\ref{p:cond-prob} that \ts $\Pb[\bX \mid \Cc]$ \ts is uniform
	in~$\Path(\la/\mu)$.  The following lemma gives a lower bound on \ts $\Pb[ \Cc]$.
	
	\begin{lemma}\label{l:lower bound for Cc}
		Fix $d\geq 2$. There exists an absolute constant $C_d>0$ such that the
		following holds.  Let $\lambda/\mu\vdash n$, and \ts $\mu_i < \lambda_i$,
		for all $1\le i \le d$.  Then
		$$
		\Pb[ \Cc]  \,\geq  \, C_d \. n^{- \frac{d^2-1}{2}}.
		$$
	\end{lemma}

	\begin{proof}
		It follows from the proof of Proposition~\ref{p:cond-prob} that
		$$
		\Pb \big[ \Cc \big] \, =   \, \frac{\Pb\bigl[\ts \bX \ts \bigr]}{\Pb\bigl[\ts \bX \mid \ts \Cc \. \bigr]}\, =   \,  f(\lambda/\mu) \, \prod_{i=1}^d \left(\frac{\lambda_i-\mu_i}{n}\right)^{\lambda_i-\mu_i}.
		$$
		Recall the definition of $F(\lambda/\mu)$ in~\eqref{eq:F-def}.
		Theorem~\ref{t:NHLF-asy} and definition~\eqref{eq:hook-def} give:
		\begin{align*}
		f(\lambda/\mu) \ \geq \ F(\lambda/\mu)  \, =  \,
		n! \. \prod_{(i,j) \in \lambda \setminus \mu} \frac{1}{h_\lambda(i,j)}
		\, \geq   \, n! \, \prod_{i=1}^d \. \frac{1}{(\lambda_i-\mu_i+d-i)!}\..
		\end{align*}
		Combining the two equations above, we then get that $\Pb \big[ \Cc \big]$ is bounded
		from below by
		\begin{align*}
		\Pb \big[ \Cc \big] \, & \ge   \,  \frac{n!}{n^n} \, \prod_{i=1}^d \frac{(\lambda_i-\mu_i)^{\lambda_i-\mu_i}}{(\lambda_i-\mu_i+d-i)!} \\
		& \geq   \sqrt{2 \ts\pi\ts n} \. e^{-n} \. \prod_{i=1}^d \frac{ (\la_i-\mu_i)^{\la_i-\mu_i}}{ (\la_i-\mu_i+d-i)^{\la_i-\mu_i+d-i+1/2} \. e^{-\la_i+\mu_i-d+i+1}} \\
		& \geq   \sqrt{2 \ts\pi\ts n} \. e^{-\frac{d(d-1)}{2}}  \.  \prod_{i=1}^d \. e^{-d+i+1} \.\frac{1}{ (\la_i -\mu_i+d-i)^{d-i+1/2} } \\
		& \geq  \sqrt{2\pi} \. e^{-d(d-2)} n^{-d^2/2+1/2}\ts.
		\end{align*}
		Here we used Stirling's formula~\eqref{eq:Stirling-formula} to bound the factorials and
		$$
		\left(1 + \frac{d-i}{\la_i-\mu_i}\right)^{\la_i-\mu_i} \, \leq \, e^{d-i}\ts.
		$$
		The assumption \ts $\mu_i < \lambda_i$ \ts for all $1\le i \le d$, is used to conclude that \ts
		$(\la_i-\mu_i +d-i) \leq n$.  Taking \ts $C_d = \sqrt{2\pi}e^{-d(d-2)}$ \ts implies the result.
	\end{proof}
	
	
	\smallskip

	\subsection{Most triplets are $\ep$-admissible}\label{ss:most-triplets-admissible}
	We can now prove the main result of this section, that the probability
	of \ts $(\lambda,Z_t,\mu)$ \ts not being $\ep$-admissible is exponentially small.
	
	\smallskip
	
	\begin{lemma}\label{l:most-triplets-admissible}
		Fix $d \ge 2$ and $\ep >0$.  Let \ts $(\lambda,\mu)\in \La(n,d,\ep)$ \ts
		be an $\ep$-admissible pair.  Suppose \ts $t\in \nn$ \ts satisfies
		\begin{equation}\label{eq:assumption on t}
		\ep^2 \. \leq \, \frac{t}{n} \. \leq \. 1-\ep^2.
		\end{equation}
		Then there exists a  constant \ts $C_{d,\ep}>0$, such that
		\begin{equation*}
		\Pblm\bigl[(\lambda,Z_t,\mu) \notin \Omega(n,d,\ep)  \bigr] \, \leq \,
		C_{d,\ep} \. n^{\frac{d^2-1}{2}} \. e^{- 2 \sqrt{n}} \ts.
		\end{equation*}
	\end{lemma}
	
	\smallskip
	
	\begin{proof}
		Let
		$$
		\xi_t \, := \, \frac{t}{n} \ts (\la-\mu) \. +\. \mu\ts.
		$$
		Note that $\xi_t\in \rr_+^d$ is not necessarily in~$\pp_d$.  Suppose that $Z_t$ satisfies
		\begin{equation}\label{eq:concentration}
		\left\| Z_t- \xi_t \right\| \, \leq \, n^{3/4}.
		\end{equation}
		This assumption implies:
		\begin{align*}
		\left|Z_t(i)- \mu_i \right| \  & \geq  \ |\xi_t(i)- \mu_i | \. - \. n^{3/4} \, = \, \frac{t}{n} (\lambda_i-\mu_i)  \, - \,  n^{3/4} \\
		& \geq_{\eqref{eq:assumption on t}} \  \ep^2  (\lambda_i-\mu_i)  \, - \,  n^{3/4} \, \geq_{\eqref{eq:def-admiss-pair}} \ \ep^3 |\lambda|  \, - \,  n^{3/4} \, \geq \, \frac{\ep^3}{2} \. |\lambda|\.,
		\end{align*}
		for all \ts $1\le i \le d$, and $n$ large enough. \ts
		By the same reasoning, the assumption~\eqref{eq:concentration} implies:
		$$
		\left|\lambda_i -Z_t(i) \right| \, \geq \, \frac{\ep^3}{2} \. |\lambda|\.,
		$$
		for all \ts $1\le i \le d$, and $n$ large enough.  By the definitions~\eqref{eq:separated-def}
		and~\eqref{eq:progressive-def} of $\ep$-admissible triplets, the
		assumption~\eqref{eq:concentration} implies that \ts $(\lambda, Z_t,\mu) \in \Om(n,d,\ve)$
		\ts for $n$ large enough. We conclude:
		\begin{equation*}
		\Pblm\bigl[(\lambda,Z_t,\mu) \notin \Omega(n,d,\ep)  \bigr] \, \leq \,
		\Pb \big[\left\| Z_t- \xi_t \right\| \. \geq \, n^{3/4}  \big]\ts,
		\end{equation*}
		for $n$ large enough.
		
		By Proposition~\ref{p:cond-prob}, the lattice random walk~$\bX$ conditioned to~$\Cc$
		coincides with~$\bZ$. Observe that \ts $\xi_t = \Eb[X_t]$.  Since \ts $t \leq n$, we have:
		\begin{align*}
		\Pb \big[\left\| Z_t- \xi_t \right\| \. \geq \, n^{3/4}  \big] \,
		& \le \,  \Pb \big[\left\| X_t- \xi_t \right\| \. \geq \, n^{3/4}  \ \mid \ \Cc \big]
		\, \le \, \frac{1}{\Pb[\Cc]} \.\cdot \. \Pb \big[\left\| X_t- \xi_t \right\| \. \geq \, n^{3/4}   \big] \\
		& \leq_{(\text{Thm}~\ref{t:Hoeffding's inequality})}   \ \frac{1}{\Pb[\Cc]} \.\cdot \. 2\. e^{- 2 \ts n^{3/2}/t}
		\ \leq   \ \frac{1}{\Pb[\Cc]} \. \cdot \. 2\. e^{- 2 \sqrt{n}}  \\
		& \leq_{(\text{Lem}~\ref{l:lower bound for Cc})}  \ 2\ts C_{d} \,  n^{\frac{d^2-1}{2}} \. \. e^{- 2 \sqrt{n}} \.,
		\end{align*}
		for $n$ large enough, and where \ts $C_d>0$ \ts is the constant from Lemma~\ref{l:lower bound for Cc}.
		This implies the result.
	\end{proof}
	
	\bigskip
	
	\section{Asymptotics and bounds for lattice paths}\label{s:technical}
	
	This section contains bounds and estimates used to bound
	the sorting probability in the proof of
	Main Lemma~\ref{l:main}.
	
	\smallskip
	
	\subsection{Asymptotics for hook-lengths}
	In this subsection, we prove an asymptotic estimate for $F(\lambda/\mu)$
	defined in~\eqref{eq:F-def}, for all $\ep$-admissible pairs $(\lambda,\mu)$.
	First, we need the following technical lemma.
	
	\begin{lemma}\label{l:mu hook length estimate}
		Fix $d\geq 2$ and $\ep >0$.  Let \ts $(\lambda,\mu)\in \La(n,d,\ep)$ \ts  be an $\ep$-admissible pair.
		Then:
		$$
		\frac{\lambda_i!}{(\lambda_i-\mu_i)!} \, \leq \,  \prod_{j=1}^{\mu_i} \. h_\lambda(i,j)
		\, \leq \,  \ep^{-(d-i)} \.  \frac{\lambda_i!}{(\lambda_i-\mu_i)!}\.,
		$$
		for all \ts $1\le i \le d$.
	\end{lemma}
	
	\begin{proof}
		The lower bound is clear since $h_{\la}(i,j) \geq \la_i - j+1$.  For the upper bound,
		let $J$ be the largest nonnegative integer such that \ts $\lambda_{J+i} \geq \mu_i$.
		It follows from the definition of hook-lengths that
		\begin{equation}\label{eq:proof-lemma-hook-estimates}
		\prod_{j=1}^{\mu_i} \. h_{\lambda} (i,j) \, =  \,
		\frac{\lambda_i!}{(\lambda_i-\mu_i)!} \, \cdot \,  \prod_{k=1}^{d-i}  \. (\lambda_i+k)  \,
		\prod_{k=1}^{J} \. \frac{1}{\lambda_i-\mu_i+k} \, \prod_{k=J+1}^{d-i} \. \frac{1}{\lambda_i-\lambda_{i+k}+k} \..
		\end{equation}
		First, note that for all $k >J$, we have
		\begin{align*}
		\lambda_i-\lambda_{i+k} \ \geq & \ \lambda_i -\mu_i  \ \geq \ \ep|\lambda|,
		\end{align*}
		where the first inequality follows from the maximality of~$J$,
		and the second inequality follows from  \eqref{eq:def-admiss-pair}.
		This implies
		\begin{align*}
		& \prod_{k=1}^{d-i}  (\lambda_i+k) \, \prod_{k=1}^{J} \frac{1}{\lambda_i-\mu_i+k} \, \prod_{k=J+1}^{d-i} \frac{1}{\lambda_i-\lambda_{i+k}+ k} \\
		&  \qquad \leq  \,   \prod_{k=1}^{d-i}  (|\lambda|+k) \, \prod_{k=1}^{J} \frac{1}{\ep|\lambda|+k} \, \prod_{k=J+1}^{d-i} \frac{1}{\ep|\lambda|+k}  \, \leq  \, \prod_{k=1}^{d-i} \frac{|\lambda|+k}{\ep|\lambda|+k}
		\, \leq   \, \ep^{-(d-i)},
		\end{align*}
		where the last inequality follows since \ts
		$\ep^{-1}\geq d \geq 1$, by Proposition~\ref{p:admis-1d}.
		Together with~\eqref{eq:proof-lemma-hook-estimates}, this completes the proof.
	\end{proof}
	
	\begin{lemma}\label{l:F -- asymptotic estimate}
		Fix $d\geq 2$ and $\ep >0$.  Let $(\lambda,\mu)\in \La(n,d,\ep)$.  Then:
		$$
		G(\la/\mu) \, \leq \, F(\lambda/\mu) \, \leq \, \ep^{-\frac{d(d-1)}{2}} \. G(\la/\mu)\.,
		$$
		where
		\begin{equation}\label{eq:G-function-def}
		G(\lambda/\mu) \, := \,  \frac{n!}{(\lambda_1-\mu_1)! \. \cdots \ts (\lambda_d-\mu_d)!} \ \prod_{1 \leqslant i < j \leqslant d}  \frac{\lambda_i-\lambda_j+j-i}{\lambda_i+j-i}\..
		\end{equation}
	\end{lemma}
	
	\begin{proof}
		By definition~\eqref{eq:F-def}, we have:
		$$
		F(\lambda/\mu) \, =  \, n! \, \prod_{(i,j)\in \lambda} \. \frac{1}{h_\lambda(i,j)}
		\.  \prod_{(i,j) \in \mu} h_{\lambda}(i,j) \,
		= \, \frac{n!}{\lambda_1! \. \cdots \lambda_d!} \,
		\prod_{1 \leqslant i < j \leqslant d}  \frac{\lambda_i-\lambda_j+j-i}{\lambda_i+j-i}
		\, \prod_{(i,j) \in \mu} h_{\lambda}(i,j)\ts.
		$$
		The result now follows by substituting the upper and lower bounds
		in Lemma~\ref{l:mu hook length estimate} to the
		products of hooks on the RHS, over all \ts $1\le i \le d$.
	\end{proof}
	
	\smallskip

	\subsection{Asymptotics for binomial coefficients}
	Consider a triplet of partitions \ts $(\lambda,\gamma,\mu)$, such that \ts $\la/\mu\vdash n$.
	Denote by \ts $\by= \by(\lambda,\gamma,\mu)$ \ts the vector $(y_1,\ldots, y_d)\in \rr^d$,
	defined as
	\begin{equation} \label{eq:definition yi}
	y_i \, := \, \frac{ \gamma_i \. - \. (1-p)\mu_i \. - \. p\lambda_i }{\sqrt{n}}\,,
	\quad \text{where \. $p \in [0,1]$ \. is given by~\eqref{eq:definition p}.}
	\end{equation}

	\begin{lemma}\label{l:binomial formula asymptotic}
		Fix $d\geq 2$ and $\ep >0$.
		Let \ts $(\lambda,\gamma,\mu)\in \Om(n,d,\ep)$ \ts be an $\ep$-admissible triplet.
		Then there exists an absolute constant $B_{d,\ep}>0$, such that
		$$
		\frac{\binom{\lambda_1-\mu_1}{\gamma_1- \mu_1}  \ldots \binom{\lambda_d-\mu_d}{\gamma_d-\mu_d}}{ \binom{n}{|\gamma|-|\mu|} }
		\ \leq  \  B_{d,\ep} \, n^{-\frac{(d-1)}{2}} \,
		\exp \left[ - 2 \sum_{i=1}^d y_i^2\right],
		$$
		where $y_i$ are as in~\eqref{eq:definition yi}.
	\end{lemma}
	
	The lemma follows easily from Proposition~\ref{p:Stirling's formula}
	and the $\ve$-separation property~\eqref{eq:separated-def}. We omit the details.
	
	\smallskip

	\subsection{Technical lemmas on the bounds}
	Denote by $\ell_i$, $g_i$, and $m_i$ the shifted values of $\lambda_i$, $\gamma_i$, and $\mu_i$,
	respectively:
	\begin{equation}\label{eq:definition l,g,m}
	\ell_i\. := \. \lambda_i +d-i,\quad  g_i\. := \. \gamma_i +d-i, \quad m_i:= \mu_i +d-i,
	\end{equation}
	for all \ts $1\le i \le d$. Note that
	\begin{equation}\label{eq:l,g,m is positive}
	\ell_{i}-\ell_{j} \. \geq \. 1,  \quad  g_{i}- g_{j} \. \geq \. 1, \quad  m_{i}- m_{j} \. \geq \. 1,
	\end{equation}
	for all \ts $1 \leq i<j \leq d$.
	
	\smallskip
	
	\begin{lemma}\label{l:minimum bound with m,g,l}
		Let $d\geq 2$, $\ep >0$.
		Let \ts $(\lambda,\gamma,\mu)\in \Om(n,d,\ep)$ \ts be an  $\ep$-admissible triplet.
		Then:
		\begin{equation}\label{eq:min-bound-mgl}
		\left(\frac{g_i-g_j}{\gamma_i+j-i}\right) \, \frac{\min\bigl\{m_i-m_j, \frac{g_i}{g_i-g_j} \bigr\}  \, \min\bigl\{g_i-g_j, \frac{\ell_i}{\ell_i-\ell_j}\bigr\} }{\min\bigl\{m_i-m_j, \frac{\ell_i}{\ell_i-\ell_j}\bigr\} }  \ \leq \  \frac{32}{d^{2}\ts \ep^{12}}  \. \bigl((y_i-y_j)^2 + 1\bigr),
		\end{equation}
		for all \ts $1 \leqslant i < j \leqslant d$.
	\end{lemma}
	
	\smallskip
	
	We now build toward the proof of Lemma~\ref{l:minimum bound with m,g,l}.

	\smallskip
	
	\begin{sublemma}\label{l:minimum constant can be extracted}
		Let \ts $x,y,c\in \rr_+$, we have:
		$$
		\min\{1,c \}  \cdot \min \{x,y  \} \, \leq \, \min\{x,cy\} \, \leq \,  \max\{1,c \} \cdot \min\{x,y\}.
		$$
	\end{sublemma}
	
	\smallskip
	
	\begin{sublemma}\label{l:min of three things bound}
		For all \ts $x,y,z \in \rr_+$, we have:
		$$
		y \. \cdot \. \frac{\min\{x,\frac{1}{y}\} \ts \cdot \ts
			\min\{y,\frac{1}{z}\}}{\min\{x,\frac{1}{z}\}}  \, \leq \, 4 \left(y-\frac{x+z}{2} \right)^2 + 4.
		$$
	\end{sublemma}

	\smallskip
	
	Both sublemmas are elementary; we omit their proof for brevity.
	
	\smallskip
	
	\begin{proof}[Proof of Lemma~\ref{l:minimum bound with m,g,l}]
		We start with estimating $g_i$, $\ell_i$, and $\gamma_i$.
		Since $(\lambda,\gamma,\mu)$ is $\ep$-admissible, we have
		the following upper bound for $g_i$ and $\ell_i$:
		\begin{equation}\label{eq:g l upper bound}
		g_i \, \leq  \, \ell_i \, = \,  \lambda_i +d-i \, \leq \, |\lambda| + d-i \,
		\leq \, d \ts |\lambda| \, \leq_{\eqref{eq:def-admiss-pair}}
		\ \sum_{i=1}^d \. \frac{\lambda_i-\mu_i}{\ep}
		\, = \,  \frac{n}{\ep}\..
		\end{equation}
		Similarly, we have the following lower bounds:
		\begin{equation}\label{eq:l lower bound}
		\ell_i \, = \, \lambda_i +d-i \, \geq \, \lambda_i - \mu_i  \,
		\geq_{\eqref{eq:def-admiss-pair}}  \ \ep \ts |\lambda| \, \geq \, \ep \ts n\ts,
		\end{equation}
		\begin{equation}\label{eq:gamma lower bound}
		\gamma_i+j-i \, \geq  \, \gamma_i - \mu_i  \, \geq_{\eqref{eq:separated-def}} \ \frac{\ep^3}{2} \. |\lambda| \, \geq \, \frac{\ep^3}{2} \. n\ts.
		\end{equation}
		Finally, we have the following lower and upper bounds for~$p$ defined in~\eqref{eq:definition p}:
		\begin{equation}
		p \, =  \, 1- \frac{|\lambda|-|\gamma|}{n} \, \leq_{\eqref{eq:separated-def}} \
		1\. - \. \frac{d\ts \ep^3\ts |\lambda|}{2\ts n} \, \leq \, 1-\frac{d\ts \ep^3}{2}\,, \label{eq:p upper bound}
		\end{equation}
		\begin{equation}\label{eq:p lower bound}
		p  \, =  \, \frac{|\gamma|-|\mu|}{n}  \, \geq_{\eqref{eq:separated-def}} \
		\frac{d\ts \ep^3\ts |\lambda|}{2\ts n} \, \geq \, \frac{d\ts \ep^3}{2}\..
		\end{equation}
		Now note that
		\begin{equation}\label{eq:minimum 4}
		\begin{split}
		\frac{g_i-g_j}{\gamma_i+j-1} \, \leq_{\eqref{eq:gamma lower bound}} \ \frac{2\ts (g_i-g_j)}{\ep^3\ts n}.
		\end{split}
		\end{equation}
		Using repeatedly Sublemma~\ref{l:minimum constant can be extracted}, Proposition~\ref{p:admis-1d} and the above inequalities, we obtain:
		\begin{equation}\label{eq:minimum 1}
		\min \left\{m_i-m_j, \frac{g_i}{g_i-g_j} \right\} \,
		\leq_{\eqref{eq:p upper bound}, \. \eqref{eq:g l upper bound} }   \ \frac{1}{d\ep^4} \. \min \left\{{2\ts (1-p) \ts (m_i-m_j)}, \. \frac{n}{g_i-g_j} \right\},
		\end{equation}
		\begin{equation}\label{eq:minimum 2}
		\min \left\{g_i-g_j, \frac{\ell_i}{\ell_i-\ell_j} \right\} \,
		\leq_{\eqref{eq:g l upper bound}} \ \frac{2}{\ep} \. \min \left\{{g_i-g_j},  \.\frac{n}{2\ts p \. (\ell_i-\ell_j)} \right\},
		\end{equation}
		\begin{equation}\label{eq:minimum 3}
		\min \left\{m_i-m_j, \frac{\ell_i}{\ell_i-\ell_j} \right\} \
		\geq_{\eqref{eq:l lower bound},\. \eqref{eq:p lower bound}} \  \frac{d \ep^4}{2}\. \min \left\{{2\ts (1-p)\ts (m_i-m_j)}, \. \frac{n}{2\ts p \. (\ell_i-\ell_j)} \right\}.
		\end{equation}
		
		\nin
		By dividing \eqref{eq:minimum 1}--\eqref{eq:minimum 3} by $\sqrt{n}$, and combining these upper bounds with~\eqref{eq:minimum 4},
		we conclude:
		\begin{equation*}
		\text{LHS \ts in~\eqref{eq:min-bound-mgl}} \ \le \, \frac{8}{d^2\ep^{12}} \ts \left(\frac{g_i-g_j}{\sqrt{n}}\right)  \.
		\frac{\min\bigl\{\frac{2(1-p)(m_i-m_j)}{\sqrt{n}}, \frac{\sqrt{n}}{g_i-g_j} \bigr\}
			\,\cdot \,
			\min\bigl\{\frac{g_i-g_j}{\sqrt{n}}, \frac{\sqrt{n}}{\ell_i-\ell_j}\bigr\} }{
			\min\bigl\{\frac{m_i-m_j}{\sqrt{n}}, \frac{\sqrt{n}}{\ell_i-\ell_j}\bigr\} } \..
		\end{equation*}
		By Sublemma~\ref{l:min of three things bound}, the RHS of the equation above
		is bounded by
		$$\aligned
		& \frac{32}{d^{2}\ts \ep^{12}}\. \left( B_{ij}^2 \. + \. 1 \right), \quad \text{
			where}  \quad
		B_{ij} \,  = \, \frac{ (g_i-g_j) \. - \. (1-p)(m_i-m_j) \. - \. p\ts (\ell_i-\ell_j)}{\sqrt{n}} \, = \\
		& \hskip2.3cm =_{\text{\eqref{eq:definition l,g,m}}} \
		\frac{ (\gamma_i-\gamma_j) \. -  \. (1-p)\ts (\mu_i-\mu_j) \. - \.
			p\ts (\lambda_i-\lambda_j)}{\sqrt{n}} \, =_{\text{\eqref{eq:definition yi}}} \ (y_i\ts -\ts y_j)\ts.
		\endaligned
		$$
		This completes the proof.
	\end{proof}
	
	\smallskip
	
	\subsection{Upper bounds for solid triplets}
	Recall the definition of solid triplets in~$\S$\ref{ss:outline-solid-triplets-def}.
	We can now give an upper bound for the probability that a tableau random walk \ts $\bZ: \mu\to \la$ \ts
	goes through~$\ga$.
	
	\begin{lemma}\label{l:pmf of Zt upper bound}
		Fix $d\geq 2$ and $\ep >0$.  Let \ts $(\lambda,\gamma,\mu)\in \Om(n,d,\ep)$ \ts
		be an $\ep$-admissible solid triplet, with the solid constant~$C$ defined
		in~\eqref{eq:if conjecture is true}.
		Let \ts $k:=|\ga|-|\mu|$.   Then
		\begin{equation}\label{eq:lemma-pmf of Zt upper bound}
		\Pblm \big[Z_{k} = \gamma  \big] \, \leq \, C^3 \ts C_{d,\ep} \.
		n^{\frac{1-d}{2}} \. \prod_{1 \leqslant i < j \leqslant d } \. \bigl((y_i-y_j)^2 + 1\bigr) \.\cdot \.
		\exp \left[ - 2 \sum_{i=1}^d y_i^2\right],
		\end{equation}
		for an absolute constant \ts $C_{d,\ep} >0$.
	\end{lemma}
	
	\smallskip
	
	Note that the RHS in the lemma does not depend on~$k$.  This is by design, as
	$k$ will not be known, so we need a general upper bound.
	\smallskip
	
	\begin{proof}
		By directly counting the number of lattice paths \ts $\mu \to \gamma\to \la$,
		we obtain:
		\begin{equation}\label{eq:pmf-two-products}
		\Pblm \big[Z_{k} \, = \, \gamma  \big]  \, = \, \frac{f(\gamma/\mu) \. f(\lambda/\gamma)}{f(\lambda/\mu)} \,
		\le_{\eqref{eq:if conjecture is true}} \  C^3 \.\left[\frac{F(\gamma/\mu) \.  F(\lambda/\gamma)}{F(\lambda/\mu)}\right]
		\. \left[\frac{\eG(\gamma/\mu) \, \eG(\lambda/\gamma)}{\eG(\lambda/\mu)}\right].
		\end{equation}
		We now give an upper bound for the first product term:
		\begin{equation}\label{eq:pmf 2}
		\aligned
		&\frac{F(\gamma/\mu) \. F(\lambda/\gamma)}{F(\lambda/\mu)} \ \leq_{\text{Lem}~\ref{l:F -- asymptotic estimate}} \
		\ep^{-d(d-1)} \. \frac{G(\gamma/\mu) \. G(\lambda/\gamma)}{G(\lambda/\mu)} \\
		&\qquad \le_{\eqref{eq:G-function-def}} \
		\ep^{-d(d-1)} \.
		\left[\binom{\lambda_1-\mu_1}{\gamma_1-\mu_1}\cdots \binom{\lambda_d-\mu_d}{\gamma_d-\mu_d} \binom{n}{k}^{-1}\right] \.
		\prod_{1 \leqslant i < j \leqslant d}  \.  \frac{g_i-g_j}{\gamma_i+j-i}\\
		&\qquad \le_{\text{Lem~\ref{l:binomial formula asymptotic}}} \ \ep^{-d(d-1)} \.  B_{d,\ep} \, n^{-\frac{(d-1)}{2}}
		\. \exp \left[ - 2 \sum_{i=1}^d y_i^2\right] \. \prod_{1 \leqslant i < j \leqslant d}  \.  \frac{g_i-g_j}{\gamma_i+j-i}.
		\endaligned
		\end{equation}
		Combining the last products in RHS of~\eqref{eq:pmf-two-products} and~\eqref{eq:pmf 2}, we have
		\begin{equation}\label{eq:pmf-last-products}
		\aligned
		& \prod_{1 \leqslant i < j \leqslant d}  \.  \frac{g_i-g_j}{\gamma_i+j-i}\. \left[\frac{\eG(\gamma/\mu) \, \eG(\lambda/\gamma)}{\eG(\lambda/\mu)}\right] \\ & \qquad\le_{\eqref{eq:G-def}} \
		\prod_{1 \leqslant i < j \leqslant d} \.  \left(\frac{g_i-g_j}{\gamma_i+j-i}\right) \frac{\min\bigl\{m_i-m_j, \frac{g_i}{g_i-g_j} \bigr\}  \, \min\bigl\{g_i-g_j, \frac{\ell_i}{\ell_i-\ell_j}\bigr\} }{\min\bigl\{m_i-m_j, \frac{\ell_i}{\ell_i-\ell_j}\bigr\} } \\
		& \qquad \le_{\text{Lem~\ref{l:minimum bound with m,g,l}}} \ \left(\frac{32}{d^2\ts \ep^{12}}\right)^{d(d-1)/2} \, \prod_{1 \leqslant i < j \leqslant d } \. \bigl((y_i-y_j)^2 + 1\bigr).
		\endaligned
		\end{equation}
		The lemma now follows by combining \eqref{eq:pmf-two-products}, \eqref{eq:pmf 2} and~\eqref{eq:pmf-last-products}.
	\end{proof}
	
	\bigskip
	
	\section{Sorting probability via lattice paths}\label{s:sorting-lattice-path}

We use an upper bound for the probability mass function
of $Z_t$  and the results of Section~\ref{s:paths} which show that most triples 
are $\ep$-admissible, see Lemma~\ref{l:pdf upper bound} below for a precise statement. 
The upper bounds are derived via some technical asymptotic bounds from Section~\ref{s:technical}.

	\smallskip
	
	\subsection{Sorting probability of $\ep$-admissible pairs}
	The following technical lemma is central to our proof.
	
	\begin{lemma}\label{l:pdf upper bound}
		Fix $d\geq 2$ and $\ep >0$. Let \ts $(\lambda,\mu)\in \La(n,d,\ep)$ \ts
		be an $\ep$-admissible pair.  Let  $a$ be an integer that satisfies
		\begin{equation}\label{eq:a-ep-assumpltion}
		\ep \ \leq \ \frac{a-\mu_1}{\lambda_1-\mu_1} \ \leq \ 1-\ep.
		\end{equation}
		Suppose there exists a constant \ts $C=C_{\la,\mu}>0$, such that
		for every $\gamma$ for which $(\lambda,\gamma,\mu)\in \Om(n,d,\ep)$,
		this triplet is solid with solid constant~$C$.
		Then, there exists an absolute constant $C_{d,\ep}>0$ such that
		$$
		\vp(a) \. \leq \. C_{d,\ep} \.  \frac{C^3+1}{\sqrt{n}}\..
		$$
	\end{lemma}
	
	\begin{proof}
		Let $b$ be an arbitrary integer in $[\mu_2, \lambda_2]$.
		It follows from the definition of $\vp(a)$ in \eqref{eq:epsilon(a)} that it suffices to show that
		\begin{align*}
		\Pblm \big[\Ac(a,b) \big] \, \leq  \, C_{d,\ep} \.\frac{C^3+1}{\sqrt{n}}\,, \quad \text{for all} \ \ \,\mu_2\. \le\. b\. \le \. \la_2\..
		\end{align*}
		We start with
		\begin{equation}\label{eq:pdf 1}
		\begin{split}
		& \Pblm \big[\Ac(a,b) \big] \ =  \  \Pblm \big[Z_t(1)=a, \, Z_t(2)=b   \  \text{ for some } t\geq 0  \big] \\
		&  \quad =   \
		\Pblm \big[Z_t(1)=a, \,  Z_t(2)=b, \, (\lambda,Z_t,\mu) \notin\Om(n,d,\ep) \, \text{ for some $t\geq 0$} \big]  \\
		& \qquad \ +  \, \Pblm \big[Z_t(1)=a, \,  Z_t(2)=b, \, (\lambda,Z_t,\mu) \in \Om(n,d,\ep)  \text{ for some $t\geq 0$} \big].
		\end{split}
		\end{equation}
		We will bound each term in the RHS separately.
		
		Since $Z_t(1)=a$, by definition of $Z_t$ we have:
		$$
		a-\mu_1 \ \leq \  t \ \leq \   n - (\lambda_1-a).
		$$
		From~\eqref{eq:a-ep-assumpltion} and the assumption that $(\lambda,\mu)\in \La(n,d,\ep)$, it follows that
		$$
		\aligned
		\frac{t}{n} \ \geq \  \frac{a-\mu_1}{n} \ \geq \ \ep \. \frac{\lambda_1-\mu_1}{n} \ \geq \ \ep^2 \. \frac{|\lambda|}{n} \ & \geq \ \ep^2, \quad \text{and}\\
		1- \frac{t}{n} \ \geq \ \frac{a-\lambda_1}{n} \ \geq \ \ep \. \frac{\lambda_1-\mu_1}{n} \ & \geq \ \ep^2.
		\endaligned
		$$
		This implies that  condition \eqref{eq:assumption on t} holds.  By Lemma~\ref{l:most-triplets-admissible}, we then get:
		$$\Pblm\bigl[Z_t(1) = a, \, \. (\lambda,Z_t,\mu)\notin \Om(n,d,\ep) \bigr] \,
		\leq \, C_{d,\ep}  \, n^{\frac{d^2-1}{2}} \. e^{- 2 \sqrt{n}},
		$$
		for all $0\le t\le n$, and for some absolute constant \ts $C_{d,\ep}>0$.
		Thus, for the first term in the RHS of \eqref{eq:pdf 1}, we have:
		\begin{equation}
		\begin{split}\label{eq:pdf 2}
		&  \Pblm[Z_t(1) = a,  Z_t(a)=b, \ (\lambda,Z_t,\mu) \notin\Om(n,d,\ep) \text{ for some $t\geq 0 $}  ] \\
		&  \leq  \ \sum_{t=0}^{n-1} \. C_{d,\ep} \, n^{\frac{d^2-1}{2}} \. e^{- 2 \sqrt{n}}\,  =  \,
		C_{d,\ep} \, n^{\frac{d^2+1}{2}} \. e^{- 2 \sqrt{n}}.
		\end{split}
		\end{equation}
		For the second term in the RHS of~\eqref{eq:pdf 1}, denote by \ $\Gc(a,b)$ \ts
		the set of partitions given by
		$$
		\Gc(a,b) \, := \, \bigl\{ \gamma \ \mid \ \gamma_1=a, \ \gamma_2=b, \ \text{ and } \Gc(a,b)\in \La(n,d,\ep) \bigr\}.
		$$
		Then
		\begin{equation}\label{eq:pdf 3}
		\begin{split}
		&\    \Pblm \big[Z_t(1)=a, \,  Z_t(2)=b, \, (\lambda,Z_t,\mu) \in \Om(n,d,\ep)  \text{ for some $t\geq 0$}\big] \
		\leq  \. \sum_{\gamma \in \Gc(a,b) }  \Pblm \big[Z_t=\gamma \big]\\
		& \qquad \leq_{\text{Lem~\ref{l:pmf of Zt upper bound}}} \ \sum_{\gamma \in \Gc(a,b) } \. C^3  \.   n^{-\frac{(d-1)}{2}}  \
		\exp \left[ - 2 \ts \sum_{i=1}^d \ts y_i^2 \right] \, \prod_{1 \leqslant i < j \leqslant d } \bigl((y_i-y_j)^2 + 1\bigr).
		\end{split}
		\end{equation}
		
		Using \ts $(x-z)^2+1 \le (x^2+1)(z^2+1)$, we obtain:
		\begin{align*}
		\exp \left[ - 2 \sum_{i=1}^d y_i^2\right] \, \prod_{1 \leqslant i < j \leqslant d } \bigl((y_i-y_j)^2 + 1\bigr)
		\,\leq \,   \prod_{i=1}^d \.  \bigl(y_i^2+1\bigr)^{d-1} \ts e^{-2y_i^2}.
		\end{align*}
		Plugging this upper bound into \eqref{eq:pdf 3}, we obtain:
		\begin{equation}\label{eq:pdf 4}
		\text{RHS of~\eqref{eq:pdf 3}} \, \le \, C^3  \ts  n^{-\frac{(d-1)}{2}} \. \sum_{\gamma \in \Gc(a,b) }
		\. \prod_{i=1}^d  \. (y_i^2+1)^{d-1} \ts e^{-2y_i^2}.
		\end{equation}
		
		Note that for all $\gamma \in \Gc(a,b)$,
		the value $y_1$ and $y_2$ is fixed by the assumption that $\gamma_1=a$ and $\gamma_2=b$.
		For \ts $i \in \{3,\ldots,d\}$, it follows from \eqref{eq:definition yi} that as $\ga$
		varies between $\mu$ and $\la$, an increment of $\ga_i$ to $\ga'_i=\ga_i+1$ would
		lead to an increment in the $y$'s of order \ts $ \bigl|y_i'-y_i\bigr| \. =  \. n^{-1/2} (1- \frac{\lambda_i-\mu_i}{n} ) \. \geq  \. n^{-1/2} (d-1) \ep$ (by $\ep$-admissibility).
		Thus, we can bound each term for \ts $i\in \{3,\ldots,d\}$, as
		$$
		\sum_{z \in n^{-1/2} (d-1) \ep \, \zz} \. (z^2+1)^{d-1} \ts e^{-2z^2} \, \leq \,
		\frac{\sqrt{n}}{(d-1) \ep}\. \int_{-\infty}^{+\infty} (z^2+1)^{d-1} \ts e^{-2z^2} dz \, \leq \, \sqrt{n} \. C'_{d,\ep},
		$$
		since the integral converges.
		This allows us to bound~\eqref{eq:pdf 4} as
		\begin{equation*}
		\begin{split}
		\ & \sum_{\gamma \in \Gc(a,b) }   \
		\prod_{i=1}^d  (y_i^2+1)^{d-1} e^{-2y_i^2}
		\leq  \
		\left[ \prod_{i=1,2} \. (y_i^2+1)^{d-1} e^{-2y_i^2} \right] \, n^{\frac{d-2}{2}} \. \bigl(C'_{d,\ep}\bigr)^{d-2} \,
		\leq \, C''_{d,\ep} \, n^{\frac{d-2}{2}},
		\end{split}
		\end{equation*}
		where \. $C''_{d,\ep} := \left(\frac{d-1}{2}\right)^{2(d-1)}e^{-2(d+3)} (C_{d,\ep}')^{d-2}$.
		Thus we get the following upper bound for the second term in the RHS of~\eqref{eq:pdf 1}:
		\begin{equation}\label{eq:pdf 5}
		\begin{split}
		\    \Pblm \big[Z_t(1)=a, \,  Z_t(2)=b, \,\. (\lambda,Z_t,\mu)\in \Om(n,d,\ep) \, \text{ for some $t\geq 0$}\big]
		\leq   \ C_{d,\ep}'' \, C^3 \. \frac{1}{\sqrt{n}}.
		\end{split}
		\end{equation}
		Using  the upper bounds from \eqref{eq:pdf 2} and \eqref{eq:pdf 5}  in~\eqref{eq:pdf 1}, gives us:
		\begin{align*}
		\Pblm \big[\Ac(a,b) \big] \, \leq  \, C_{d,\ep} \, n^{\frac{d^2+1}{2}} \. e^{- 2 \sqrt{n}} \. + \.  C_{d,\ep}'' \, C^3 \. \frac{1}{\sqrt{n}}\..
		\end{align*}
		Since the second term dominates for sufficiently large~$n$, we obtain:
		\begin{align*}
		\Pblm \big[\Ac(a,b) \big] \, \leq  \, C_{d,\ep} \,  \frac{C^3+1}{\sqrt{n}}\.,
		\end{align*}
		as desired.
	\end{proof}
	
	\smallskip
	
	\subsection{Proof of Main Lemma~\ref{l:main}}\label{ss:technical-ML-proof}
	Let \ts $a:=\lfloor\frac{\mu_1+\lambda_1}{2}\rfloor$, so
	the first condition in Lemma~\ref{l:pdf upper bound} is satisfied.
	The second condition in Lemma~\ref{l:pdf upper bound} is satisfied
	by~\eqref{eq:if conjecture is true} and the definition of solid triplets.
	Lemma~\ref{l:pdf upper bound} combined with
	Lemma~\ref{l:quantitative bound for Linial}, gives:
	\[ \delta(P_{\lambda/\mu}) \, \leq \, 2 \ts C_{d,\ep} \.  \frac{C^3+1}{\sqrt{n}}\., \]
	for some absolute constant $C_{d,\ep} >0$, as desired.  \ $\sq$
	
	\bigskip
	
	\section{Upper bounds for the number of standard Young tableaux}\label{s:upper-bounds-SYTs}
	
	\subsection{Upper bound via Schur polynomials}
	In this subsection we give an upper bound to $f(\lambda/\mu)$ in terms of \ts
	$F(\lambda/\mu)$ (see \eqref{eq:F-def}), and evaluations of Schur polynomial
	(see \eqref{eq:Schur polynomial}). For $\la/\mu\vdash n$, recall the definition
	of shifted values \ts $\ell_i$ and~$m_i$ (see~\eqref{eq:definition l,g,m}).
	
	\smallskip
	
	\begin{lemma}\label{l:asymptotic upper bound for hook length of a cell}
		Let $\lambda$ be a partition.
		Then, for every \ts $(i,j) \in \lambda$, and every \ts $k\geq 0$, we have:
		$$
		\frac{h_{\lambda}(i+k,j+k)}{h_{\lambda}(i,j)}  \, \leq \, \frac{\ell_{i+k}}{\ell_{i}}\,.
		$$
	\end{lemma}

	\begin{proof}
		We have:
		$$
		h_{\la}(i+k,j+k) \. \leq \. h_{\la}(i,j) -2k +\la_{i+k}-\la_i \. = \. h_{\la}(i,j) -k +\ell_{i+k}-\ell_i\..
		$$
		Note that \ts $h_{\la}(i,j) = \la_i-i +\la'_j-j+1 \leq \la_i+d-i$.  Hence:
		$$
		\frac{h_{\la}(i+k,j+k)}{h_{\la}(i,j)} \, \leq \, 1 \. - \. \frac{\ell_i - \ell_{i+k} +k}{h_{\la}(i,j)}
		\, \leq \, 1 \. - \. \frac{\ell_i - \ell_{i+k} +k}{\ell_i} \, \leq  \, \frac{\ell_{i+k}}{\ell_i}\,,
		$$
		as desired.
	\end{proof}
	
	\smallskip
	
	We now apply Lemma~\ref{l:asymptotic upper bound for hook length of a cell}
	to derive  an upper bound for the product of hooks of a flagged tableau.
	Let \ts $T\in \FT(\lambda/\mu)$.  Recall the notation~\eqref{eq:def-ti},
	for the number \ts $t_i(T)$ \ts of $i$'s in~$T$.
	Lemma~\ref{l:asymptotic upper bound for hook length of a cell}
	immediately gives:
	
	\smallskip
	
	\begin{cor}\label{c:asymptotic upper bound for an excited diagram}
		Let $d \geq 2$, and let $T$ be a flagged tableau of $\lambda/\mu$.
		Then:
		$$
		\prod_{(i,j) \in \mu} \. \frac{h_{\lambda}\bigl(T(i,j), \ts j+T(i,j)-i\bigr)}{h_{\lambda}(i,j)}
		\ \leq \    \prod_{i=1}^d  \. \frac{(\ell_i)^{t_i(T)}}{(\ell_i)^{\mu_i}}\,.
		$$
	\end{cor}
	
	\smallskip
	
	We now arrive to the main result of this subsection.
	
	\smallskip
	
	\begin{lemma}\label{l:NHLF  upper bounded by schur polynomial}
		Let $d \geq 2$, and let $\lambda, \mu$ be partitions such that $\mu \subseteq  \lambda$.
		Then
		\[ 1 \ \leq \  \frac{f(\lambda/\mu)}{F(\lambda/\mu)}   \ \leq \ \frac{s_{\mu}(\ell_1,\ldots,\ell_d)}{\ell_1^{\mu_1} \ldots \ell_d^{\mu_d}}\..\]
	\end{lemma}
	
	\begin{proof}
		The lower bound is given in Theorem~\ref{t:NHLF-asy}.  \ts
		For the upper bound, we have:
		\begin{align*}
		\frac{f(\lambda/\mu)}{F(\lambda/\mu)}   \ & =_{\text{Thm~\ref{t:NHLF}}} \  \sum_{T \in \Ec(\lambda/\mu)}
		\, \prod_{(i,j)  \in \mu} \frac{h_{\lambda}\bigl(T(i,j), j+T(i,j)-i\bigr)}{h_{\lambda}(i,j)}
		\ \leq_{\text{Cor~\ref{c:asymptotic upper bound for an excited diagram}}} \ \sum_{T \in \Ec(\lambda/\mu)}
		\,  \prod_{i=1}^d \, \frac{\ell_i^{t_i(T)}}{\ell_i^{\mu_i}} \\
		\ & \leq_{\text{\eqref{eq:definition flagged tableau}}} \  \frac{1}{\ell_1^{\mu_1} \ldots \ell_d^{\mu_d}} \, \sum_{T \in \text{SSYT}(\mu)} \ell_1^{t_1(T)}\ldots \ell_d^{t_d(T)} \ \leq_{\text{\eqref{eq:Schur polynomial}}} \  \frac{s_{\mu}(\ell_1,\ldots,\ell_d)}{\ell_1^{\mu_1} \ldots \ell_d^{\mu_d}}\,,
		\end{align*}
		as desired.
	\end{proof}

	\smallskip
	
	\subsection{Interval decomposition upper bound}
	In this subsection we give a refinement to  the upper bound in
	Lemma~\ref{l:NHLF  upper bounded by schur polynomial}.
	An interval decomposition of \ts $[d]=\{1,\ldots,d\}$ \ts is defined
	as the following collection of subsets: \ts $\CB:=(B_1,\ldots, B_r)$, where
	\begin{equation}\label{eq:def-CB}
	B_1\. := \. \{ 1,\ldots, b_1\}\,, \quad  B_2\. := \. \{ b_1+1,\ldots, b_2\}\,,
	\ \ \ldots \ \ , \ B_{r}\.:=\. \{b_{r-1}+1,  \ldots,d \},
	\end{equation}
	for some \ts $0 \ts =\ts b_0 < b_1 < b_2 < \ldots < b_r \ts =\ts d$ \ts
	and \ts $r \geq 1$.
	
	For all \ts $i, j \in [d]$, we write \ts $i \Bsim j$ \ts when $i$ and $j$ are
	contained in the same partition in $B_1,\ldots,B_r$, and \ts $i \Bnsim j$ \ts otherwise.
	We drop~$\CB$ when the partition is clear.
	%
	Let
	\begin{equation}\label{eq:N(lambda,B)}
	N(\ell,\CB) \ :=   \ \max \left \{ \, \frac{\ell_i}{\ell_i-\ell_j} \
	\Big| \ {1 \leqslant i < j \leqslant d  \ \text{ and } \ i \Bnsim j} \right\}\ts,
	\end{equation}
	and let $N(\lambda,\CB):= 0$ for $r=d$.
	The main result of this section is the following  upper bound:
	
	\smallskip
	
	\begin{thm}
		\label{t:interval}
		Fix \ts $d\geq 2$.  Let \ts $\lambda/\mu\vdash n$, such that \ts $\la,\mu \in \pp_d$, and let
		$\CB$ be an interval decomposition of~$[d]$.
		Then:
		\begin{equation}\label{eq:interval-thm}
		\frac{s_{\mu}(\ell_1,\ldots,\ell_d)}{\ell_1^{\mu_1} \ldots \ell_d^{\mu_d}} \ \leq \ C_d \prod_{\substack{1 \leqslant i < j \leqslant d \\ i \eBsim j}} (m_i-m_j+  N(\ell,\CB)) \ \prod_{\substack{1 \leqslant i < j \leqslant d \\ i \eBnsim j}} \frac{\ell_i}{\ell_i-\ell_j}\,,
		\end{equation}
		for some absolute constant $C_d >0$.
	\end{thm}
	
	Lemma~\ref{l:NHLF  upper bounded by schur polynomial} and Theorem~\ref{t:interval} immediately imply:
	
	\begin{cor}[\text{\rm Interval Upper Bound}] \label{c:interval}
		In notation of Theorem~\ref{t:interval},
		$$
		\frac{f(\lambda/\mu)}{F(\lambda/\mu)}   \ \leq \ \ C_d \prod_{\substack{1 \leqslant i < j \leqslant d \\ i \eBsim j}} (m_i-m_j+  N(\ell,\CB)) \ \prod_{\substack{1 \leqslant i < j \leqslant d \\ i \eBnsim j}} \frac{\ell_i}{\ell_i-\ell_j}\,.
		$$
	\end{cor}
	
	\smallskip
	
	\subsection{Expanding the determinant}
	We now build toward the proof of Theorem~\ref{t:interval}.
	Our strategy is to break down the Schur function evaluated at the sequence $\ell_1,\ldots,\ell_d$
	into evaluations of separate parts, and use either~\eqref{eq:Schur-def}
	when the values of $\ell_i$ are sufficiently distinct,
	or~\eqref{eq:HCF} when they are close.
	Denote
	\begin{equation}\label{eq:def-rM}
	\rM \, := \, \Big(x_{j}^{m_i}\Big)_{i,j=1}^d \..
	\end{equation}
	
	\smallskip
	
	\begin{lemma}\label{l:determinant estimate}
		Fix \ts $d\geq 2$.  Let \ts $\mu\in \pp_d$, and let \ts
		$x_1 \ge \ldots \ge x_d>0$. Then:
		$$0 \, \le \,
		\det \rM  \ \leq \
		{x_1^{m_1}\ldots x_d^{m_d}} \prod_{1 \leqslant i < j \leqslant d} \frac{(m_i-m_j)(x_i-x_j)}{(j-i) \ts x_i}\,.
		$$
	\end{lemma}
	
	\begin{proof}
		The first inequality follows from~\eqref{eq:Schur-def}:
		\begin{equation}\label{eq:det-Schur-lemma}
		\det \rM  \,\. = \,
		s_{\mu}(x_1,\ldots,x_d) \, \prod_{1 \leqslant i < j \leqslant d} (x_i-x_j) \, = \,
		\sum_{A \in \text{SSYT}(\mu)} {x_1^{t_1(A)}\ldots x_d^{t_d(A)}} \, \prod_{1 \leqslant i < j \leqslant d} (x_i-x_j) \, \geq \, 0.
		\end{equation}
		For the second inequality, since \ts $x_1\geq  \ldots \geq x_d$, and \ts $\mu_1\ge t_1(A)$, $\mu_1+\mu_2 \ge  t_1(A) + t_2(A)$,
		\ldots, we have:
		\begin{equation}\label{eq:majorization-Schur}
		x_1^{t_1(A)}\ldots x_d^{t_d(A)} \, \leq \, x_1^{\mu_1}\ldots x_d^{\mu_d}.
		\end{equation}
		We conclude:
		$$
		\aligned
		\det \rM \ &\leq_{\eqref{eq:det-Schur-lemma}, \. \eqref{eq:majorization-Schur}}  \
		\sum_{A \in \text{SSYT}(\mu)} {x_1^{\mu_1}\ldots x_d^{\mu_d}} \,
		\prod_{1 \leqslant i < j \leqslant d} (x_i-x_j) \\
		& \leq_{\eqref{eq:HCF}} \
		{x_1^{\mu_1}\ldots x_d^{\mu_d}} \
		\prod_{1 \leqslant i < j \leqslant d} \, \frac{(m_i-m_j)(x_i-x_j)}{(j-i)}\,,
		\endaligned
		$$
		which implies the result by the definition~\eqref{eq:definition l,g,m}.
	\end{proof}
	
	\smallskip
	
	To simplify presentation, we use notation \ts $\DET(A):=|\det(A)|$.
	
	\begin{lemma}\label{l:Schur-pol-exp}
		Fix \ts $d\geq 2$.  Let \ts $\mu\in \pp_d$, \ts
		$x_1 \ge \ldots \ge x_d>0$, and let $\CB$ be an interval decomposition of $[d]$.
		Then:
		\begin{align}\label{eq:Schur-pol-exp}
		{s_{\mu}(x_1,\ldots,x_d)}
		\ \leq \          \ \sum_{\sigma \in S_d} \,
		\frac{x_1^{m_{\sigma(1)}}\ldots x_d^{m_{\sigma(d)}} }{x_1^{d-1}\ldots x_d^{d-d} }
		\, \prod_{ \substack{ 1 \leqslant i < j \leqslant d \\
				i \eBsim j } } \ts \frac{|m_{\sigma(i)}-m_{\sigma(j)}|}{j-i}  \,
		\prod_{ \substack{ 1 \leqslant i < j \leqslant d \\
				i \eBnsim j } } \, \frac{x_i}{x_i - x_j}.
		\end{align}
	\end{lemma}
	
	\begin{proof}
		Apply the Laplace expansion of \ts $\rM$ \ts along the interval decomposition \ts
		$\CB=(B_1,\ldots, B_r)$, defined as in~\eqref{eq:def-CB}.  We get:
		\begin{align*}
		\det \rM \ = \,  \sum_{\sigma \in S_d/\St(\CB)} \.
		\text{sign}(\sigma)\, \prod_{k=1}^r \. \det  \big[ x_{j}^{m_{\sigma(i)}}\big]_{i,j \in B_k}\,,
		\end{align*}
		where \. $\St(\CB)\ssu S_d$ \ts is the stabilizer subgroup of~$\CB$, so \ts
		$\St(\CB)\simeq S_{b_1} \times S_{b_2-b_1}\times \ldots \times S_{d-b_{r-1}}$. We have:
		\begin{equation}\label{eq:determinant expansion 1}
		\det \rM \ \,\leq  \ \sum_{\sigma \in S_d} \,
		\prod_{k=1}^r \,  \DET \. \big[\ts x_{j}^{m_{\sigma(i)}}\big]_{i,j \in B_k} \..
		\end{equation}
		We now analyze each term in the right side of \eqref{eq:determinant expansion 1} separately.
		We have for every $\sigma \in S_d$ that
		\begin{align*}
		\prod_{k=1}^r \,  \DET \. \big[\ts x_{j}^{m_{\sigma(i)}}\big]_{i,j \in B_k}  \  & \leq_{\text{Lem~\ref{l:determinant estimate}}} \
		\prod_{k=1}^r \ \prod_{i\in B_k} \, {x_i^{m_{\sigma(i)}}} \
		\prod_{ \substack{j \in B_k\\ j > i} } \frac{|m_{\sigma(i)}-m_{\sigma(j)}|\cdot (x_i-x_j)}{x_i\ts (j-i)} \\
		&  \le  \   x_1^{m_{\sigma(1)}} \ts \cdots \,x_d^{m_{\sigma(d)}} \prod_{ \substack{ 1 \leqslant i < j \leqslant d \\ i \eBsim j } }  \frac{|m_{\sigma(i)}-m_{\sigma(j)}|\cdot (x_i-x_j)}{x_i\ts (j-i)}\,.
		\end{align*}
		Using the inequality above for the RHS of~\eqref{eq:determinant expansion 1}, we obtain:
		\begin{equation*}
		\det \rM \ \leq  \ \sum_{\sigma \in S_d} \,  x_1^{m_{\sigma(1)}}\ts \cdots \, x_d^{m_{\sigma(d)}}  \,
		\prod_{ \substack{ 1 \leqslant i < j \leqslant d \\ i \eBsim j } }  \frac{|m_{\sigma(i)}-m_{\sigma(j)}|\cdot (x_i-x_j)}{x_i\ts (j-i)}
		\end{equation*}
		We conclude:
		\begin{align*}
		s_{\mu}(x_1,\ldots,x_d) \ & =_{\eqref{eq:Schur-def}} \ \det \rM  \prod_{1 \leqslant i < j \leqslant d} \, \frac{1}{x_i-x_j} \\
		& \le   \prod_{1 \leqslant i < j \leqslant d} \. \frac{1}{x_i-x_j} \ \sum_{\sigma \in S_d}  \. x_1^{m_{\sigma(1)}}\ts \cdots \,  x_d^{m_{\sigma(d)}}   \prod_{ \substack{ 1 \leqslant i < j \leqslant d \\ i \eBsim j } }  \frac{|m_{\sigma(i)}-m_{\sigma(j)}|\cdot (x_i-x_j)}{x_i\ts (j-i)}   \\
		& \le  \  \prod_{1 \leqslant i < j \leqslant d} \frac{x_i}{x_i-x_j} \sum_{\sigma \in S_d}   \frac{x_1^{m_{\sigma(1)}}\ts \cdots \,  x_d^{m_{\sigma(d)}}}{x_1^{d-1}\ts \cdots \,  x_d^{d-d}}   \prod_{ \substack{ 1 \leqslant i < j \leqslant d \\ i \eBsim j } }  \frac{|m_{\sigma(i)}-m_{\sigma(j)}|\cdot (x_i-x_j)}{x_i\ts (j-i)}  \\
		& \le  \          \ \sum_{\sigma \in S_d}
		\frac{x_1^{m_{\sigma(1)}}\ts \cdots \,  x_d^{m_{\sigma(d)}} }{x_1^{d-1}\ts \cdots \,  x_d^{d-d} }  \prod_{ \substack{ 1 \leqslant i < j \leqslant d \\ i \eBsim j } } \frac{|m_{\sigma(i)}-m_{\sigma(j)}|}{j-i}  \prod_{ \substack{ 1 \leqslant i < j \leqslant d \\ i \eBnsim j } } \frac{x_i}{x_i - x_j}\,,
		\end{align*}
		as desired.
	\end{proof}
	
\smallskip
	
\subsection{Simplifying the products}
In this subsection we will simplify the upper bound in Lemma~\ref{l:Schur-pol-exp}.
Our goal is to remove the dependence to $\sigma \in S_d$ in the RHS of ~\eqref{eq:Schur-pol-exp}.
We start with  the following two technical lemmas.

\smallskip
	
Let $\si\in S_d$, and let $1\le a<b\le d$.  We say that
$(a,b)$ is an \emph{inversion} in~$\si$, if \ts $\si(a)>\si(b)$.
Denote by $(ab)$ a transposition in~$S_d$.
	
	\smallskip
	
\begin{lemma}\label{l:rearrangement technical lemmas}
		Fix \ts $m_1\geq \ldots \geq m_d>0$ \ts and \ts $x_1\geq \ldots \geq x_d>0$.
		Let \ts $\sigma \in S_d$, and let $\tau=(ab)\in S_d$.
		Then:
		\begin{equation}\label{eq:rearrangement explicit equality}
		\prod_{i=1}^d \frac{x_i^{m_{\sigma(i)}}}{x_i^{m_{\sigma\tau(i)}}}  \ = \ \left(\frac{x_b}{x_a}\right)^{m_{\sigma(b)}-m_{\sigma(a)}}.
		\end{equation}
		Furthermore, when $(a,b)$ is an inversion of $\sigma$, we have
		\begin{equation}\label{eq:rearrangement inequality classical}
		\left(\frac{x_b}{x_a}\right)^{m_{\sigma(b)}-m_{\sigma(a)}} \ \leq  \ 1. \end{equation}
\end{lemma}
	
Both claims are straightforward; we omit the proof.
	
\smallskip
	
\begin{lemma}\label{l:exp-bound-mi-mj-simplified}
		Fix \ts $x_1\geq x_2>0$, and let $m\ge 0$. Then:
		\[  m \, \left( \frac{x_2}{x_1}\right)^m \ \leq \ \frac{x_1}{e\ts (x_1-x_2)}\..  \]
	\end{lemma}

\begin{proof}
		Substitute \ts $y=\frac{x_2}{x_1}$, and note that the function \ts $m \ts(1-y)\ts y^m$ \ts
		achieves maximum at \ts $y=(m-1)/m$, which \ts $\to 1/e$ \ts from below as \ts $m\to \infty$.
	\end{proof}
	
\smallskip
	
	Let $\sigma\in S_d$.  For all $a=1,\ldots,d$, define permutations \ts $\si_a$ \ts and \ts
	$\tau_{a} \in S_d$ \ts recursively:
	\begin{equation}\label{eq:def-si-tau}
	\si_a \, := \, \sigma \tau_1 \ldots \tau_{a-1} \qquad \text{and} \qquad
	\tau_a \, := \,
	\begin{cases}
	\ 1 \ & \text{ if } \ \si_a(a)= a\ts,\\
	\ \bigl(a\ts \si_a^{-1}(a)\bigr) \ &  \text{ if } \ \si_a(a)\neq a\ts.
	\end{cases}
	\end{equation}
	In other words, at each step $a$, we
	modify the permutation $\sigma_{a}$ so that the resulting permutation $\sigma_{a+1}$ has $a$ as a fixed point,
	by switching \ts $a$ \t and \ts $\sigma_a^{-1}(a)$ \ts if necessary.
	It follows from the construction that, at each step,
	either $a=\sigma_a^{-1}(a)$ or
	$(a, \sigma_a^{-1}(a))$ is an inversion of $\sigma_a$.
	Observe that \ts $\sigma_1=\sigma $ \ts and \ts $\sigma_{d}=1$.
	
	Denote by $R_a$ the number
	\begin{equation}\label{eq:Raij}
	R_{a}:=   \left(\frac{x_b}{x_a}\right)^{c}, \quad \text{where} \ \ b \. = \. \sigma_{a}^{-1}(a) \ \ \ \text{and} \  \ c \. = \. \frac{2}{d(d-1)}\bigl(m_{\sigma_a(b)}-m_{\sigma_a(a)}\bigr).
	\end{equation}
	It follows from Lemma~\ref{l:rearrangement technical lemmas}, that
	\begin{equation}\label{eq:Raij less than 1}
	R_{a} \, \leq \, 1\ts.
	\end{equation}
	Indeed, either we have \ts $a=b$, or by construction~\eqref{eq:def-si-tau} we have \ts $(a,b)$ \ts is an inversion in \ts $\sigma_a$\ts.
	
	\smallskip
	
	Let \ts $\CB$ \ts be an interval decomposition of~$[d]$.
	Recall from definition \eqref{eq:N(lambda,B)} that
	\begin{equation}\label{eq:N}
	N(x,\CB) \, := \,
	\max \left \{ \frac{x_i}{x_i-x_j}\ \mid \ {1 \leqslant i < j \leqslant d  \text{ and } i \eBnsim j} \right \} \, > \. 0\ts.
	\end{equation}
	For all \, $1 \leq i <j \leq d$, denote
	\begin{equation}\label{eq:Haij}
	H_{a}(i,j) \, := \,  \bigl|m_{\sigma_a(i)}- m_{\sigma_a(j)} \bigr| \. + \. (a-1) \. \frac{d(d-1)}{e}\. N(x,\CB).
	\end{equation}
	It follows from the definition~\eqref{eq:def-si-tau}, that  $H_a$ satisfies
	\begin{equation}\label{eq:Haij recursion}
	H_{a+1}(i,j) \, = \,  H_{a}(\tau_a(i), \tau_a(j)) \. + \. \frac{d(d-1)}{e} \. N(x,\CB).
	\end{equation}
	Note also that
	\begin{equation*}\label{eq:Haij symmetric}
	H_{a}(i,j) \, = \,  H_{a}(j,i) \quad \text{for all} \ \, 1\le i, \ts j \le d.
	\end{equation*}

\smallskip
	
The following two lemmas utilize and clarify the properties of numbers \ts $R_a$ \ts
and \ts $H_{a}(i,j)$ \ts defined above.
The idea is that we can now rewrite the RHS of~\eqref{eq:Schur-pol-exp} as
\[
x_1^{\mu_1}\ldots x_d^{\mu_d} \, \sum_{\sigma \in S_d} \, \prod_{a=1}^{d-1} \. R_a^{\frac{d(d-1)}{2}} \,
\prod_{ \substack{ 1 \leqslant i < j \leqslant d \\ i \eBsim j } }  \.  \frac{H_1(i,j)}{j-i}
\,  \. \prod_{ \substack{ 1 \leqslant i < j \leqslant d \\ i \eBnsim j } } \, \frac{x_i}{x_i - x_j}\,.
\]
%
Our goal is to iteratively replace all \ts $H_{1}(i,j)$'s (which depend on $\sigma$)
with  $H_{d}(i,j)$'s (which do not depend on $\sigma$), and   $R_a$'s will
be  the cost that we are paying for each iteration.

\smallskip

	\begin{lemma}[{\rm The same block estimate}]\label{l:rear-ineq-a=b}
		Fix \ts $m_1\geq \ldots \geq m_d>0$ and $x_1\geq \ldots\geq  x_d>0$.
		Let $\CB$ be an interval decomposition of $[d]$, and let $\sigma \in S_d$.
		Then, for all \ts $1\le a \le d-1$, such that \ts $a \Bsim \sigma_a^{-1}(a)$, we have:
		\begin{equation*}
		\prod_{ \substack{ 1 \leqslant i < j \leqslant d \\ i \eBsim j } }  \, H_a(i,j)    \ \leq  \   \prod_{ \substack{ 1 \leqslant i < j \leqslant d \\ i \eBsim j } } \, H_{a+1}(i,j).
		\end{equation*}
	\end{lemma}
	\begin{proof}
		It follows from~\eqref{eq:Haij recursion} that
		\begin{align}\label{eq:rearrangement when a is equal to b 1}
		\prod_{ \substack{ 1 \leqslant i < j \leqslant d \\ i \eBsim j } } \, H_{a+1}(i,j)  \ \geq  \ \prod_{ \substack{ 1 \leqslant i < j \leqslant d \\ i \eBsim j } } \, H_{a}\bigl(\tau_a(i),\tau_a(j)\bigr).
		\end{align}
		Note that the RHS can be rewritten as
		\begin{equation}\label{eq:rearrangement when a is equal to b 2}
		\prod_{ \substack{ 1 \leqslant i < j \leqslant d \\ i \eBsim j } } \,
		H_{a}\bigl(\tau_a(i),\tau_a(j)\bigr) \ =  \ \prod_{k=1}^{r} \,
		\prod_{\substack{ 1 \leqslant i < j \leqslant d, \\ i,  j   \in B_k  }  }
		\, H_{a}\bigl(\tau_a(i),\tau_a(j)\bigr) \ =  \ \prod_{k=1}^{r}
		\prod_{\substack{ 1 \leqslant i < j \leqslant d \\ i,  j   \in \tau_{a}^{-1}(B_k)  }  } \, H_{a}(i,j)\..
		\end{equation}
		
		Now note that $a$ and $b$ are contained in the same block in~$\CB$, since \ts $a \Bsim b$ by assumption.
		Since $\tau_a=(ab)$, this implies that  $\tau_a(B_k)= B_k$ for all \ts $1\le k \le r$.
		Thus, we have:
		\begin{align}\label{eq:rearrangement when a is equal to b 3}
		\prod_{k=1}^{r} \, \prod_{\substack{ 1 \leqslant i < j \leqslant d, \\ i,  j   \in \tau_{a}^{-1}(B_k)  }  } \, H_{a}(i,j) \ = \
		\prod_{k=1}^{r} \, \prod_{\substack{ 1 \leqslant i < j \leqslant d, \\ i,  j   \in B_k  }  } \, H_{a}(i,j)
		\ =  \ \prod_{ \substack{ 1 \leqslant i < j \leqslant d \\ i \eBsim j } } \, H_{a}(i,j).
		\end{align}
		The lemma now follows by combining \eqref{eq:rearrangement when a is equal to b 1}, \eqref{eq:rearrangement when a is equal to b 2}, and \eqref{eq:rearrangement when a is equal to b 3}.
	\end{proof}

	\begin{lemma}[{\rm The distinct blocks estimate}]\label{l:rear-ineq-a-neq-b-prelim}
		Fix \ts $m_1\geq \ldots \geq m_d>0$ and $x_1\geq \ldots\geq  x_d>0$.
		Let $\CB$ be an interval decomposition of $[d]$, and let $\sigma \in S_d$.
		Then, for all \ts $1\le a \le d-1$, such that \ts $a \Bnsim \sigma_a^{-1}(a)$,
		and all \ts $1 \leqslant i < j \leqslant d$, we have:
		\begin{equation}\label{eq:rear-ineq-a-neq-b-cor}
		R_a \ts H_{a}(i,j) \, \leq \. H_{a+1}(i,j).
		\end{equation}
	\end{lemma}

	\begin{proof}  We first prove the following bound:
		\begin{equation}\label{eq:rear-ineq-a-neq-b}
		R_a \ts \bigl|H_{a}(\tau_a(i),j)-H_a(i,j)\bigr|  \, \leq \, \frac{d(d-1)}{2e}\. N(x,\CB),
		\end{equation}
		all \ts $1 \leqslant i < j \leqslant d$.
		
		Let \ts $b:=\sigma_a^{-1}(a)$ \ts and \ts
		suppose that \ts $i \notin \{a,b\}$. Then \ts $\tau_a(i)=i$ \ts by the definition \eqref{eq:def-si-tau}.
		It then follows that the LHS of~\eqref{eq:rear-ineq-a-neq-b} is equal to $0$.
		Suppose now that \ts $i \in \{a,b\}$.  Equation~\eqref{eq:rear-ineq-a-neq-b} then becomes
		\[   R_a \ts \bigl|H_{a}(b,j)-H_a(a,j)\bigr|  \, \leq \, \frac{d(d-1)}{2e}\.N(x,\CB)\ts.
		\]
		Note that
		\begin{align*}
		|H_{a}(b,j)-H_a(a,j)|  \,  =_{\eqref{eq:Haij}} \ \left| \. |m_{\sigma_a(b)}- m_{\sigma_a(j)} |  \. - \.
		|m_{\sigma_a(j)}- m_{\sigma_a(a)} | \. \right| \, \leq \, \bigl|m_{\sigma_a(b)}- m_{\sigma_a(a)} \bigr| \..
		\end{align*}
		This implies that
		\begin{align*}
		R_a \ts \bigl|H_{a}(b,j)-H_a(a,j)\bigr| \ \leq \  R_a \ \big|m_{\sigma_a(b)}- m_{\sigma_a(a)} \big|
		\ =_{\eqref{eq:Raij}} \ \left(\frac{x_b}{x_a}\right)^c  \big| m_{\sigma_a(b)}-m_{\sigma_a(a)} \big|\ts,
		\end{align*}
		where $c$ is also defined in~\eqref{eq:Raij}.
		Now note that $(a,b)$ is an inversion of $\sigma_a$ by construction~\eqref{eq:def-si-tau}.
		Apply Lemma~\ref{l:exp-bound-mi-mj-simplified} with \ts $x_1\gets x_a$, \ts $x_2\gets x_b$
		and \ts $m\gets c$, to get
		\[
		\left(\frac{x_b}{x_a}\right)^{c} \. \big| m_{\sigma_a(b)}-m_{\sigma_a(a)} \big|
		\, \leq \,   \frac{d(d-1)}{2\ts e} \, \frac{x_a}{x_a-x_b}\..
		\]
		Since \ts $a \eBnsim b$ \ts and \ts $a <b$ \ts by the construction~\eqref{eq:def-si-tau}
		of $\sigma_a$, we have \ts $\frac{x_a}{x_a-x_b} \leq N(x,\CB)$ by~\eqref{eq:N},
		and the inequality~\eqref{eq:rear-ineq-a-neq-b} follows.
		
		Therefore, we have:
		\begin{align*}
		& H_{a+1}(i,j)  \, - \, R_{a} \ts H_a(i,j)  \
		=_{\eqref{eq:Haij recursion}} \ H_{a}(\tau_a(i),\tau_a(j)) \. + \. \frac{d(d-1)}{e} N(x,\CB) \. - \. R_{a} H_a(i,j) \\
		& \quad  \geq_{\eqref{eq:Raij less than 1}}  \ R_a \ts H_{a}(\tau_a(i),\tau_a(j))
		\. + \. \frac{d(d-1)}{e} \. N(x,\CB) \. - \. R_{a} \ts H_a(i,j) \\
		&  \quad \ge   \  R_a \ts H_{a}(\tau_a(i),\tau_a(j)) \. - \. R_a \ts H_{a}(\tau_a(i),j)
		\. + \.  R_a \ts H_{a}(\tau_a(i),j) \. - \. R_a \ts H_{a}(i,j)   \.  + \.
		\frac{d(d-1)}{e} \. N(x,\CB)  \\
		& \quad \geq_{\eqref{eq:rear-ineq-a-neq-b}}  \  -\.\frac{d(d-1)}{2e} \. N(x,\CB) \. - \.   \frac{d(d-1)}{2e} \. N(x,\CB)
		\. + \.  \frac{d(d-1)}{e} \. N(x,\CB)  \  =  \ 0\ts.
		\end{align*}
		This proves the lemma.
	\end{proof}
	
	\smallskip

	\subsection{Putting everything together}
	We now combine Lemma~\ref{l:rear-ineq-a=b} and Lemma~\ref{l:rear-ineq-a-neq-b-prelim}
	to get the following upper bound.
	
	\begin{lemma}\label{l:upper bound for mi-mj}
		Fix \ts $m_1\geq \ldots \geq m_d>0$ and $x_1\geq \ldots\geq  x_d>0$.
		Let $\CB$ be an interval decomposition of $[d]$, and let $\sigma \in S_d$.
		Then:
		\begin{equation}\label{eq:upper-bound-mi-mj-lemma}
		\frac{x_1^{m_{\sigma(1)}}\ldots x_d^{m_{\sigma(d)}} }{x_1^{m_1}\ldots x_d^{m_d} }    \prod_{ \substack{ 1 \leqslant i < j \leqslant d \\ i \eBsim j } }  |m_{\sigma(i)}-m_{\sigma(j)}|   \ \leq \  \prod_{ \substack{ 1 \leqslant i < j \leqslant d \\ i \eBsim j } } \left(m_{i}-m_{j}+ \frac{d(d-1)^2}{e} N(x,\CB) \right).
		\end{equation}
	\end{lemma}
	
	\begin{proof}
		We have:
		\[
		\frac{x_1^{m_{\sigma(1)}}\ldots x_d^{m_{\sigma(d)}} }{x_1^{m_1}\ldots x_d^{m_d} }  \ =_{\eqref{eq:rearrangement explicit equality}} \ \prod_{a=1}^{d-1} \left(\frac{x_{b}}{x_a}\right)^{{m_{\sigma_a(b)}-m_{\sigma_a(a)}}} \ =_{\eqref{eq:Raij}} \ \prod_{a=1}^{d-1} R_a^{\frac{d(d-1)}{2}}.
		\]
		We can rewrite the inequality~\eqref{eq:upper-bound-mi-mj-lemma} in the lemma using the definition~\eqref{eq:Haij} as follows:
		\begin{align}\label{eq:upper-bound-mi-mj-rewritten}
		\prod_{a=1}^{d-1} \. R_a^{\frac{d(d-1)}{2}} \,
		\prod_{ \substack{ 1 \leqslant i < j \leqslant d \\ i \eBsim j } }  \.  H_1(i,j)
		\ \leq  \  \prod_{ \substack{ 1 \leqslant i < j \leqslant d \\ i \eBsim j } } \. H_d(i,j)\..
		\end{align}
		First, note that the LHS of~\eqref{eq:upper-bound-mi-mj-rewritten} is bounded from above by
		\begin{align*}
		\prod_{a=1}^{d-1} \. R_a^{\frac{d(d-1)}{2}} \,   \prod_{ \substack{ 1 \leqslant i < j \leqslant d \\ i \eBsim j } }    H_1(i,j)    \ \leq_{\eqref{eq:Raij less than 1}}  \  \prod_{ \substack{ 1 \leqslant i < j \leqslant d \\ i \eBsim j } }  \left(  H_1(i,j) \prod_{a=1}^{d-1} R_a    \right).
		\end{align*}
		Hence it suffices to show that
		\begin{align}\label{eq:upper-bound-mi-mj-rewritten-2}
		\prod_{ \substack{ 1 \leqslant i < j \leqslant d \\ i \eBsim j } }  \left(  H_1(i,j) \prod_{a=1}^{d-1} R_a    \right) \ \leq \ \prod_{ \substack{ 1 \leqslant i < j \leqslant d \\ i \eBsim j } }  H_d(i,j).
		\end{align}
		First, for \ts $a \eBsim \sigma_a^{-1}(a)$, we have:
		\begin{equation*}
		\begin{split}
		\prod_{ \substack{ 1 \leqslant i < j \leqslant d \\ i \eBsim j } } \, R_a \ts H_a(i,j) \ \leq_{\eqref{eq:Raij less than 1}} \ \prod_{ \substack{ 1 \leqslant i < j \leqslant d \\ i \eBsim j } }  H_a(i,j)
		\ \leq_{\text{Lem~\ref{l:rear-ineq-a=b}}} \  \prod_{ \substack{ 1 \leqslant i < j \leqslant d \\ i \eBsim j } } H_{a+1}(i,j).
		\end{split}
		\end{equation*}
		Otherwise, for \ts $a \eBnsim \sigma_a^{-1}(a)$, we have:
		\begin{equation*}
		\begin{split}
		\prod_{ \substack{ 1 \leqslant i < j \leqslant d \\ i \eBsim j } } \, R_a \ts H_a(i,j)
		\ \leq_{\eqref{eq:rear-ineq-a-neq-b-cor}} \  \prod_{ \substack{ 1 \leqslant i < j \leqslant d \\ i \eBsim j } } \, H_{a+1}(i,j)\ts,
		\end{split}
		\end{equation*}
		so we have the same inequality in both cases.
		Now~\eqref{eq:upper-bound-mi-mj-rewritten-2} follows by induction on $a\in \{1,\ldots, d-1\}$.
		This completes the proof of the lemma.
	\end{proof}
	
	\smallskip
	
	\begin{proof}[Proof of Theorem~\ref{t:interval}]
		We have:
		\begin{align*}
		\frac{s_{\mu}(x_1,\ldots,x_d)}{x_1^{\mu_1} \ldots x_d^{\mu_d}} \ & \leq_{\text{Lem~\ref{l:Schur-pol-exp}}} \
		\sum_{\sigma \in S_d} \.
		\frac{x_1^{m_{\sigma(1)}}\ldots x_d^{m_{\sigma(d)}} }{x_1^{m_1}\ldots x_d^{m_d} }  \,
		\prod_{ \substack{ 1 \leqslant i < j \leqslant d \\ i \eBsim j } }
		\frac{|m_{\sigma(i)}-m_{\sigma(j)}|}{j-i}  \,
		\prod_{ \substack{ 1 \leqslant i < j \leqslant d \\ i \eBnsim j } } \.\frac{x_i}{x_i - x_j}\\
		& \le \   \sum_{\sigma \in S_d}
		\frac{x_1^{m_{\sigma(1)}}\ldots x_d^{m_{\sigma(d)}} }{x_1^{m_1}\ldots x_d^{m_d} }  \,
		\prod_{ \substack{ 1 \leqslant i < j \leqslant d \\ i \eBsim j } } \. \bigl|m_{\sigma(i)}-m_{\sigma(j)}\bigr| \,
		\prod_{ \substack{ 1 \leqslant i < j \leqslant d \\ i \eBnsim j } } \. \frac{x_i}{x_i - x_j}\\
		&    \le_{\text{Lem~\ref{l:upper bound for mi-mj}}} \      \sum_{\sigma \in S_d} \,
		\prod_{ \substack{ 1 \leqslant i < j \leqslant d \\ i \eBsim j } } \. \left( m_{i}-m_{j} \. + \. \frac{d(d-1)^2}{e}\. N(\ell,\CB) \right) \prod_{ \substack{ 1 \leqslant i < j \leqslant d \\ i \eBnsim j } } \frac{x_i}{x_i - x_j}  \\
		& \le \ C_d \.\cdot
		\prod_{ \substack{ 1 \leqslant i < j \leqslant d \\ i \eBsim j } }\. \bigl( m_{i}-m_{j} +N(\ell,\CB) \bigr) \.
		\prod_{ \substack{ 1 \leqslant i < j \leqslant d \\ i \eBnsim j } } \. \frac{x_i}{x_i - x_j}\.,
		\end{align*}
	where \ts $C_d:= d! \max\{1, \frac{d(d-1)^2}{e} \}$. This completes the proof.
	\end{proof}

	\bigskip
	
	\section{The case of thick Young diagrams}\label{s:straight-shaped Young posets}

	In this section we discuss the sorting probability for $\ep$-thick Young diagrams and
	present the proof of Theorem~\ref{t:thick}.

	\subsection{Using special interval decompositions}
	Fix $\ep >0$ and $\mu=(0,\ldots, 0)$, so $\la/\mu=\la$. Throughout the section we assume
	that $\la\vdash n$ and $\la$ is $\ep$-thick. This assumption implies that $(\lambda,\mu)$
	is $\ep$-admissible, since \ts $\lambda_i-\mu_i =  \lambda_i \geq \lambda_d \geq \ts \ep \ts n$.
	
	For the rest of this section, let $\CB$  be the interval decomposition of
	\ts $[d]=\{1,\ldots,d\}$ \ts that places \ts
	$i,j \in [d]$, \ts $i <j$, in the same block if and only if
	\begin{equation}\label{eq:const-CB-sqrt-n}
	\lambda_i-\lambda_j+j-i  \, \leq \, \sqrt{n}\ts.
	\end{equation}
	
	\smallskip
	
	\begin{lemma}\label{l:straight-cons-interval-theorem}
		Fix $d\geq 2$.  Let \ts $\lambda\vdash n$, \ts $\la, \gamma\in \pp_d$, \ts $\gamma \subseteq  \lambda$, and
		let $\CB$ as in~\eqref{eq:const-CB-sqrt-n}. Then:
		\[  \frac{f(\lambda/\gamma)}{F(\lambda/\gamma)}   \ \leq \ C_d \prod_{\substack{1 \leqslant i < j \leqslant d \\ i \Bsim j}} (\gamma_i-\gamma_j+j-i +   \sqrt{n}) \ \prod_{\substack{1 \leqslant i < j \leqslant d \\ i \Bnsim j}} \frac{\lambda_i+d-i}{\lambda_i-\lambda_j+j-i}\..\]
		for some absolute constant $C_d >0$.
	\end{lemma}
	
	\begin{proof}
		It follows from Corollary~\ref{c:interval}, by substituting $\mu$ with $\gamma$, that
		\[
		\frac{f(\lambda/\gamma)}{F(\lambda/\gamma)}   \, \leq \, C_d \. \prod_{\substack{1 \leqslant i < j \leqslant d \\ i \eBsim j}}  (\gamma_i-\gamma_j+j-i +  N(\ell,\CB)) \, \prod_{\substack{1 \leqslant i < j \leqslant d \\ i \eBnsim j}} \. \frac{\lambda_i+d-i}{\lambda_i-\lambda_j+j-i}\.,
		\]
		where \ts $N(\ell, \CB)$ \ts is as defined in \eqref{eq:N(lambda,B)}, and $C_d > 0$ is an absolute constant.
		Note that
		\begin{align*}
		N(\ell,\CB) \, & =  \, \max_{\substack{1 \leqslant i < j \leqslant d \\ i \eBnsim j}}
		\left \{ \frac{\lambda_i+d-i}{\lambda_i-\lambda_j+j-i}  \right \}
		\, \leq_{\eqref{eq:const-CB-sqrt-n}}  \ \frac{\lambda_1+d-1}{\sqrt{n}}
		\, \leq \, \frac{d\ts n}{\sqrt{n}} \, =  \, d \ts\sqrt{n} \ts.
		\end{align*}
		We conclude:
		\begin{align*}
		\frac{f(\lambda/\gamma)}{F(\lambda/\gamma)}   \ \leq & \ C_d \prod_{\substack{1 \leqslant i < j \leqslant d \\ i \eBsim j}} (\gamma_i-\gamma_j+j-i + d \sqrt{n} ) \ \prod_{\substack{1 \leqslant i < j \leqslant d \\ i \eBnsim j}} \frac{\lambda_i+d-i}{\lambda_i-\lambda_j+j-i}\\
		\ \leq & \  d^{\frac{d(d-1)}{2}} C_d \prod_{\substack{1 \leqslant i < j \leqslant d \\ i \eBsim j}} (\gamma_i-\gamma_j+j-i +  \sqrt{n} ) \ \prod_{\substack{1 \leqslant i < j \leqslant d \\ i \eBnsim j}} \frac{\lambda_i+d-i}{\lambda_i-\lambda_j+j-i},
		\end{align*}
		which proves the lemma.
	\end{proof}
	
	We now derive  upper bounds for each term in the right side of Lemma~\ref{l:straight-cons-interval-theorem}.
	We collect these  upper bounds  in the following two lemmas.
	
	\begin{lemma}[{\rm Same blocks estimate}]
		\label{l:straight-shaped i Bnsim j case}
		Fix $d\geq 2$.  Let \ts $\lambda\vdash n$, \ts $\la, \gamma\in \pp_d$, \ts $\gamma \subseteq  \lambda$, and
		let $\CB$ as in~\eqref{eq:const-CB-sqrt-n}. Then, for all \ts $1 \leq i <j \leq d$ \ts satisfying $i \Bnsim j$, we have:
		\[
		\frac{\gamma_i-\gamma_j+j-i}{n} \, \frac{\lambda_i+d-i}{\lambda_i-\lambda_j+j-i}  \, \leq \,   d \ts \bigl(|y_i-y_j| + 1\bigr)\.,
		\]
		where $y_i$ are defined in~\eqref{eq:definition yi}.
	\end{lemma}
	\begin{proof}
		We have
		\begin{align*}
		\frac{\gamma_i-\gamma_j+j-i}{n} \frac{\lambda_i+d-i}{\lambda_i-\lambda_j+j-i}   \ \leq &  \ \frac{\gamma_i-\gamma_j+j-i}{n} \frac{dn}{\lambda_i-\lambda_j+j-i}
		\ =  \ d   \frac{\gamma_i-\gamma_j+j-i}{\lambda_i-\lambda_j+j-i}.
		\end{align*}
		Note that
		\begin{equation}\label{eq:gammai-gammaj}
		\begin{split}
		\gamma_i-\gamma_j + j-i  \ =_{\eqref{eq:definition p}}  \ \sqrt{n}\. (y_i-y_j) \. + \. p \ts
		\left(\lambda_i-\lambda_j \right) + j-i \
		\leq   \, \sqrt{n}\. (y_i-y_j) +  \lambda_i-\lambda_j + j-i.
		\end{split}
		\end{equation}
		Since $i \eBnsim j$, we have \ts $\la_i-\la_j >\sqrt{n}$ by~\eqref{eq:const-CB-sqrt-n}. Therefore:
		\begin{align*}
		\frac{\gamma_i-\gamma_j+j-i}{\lambda_i-\lambda_j+j-i} \, \leq \, 1 \. + \. \frac{\sqrt{n}\, (y_i-y_j) }{\sqrt{n}}\..
		\end{align*}
		Combining the inequalities implies the result.
	\end{proof}
	
	\smallskip
	
	\begin{lemma}[{\rm Distinct blocks estimate}]
		\label{l:straight-shaped i Bsim j case}
		Fix $d\geq 2$.  Let \ts $\lambda\vdash n$, \ts $\la, \gamma\in \pp_d$, \ts $\gamma \subseteq  \lambda$, and
		let $\CB$ as in~\eqref{eq:const-CB-sqrt-n}. Then, for all \ts $1 \leq i <j \leq d$ \ts satisfying $i \Bsim j$,
		we have:
		\[\frac{\gamma_i-\gamma_j+j-i}{n} \bigl(\gamma_i-\gamma_j+j-i +   \sqrt{n}\bigr)   \, \leq \,
		2 \bigl(|y_i-y_j| + 1\bigr)^2\ts.
		\]
	\end{lemma}
	\begin{proof}
		By the same argument as in \eqref{eq:gammai-gammaj}, we have
		\[ \gamma_i-\gamma_j + j-i
		\ \leq   \ \sqrt{n}\, (y_i-y_j) + \lambda_i-\lambda_j + j-i\ts.
		\]
		Since $i \eBsim j$, it then follows that
		\[ \gamma_i-\gamma_j + j-i
		\, \leq  \, \sqrt{n}\, (y_i-y_j) \. + \, \sqrt{n}  \, \leq \, \sqrt{n}\.  (y_i-y_j+1).
		\]
		This then implies that
		\begin{align*}
		& \ \frac{\gamma_i-\gamma_j+j-i}{n} \bigl(\gamma_i-\gamma_j+j-i +   \sqrt{n}\bigr)   \, \leq \,
		\frac{\sqrt{n}\,  (y_i-y_j+1)}{n} \. \sqrt{n}\,  (y_i-y_j+2) \, \leq  \, 2\ts (|y_i-y_j|+1)^2\ts,
		\end{align*}
		which  completes the proof.
	\end{proof}

	\subsection{Lattice paths}
	The main ingredient in the proof of Theorem~\ref{t:thick} is the following lemma,
	a direct analogue for straight shapes of Lemma~\ref{l:pmf of Zt upper bound}.  Recall the definition of random integer paths \ts
	$\bZ=(Z_0,\ldots,Z_n)$ \ts in~$\S$\ref{ss:paths-setup}, and the definition of $y_i$ in \eqref{eq:definition yi}.
	
	\begin{lemma}\label{l:straight-exp-decay}
		Fix \ts $d\geq 2$ \ts and \ts $\ep >0$.
		Let $\lambda \in \pp_d$, such that $\la$ is $\ep$-thick.
		Then, for every $(\lambda,\gamma,\emp)\in \La(n,d,\ep)$, \ts
		$\ga\vdash k$, we have:
		\begin{equation*}
		\Pblm \big[Z_{k} = \gamma  \big] \, \leq \, C_{d,\ep}\, n^{-\frac{(d-1)}{2}} \. \exp \left[ - 2 \sum_{i=1}^d y_i^2
		\right] \ \prod_{1 \leqslant i < j \leqslant d } \. \bigl( (y_i-y_j)^2 + 1\bigr),
		\end{equation*}
		for some absolute constant  $C_{d,\ep}>0$.
	\end{lemma}

	\begin{proof}
		Following the proof of Lemma~\ref{l:pmf of Zt upper bound}, we have:
		\begin{equation}\label{eq:straight-shaped 1}
		\Pblm \big[Z_{k} = \gamma  \big]  \, =_{\eqref{eq:pmf-two-products}} \, \frac{f(\gamma) \. f(\lambda/\gamma)}{f(\lambda)}
		\ = \ f(\gamma) \. \frac{F(\lambda/\gamma)}{f(\lambda)} \, \frac{f(\lambda/\gamma)}{F(\lambda/\gamma)}.
		\end{equation}
		We give a separate bound for each of the three  terms in the RHS of~\eqref{eq:straight-shaped 1}.
		
		First note that,
		\begin{align*}
		f(\gamma) \ & =_{\eqref{eq:SYT-Frob}} \ \frac{k!}{\gamma_1! \ldots \gamma_d!} \prod_{1 \leqslant i < j \leqslant d} \frac{\gamma_i-\gamma_j+j-i}{\gamma_i+j-i}
		\leq_{\eqref{eq:separated-def}}  \ \frac{k!}{\gamma_1! \cdots \gamma_d!}  \prod_{1 \leqslant i < j \leqslant d} \frac{\gamma_i-\gamma_j+j-i}{\frac{\ep^3}{2} n} 
		\\
		& \leq \ \left(\frac{\ep^3}{2}\right) ^{-\frac{d(d-1)}{2}}  \frac{k!}{\gamma_1! \cdots \gamma_d!}  \prod_{1 \leqslant i < j \leqslant d} \frac{\gamma_i-\gamma_j+j-i}{n}\..
		\end{align*}
		Note also that
		\begin{align*}
		\frac{F(\lambda/\gamma)}{f(\lambda)} \ =_{\eqref{eq:F-def}, \.\eqref{eq:HLF}}  \ \frac{(n-k)!}{n!} \prod_{(i,j) \in \gamma} h_{\lambda}(i,j)
		\ \leq_{\text{\eqref{eq:separated-def}, \. Lem~\ref{l:mu hook length estimate}}}  \ \left(\frac{\ep^3}{2}\right) ^{-\frac{d(d-1)}{2}}    \frac{(n-k)!}{n!} \ \prod_{i=1}^d \. \frac{\lambda_i!}{(\lambda_i-\gamma_i)!}\,.
		\end{align*}
		Finally, by Lemma~\ref{l:straight-cons-interval-theorem} we have:
		\[\frac{f(\lambda/\gamma)}{F(\lambda/\gamma)}   \ \leq \ C_d \prod_{\substack{1 \leqslant i < j \leqslant d \\ i \Bsim j}} \. \bigl(\gamma_i-\gamma_j+j-i +   \sqrt{n}\bigr) \, \prod_{\substack{1 \leqslant i < j \leqslant d \\ i \Bnsim j}} \.\frac{\lambda_i+d-i}{\lambda_i-\lambda_j+j-i}\,,
		\]
		for some absolute constant $C_d>0$, where $\CB$ is defined in~\eqref{eq:const-CB-sqrt-n}.
		
		Substituting the above three estimates in~\eqref{eq:straight-shaped 1},
		we obtain:
		\begin{align*}
		& \Pblm \big[Z_{k} = \gamma  \big] \,
		\le \,  C_d \left(\frac{\ep^6}{4}\right)^{-\frac{d(d-1)}{2}}  \ \binom{\lambda_1}{\gamma_1}\, \cdots \, \binom{\lambda_d}{\gamma_d} \. \binom{n}{k}^{-1} \, \times \\ & \qquad \times \, \prod_{\substack{1 \leqslant i < j \leqslant d \\ i \eBsim j}} \frac{\gamma_i-\gamma_j+j-i}{n} \bigl(\gamma_i-\gamma_j+j-i +   \sqrt{n}\bigr) \  \prod_{\substack{1 \leqslant i < j \leqslant d \\ i \eBnsim j}} \frac{\gamma_i-\gamma_j+j-i}{n} \frac{\lambda_i+d-i}{\lambda_i-\lambda_j+j-i}\,.
		\end{align*}
		By Lemmas~\ref{l:straight-shaped i Bnsim j case} and~\ref{l:straight-shaped i Bsim j case}, the last two products are bounded by
		\begin{align*}
		\prod_{\substack{1 \leqslant i < j \leqslant d \\ i \eBsim j}} 2\ts \bigl(|y_i-y_j|+1\bigr)^2
		\ \prod_{\substack{1 \leqslant i < j \leqslant d \\ i \eBnsim j}} d\bigl(|y_i-y_j|+1\bigr) \,
		\leq  \, \bigl[(2\ts (d+2)\bigr]^{-\frac{d(d-1)}{2}} \, \prod_{{1 \leqslant i < j \leqslant d }} \bigl((y_i-y_j)^2+1\bigr).
		\end{align*}
		On the other hand, Lemma~\ref{l:binomial formula asymptotic} gives
		\[
		\binom{\lambda_1}{\gamma_1}\, \cdots \, \binom{\lambda_d}{\gamma_d} \. \binom{n}{k}^{-1}
		\, \leq  \,  C_{d,\ep} \, n^{-\frac{(d-1)}{2}} \
		\exp \left[ - 2 \sum_{i=1}^d y_i^2
		\right]\]
		for some absolute constant \ts $C_{d,\ep}>0$.
		Combining the last three inequalities, we conclude:
		\[   \Pblm \big[Z_{k} = \gamma  \big] \, \leq \,  C_{d,\ep}' \, n^{-\frac{(d-1)}{2}} \
		\exp \left[ - 2 \sum_{i=1}^d y_i^2
		\right]  \,  \prod_{{1 \leqslant i < j \leqslant d }} \. \bigl((y_i-y_j)^2+1\bigr)\,,
		\]
		for some absolute constant \ts $C_{d,\ep}'>0$.
	\end{proof}
	
	\subsection{Proof of Theorem~\ref{t:thick}}
	We follow the proof of Main Lemma~\ref{l:main} in~$\S$\ref{ss:technical-ML-proof}.
	Let \ts $a=\lfloor\frac{\lambda_1}{2}\rfloor$.  Then
	the first  condition in Lemma~\ref{l:pdf upper bound} is satisfied.
	Also note that the second condition in Lemma~\ref{l:pdf upper bound}
	is satisfied as a consequence of Lemma~\ref{l:straight-exp-decay}.  We conclude:
	\[ \delta(P_{\lambda/\mu}) \ \leq_{\text{Lem~\ref{l:quantitative bound for Linial}}}
	\ 2 \ts \vp(a) \ \leq_{\text{Lem~\ref{l:pdf upper bound}, Lem~\ref{l:straight-exp-decay} }} \
	2\ts C_{d,\ep}' \. \frac{C_{d,\ep}^3+1}{\sqrt{n}}\,,
	\]
	for some absolute constants \ts $C_{d,\ep}, C_{d,\ep}'>0$.
	This completes the proof. \ $\sq$
	
	\bigskip
	
	\section{The case of smooth skew Young diagrams}\label{s:main-thm}
	
In this section we prove  Theorem~\ref{t:main}. 

	\subsection{Proof of Lemma~\ref{l:asy-smooth}}
	Let $\la/\mu \in \pp_d$, $\la/\mu\vdash n$, and suppose $\la$ is $\ep$-smooth.
	Then we have:
	\begin{equation}\label{eq:distinct row 0}
	\frac{\lambda_i+d-i}{\lambda_i-\lambda_j+j-i}  \, \leq \, \frac{n+d-1}{\ep \ts n}  \, \leq \, \frac{d}{\ep}\,,
	\quad \text{for all} \quad  1 \leq i <j \leq d.
	\end{equation}
	Therefore,
	\begin{equation}\label{eq:eG-tech-smooth}
	1 \ \leq \ \min \left\{\mu_i-\mu_j+j-i, \. \frac{\lambda_i+d-i}{\lambda_i-\lambda_j+j-i}   \right\}
	\, \leq \, \frac{d}{\ep}\,.
	\end{equation}
	By the definition of $\eG(\lambda/\mu)$, see \eqref{eq:G-def}, we get:
	\begin{equation}\label{eq:eG-const-smooth}
	1 \, \leq  \, \eG(\lambda/\mu)  \, \leq \,  \left(\frac{d}{\ep}\right)^{\frac{d(d-1)}{2}}\,.
	\end{equation}
	i.e., function \ts $\eG(\lambda/\mu)$ \ts is of a constant order.
	Therefore, the result follows from the following bounds:
	\begin{equation}\label{eq:distinct row 1}
	1   \ \leq \ \frac{f(\lambda/\mu)}{F(\lambda/\mu)} \, \leq  \, d! \. \left(\frac{d}{\ep}\right)^{\frac{d(d-1)}{2}}\,.
	\end{equation}
	The lower bound in \eqref{eq:distinct row 1} follows from Theorem~\ref{t:NHLF-asy}.
	For the upper bound in \ref{eq:distinct row 1}, we use Corollary~\ref{c:interval} applied to the interval
	decomposition \ts $\CB:=\{B_1,\ldots,B_d\}$, where $B_i=\{i\}$.  In this case Corollary~\ref{c:interval} gives:
	\begin{align*}
	\frac{f(\lambda/\mu)}{F(\lambda/\mu)}  \ \leq \ d! \prod_{1 \leqslant i < j \leqslant d}  \frac{\lambda_i+d-i}{\lambda_i-\lambda_j+j-i}
	\, \leq  \, d! \. \left(\frac{d}{\ep}\right)^{\frac{d(d-1)}{2}}\,, 
	\end{align*}
	which proves the upper bound in~\eqref{eq:distinct row 1}. \ $\sq$
	
	\smallskip
	
	\subsection{Proof of Theorem~\ref{t:main}}
	By Lemma~\ref{l:main}, it suffices to check that for every \ts
	$(\lambda,\gamma,\mu)\in \Om(n,d,\ve)$, we have
	\begin{equation}\label{eq:main-proof-three}
	\frac{f(\gamma/\mu)}{F(\gamma/\mu)}   \, \leq  \,  {C_{d,\ep}} \,  \eG(\gamma/\mu)\,, \quad
	\frac{f(\lambda/\gamma)}{F(\lambda/\gamma)}   \, \leq \, {C_{d,\ep}} \, \eG(\lambda/\gamma)
	\quad \text{ and }  \quad  \frac{f(\lambda/\mu)}{F(\lambda/\mu)}   \, \geq \,
	\frac{1}{C_{d,\ep}} \, \eG(\lambda/\mu)\.,
	\end{equation}
	for some absolute constant $C_{d,\ep} >0$.
	Note that the last two inequalities follow immediately from Lemma~\ref{l:asy-smooth}.
	We now prove that the first inequality holds.
	
	By the progressive assumption on $(\lambda,\gamma,\mu)\in \La(n,d,\ep)$, we have:
	\begin{equation*}
	\gamma_i - \gamma_{i+1} \, \geq_{\eqref{eq:progressive-def}} \ p \. (\lambda_i-\lambda_{i+1}) \. + \. (1-p)(\mu_i-\mu_{i+1})
	\.  -  \. 2\ts n^{\frac{3}{4}} \, \geq \, p\. (\lambda_i-\lambda_{i+1}) \.  -  \. 2\ts n^{\frac{3}{4}} \,
	\geq \, p\.\ep \. |\lambda| \.  -  \. 2\ts n^{\frac{3}{4}}\ts,
	\end{equation*}
	for every \ts $1\le i \le d-1$. Similarly, by the $\ep$-separation assumption on \ts
	$(\lambda,\gamma,\mu)$, we have:
	\begin{equation*}
	p \, = \, \frac{|\gamma|-|\mu|}{n}
	\, =  \,   \sum_{i=1}^d \frac{\gamma_i-\mu_i}{n} \, \geq_{\eqref{eq:separated-def}} \
	\left(\frac{d\ep^3}{2}\right)  \frac{|\lambda|}{n} \, \geq \,  \frac{d\ts\ep^3}{2}\,.
	\end{equation*}
	Thus, for sufficiently large \ts $n$, we have:
	\begin{align}\label{eq:distinct row 6}
	\gamma_i - \gamma_{i+1} \, \geq \, \frac{d\ts\ep^4}{2} |\lambda| \.  -  \. 2 \ts n^{\frac{3}{4}}
	\, \geq \, \frac{d\ts\ep^4}{4} \. |\lambda| \, \geq \, \frac{d\ts\ep^4}{4} \. |\gamma|\ts.
	\end{align}
	By the same argument as above, for
	sufficiently large~$n$, we have:
	\begin{align}\label{eq:distinct row 7}
	\gamma_d \, \geq \, \frac{d\ts\ep^4}{4} \. |\gamma|\ts.
	\end{align}
	Conditions~\eqref{eq:distinct row 6} and~\eqref{eq:distinct row 7} imply that \ts
	$\ga/\mu$ \ts is \. $(d\ts\ep^4/4)$-smooth, for $n$ large enough.
	Applying Lemma~\ref{l:asy-smooth}, we obtain the first inequality
	in~\eqref{eq:main-proof-three}. This completes the proof. \ $\sq$

	\bigskip

	\section{The case of TVK skew shapes}\label{s:TVK}
	
	In this section we give upper and lower bounds for the number of standard
	Young tableaux corresponding to TVK pairs.  We then prove Lemma~\ref{l:asy-TVK}
	and Theorem~\ref{t:TVK-skew}.
	
	\smallskip
	
	\subsection{Conditions for interval decomposition} \label{ss:TVK-types}
	We define three types of conditions for interval decomposition \ts
	$\CB=(B_1,\ldots, B_r)$ of~$[d]$.  These conditions will be used
	in combinations, to cover all possible TVK pairs.  Formally, consider:
	\begin{equation}\label{eq:lambda big gap}
	\lambda_i-\lambda_j  \, \geq \, \ep \ts |\lambda| \quad
	\text{for all} \quad i \Bnsim j\,, \ \ 1 \leqslant i < j \leqslant d\ts,
	\end{equation}
	\begin{equation}\label{eq:lambda small gap}
	\lambda_i-\lambda_j  \ \leq \ 1 \quad \text{for all} \quad i \Bsim j\,,
	\ \ 1 \leqslant i < j \leqslant d\ts,
	\end{equation}
	\begin{equation}\label{eq:mu small gap}
	\mu_i-\mu_j  \ \leq \ 1 \quad \text{for all} \quad i \Bsim j\,,  \ \ 1 \leqslant i < j \leqslant d\ts.
	\end{equation}
	The motivation behind these conditions for TVK $(\al,\be)$-shapes will become
	apparent later in this section.  For now, we treat them as abstract
	constraints on the interval decompositions,
	
	\smallskip
	
	\subsection{Upper bounds}\label{ss:TVK-upper}
	We start with estimating each term in the definition of $\eG(\lambda/\mu)$,
	see~\eqref{eq:G-def}, and we collect these estimates in the next three lemmas.
	
	\begin{lemma}\label{l:condition 1 min estimate}
		Fix \ts $d\geq 2$ \ts and \ts $\ep >0$. Let $\lambda, \mu\in\pp_d$, such that \ts
		$\mu \subseteq  \lambda$. Suppose \eqref{eq:lambda big gap}
		holds for the interval decomposition $\CB$ of~$[d]$.
		Then:
		\begin{equation}\label{eq:condition 1 min}
		1  \,\leq \, \frac{\lambda_i+d-i}{\lambda_i-\lambda_j+j-i} \, \leq \, \frac{d}{\ep}\,,
		\quad \text{for all \ \ $i \Bnsim j$, \ \. $1 \leqslant i < j \leqslant d$. }
		\end{equation}
		In particular, we have:
		\begin{equation}\label{eq:condition 1 min estimate}
		\frac{\ep}{d} \, \frac{\lambda_i+d-i}{\lambda_i-\lambda_j+j-i}  \, \leq \,
		\min \left\{\mu_i-\mu_j+j-i, \frac{\lambda_i+d-i}{\lambda_i-\lambda_j+j-i}  \right \}
		\, \leq \, \frac{\lambda_i+d-i}{\lambda_i-\lambda_j+j-i}\,.
		\end{equation}
	\end{lemma}
	\begin{proof}
		The lower bound in \eqref{eq:condition 1 min} follows from \eqref{eq:lambdai-lambdaj is greater than 1}.
		The upper bound in \eqref{eq:condition 1 min}, follows verbatim~\eqref{eq:distinct row 0}.
		For the lower bound in \eqref{eq:condition 1 min estimate},
		we have
		\begin{align*}
		\min \left\{\mu_i-\mu_j+j-i, \frac{\lambda_i+d-i}{\lambda_i-\lambda_j+j-i}  \right \}  \ \geq_{\eqref{eq:lambdai-lambdaj is greater than 1}, \. \eqref{eq:mui-muj is greater than 1}} \ 1  \, \geq_{\eqref{eq:condition 1 min}}  \, \left(\frac{\ep}{d}\right) \ts   \frac{\lambda_i+d-i}{\lambda_i-\lambda_j+j-i}\,,
		\end{align*}
		as desired.
	\end{proof}

	\begin{lemma}\label{l:condition 2 min estimate}
		Fix \ts $d\geq 2$ \ts and \ts $\ep >0$. Let $\lambda, \mu\in\pp_d$, such that \ts
		$\mu \subseteq  \lambda$. Suppose \eqref{eq:lambda small gap}
		holds for the interval decomposition $\CB$ of~$[d]$.
		Then:
		\begin{equation}\label{eq:cond-estimates-both}
		\frac1d\.(\mu_i-\mu_j+j-i)  \, \leq \, \min \left\{\mu_i-\mu_j+j-i,
		\.\frac{\lambda_i+d-i}{\lambda_i-\lambda_j+j-i}  \right \}   \, \leq \, \mu_i-\mu_j+j-i\ts,
		\end{equation}
		for all \ts $i \eBsim j$, \ts $1 \leqslant i < j \leqslant d$.
	\end{lemma}
	\begin{proof}
		The upper bound is straightforward.  For the lower bound, it follows from \eqref{eq:lambda small gap}, that
		\begin{equation*}
		\begin{split}
		\frac{\lambda_i+d-i}{\lambda_i-\lambda_j+j-i} \, \geq \, \frac{\lambda_i+d-i}{1+j-i}
		\, \geq \, \frac{\lambda_i+d-i}{d} \, \geq \,  \frac{\mu_i-\mu_j+j-i}{d}\..
		\end{split}
		\end{equation*}
		It then follows from  the equation above that
		\begin{align*}
		\min \left\{\mu_i-\mu_j+j-i, \. \frac{\lambda_i+d-i}{\lambda_i-\lambda_j+j-i}  \right \}
		\, \geq \, \min \left\{\mu_i-\mu_j+j-i,\frac{\mu_i-\mu_j+j-i}{d}  \right \}  \, = \,
		\frac{\mu_i-\mu_j+j-i}{d}\,,
		\end{align*}
		as desired.
	\end{proof}

	\begin{lemma}\label{l:condition 3 min estimate}
		Fix \ts $d\geq 2$ \ts and \ts $\ep >0$. Let $\lambda, \mu\in\pp_d$, such that \ts
		$\mu \subseteq  \lambda$. Suppose \eqref{eq:mu small gap}
		holds for the interval decomposition $\CB$ of~$[d]$.
		Then~\eqref{eq:cond-estimates-both} holds
		for all \ts $i \eBsim j$, \ts $1 \leqslant i < j \leqslant d$.
	\end{lemma}
	
	\begin{proof}
		The upper bound is straightforward.  For the lower bound, it follows from \eqref{eq:mu small gap} that
		\begin{equation}\label{eq:condition 3 min}
		\begin{split}
		\mu_i-\mu_j+j-i \, \leq \,  1 +j-i \, \leq \, d.
		\end{split}
		\end{equation}
		Therefore,
		\begin{align*}
		\min \left\{\mu_i-\mu_j+j-i, \.\frac{\lambda_i+d-i}{\lambda_i-\lambda_j+j-i}  \right \}  \,
		\geq_\text{\eqref{eq:lambdai-lambdaj is greater than 1},\.\eqref{eq:mui-muj is greater than 1}}
		\ 1 \ \geq_{\eqref{eq:condition 3 min}} \ \frac1d \.\bigl(\mu_i-\mu_j+j-i\bigr),\end{align*}
		as desired.
	\end{proof}

	We now combine these three lemmas to give an estimate for the quantity \ts $\eG(\lambda/\mu)$ \ts
	if \eqref{eq:lambda big gap} holds and either
	 \eqref{eq:lambda small gap} or \eqref{eq:mu small gap} holds.
	Denote
	\begin{equation}\label{eq:HB(lambda/mu)}
	K_\CB(\lambda/\mu) \, := \, \prod_{\substack{1 \leqslant i < j \leqslant d \\ i \eBsim j}}
	\. (\mu_i-\mu_j+j-i) \, \prod_{\substack{1 \leqslant i < j \leqslant d \\ i \eBnsim j}}
	\. \frac{\lambda_i+d-i}{\lambda_i-\lambda_j+j-i}\,.
	\end{equation}
	
	\begin{lemma}\label{l:G exact estimate}
		Fix \ts $d\geq 2$ \ts and \ts $\ep >0$. Let $\lambda, \mu\in\pp_d$, such that \ts
		$\mu \subseteq  \lambda$. Suppose \eqref{eq:lambda big gap} and \eqref{eq:lambda small gap}
		hold for the interval decomposition $\CB$ of~$[d]$.  Then:
		\[
		\left(\frac{\ep}{d^2}\right)^{\frac{d(d-1)}{2}} K_\CB(\lambda/\mu)
		\, \leq \,    \eG(\lambda/\mu) \,  \leq \,  K_\CB(\lambda/\mu)\ts.
		\]
		The same conclusion holds if condition \eqref{eq:lambda small gap} is replaced with
		\eqref{eq:mu small gap}.
	\end{lemma}
	
	\begin{proof}
		By definition of \ts $K_\CB(\lambda/\mu)$, we have:
		\begin{align*}
		\frac{\eG(\lambda/\mu)}{K_\CB(\lambda/\mu)} \, = \,  \prod_{\substack{1 \leqslant i < j \leqslant d \\ i \eBnsim j}} \frac{  \min \left\{\mu_i-\mu_j+j-i, \frac{\lambda_i+d-i}{\lambda_i-\lambda_j+j-i}  \right \}}{\frac{\lambda_i+d-i}{\lambda_i-\lambda_j+j-i} }
		\prod_{\substack{1 \leqslant i < j \leqslant d \\ i \eBsim j}} \frac{  \min \left\{\mu_i-\mu_j+j-i, \frac{\lambda_i+d-i}{\lambda_i-\lambda_j+j-i}  \right \}}{\mu_i-\mu_j+j-i } \, \le \, 1.
		\end{align*}
		For the lower bound, note that each term in the first product is bounded
		from below by~$\ep/d$, by Lemma~\ref{l:condition 1 min estimate} and condition~\eqref{eq:lambda big gap}.
		Also note that each term in the second product is bounded from below by $1/d$,
		by Lemma~\ref{l:condition 2 min estimate} when~\eqref{eq:lambda small gap} holds,
		or by Lemma~\ref{l:condition 3 min estimate} when \eqref{eq:mu small gap}  holds.
		This implies the result.
	\end{proof}
	
	The main result of this subsection is the following upper bound.
	
	\begin{lemma}\label{l:f -- TVK upper bound}
		Fix \ts $d\geq 2$ \ts and \ts $\ep >0$. Let $\lambda, \mu\in\pp_d$, such that \ts
		$\mu \subseteq  \lambda$. Suppose \eqref{eq:lambda big gap} and \eqref{eq:lambda small gap}
		hold for the interval decomposition $\CB$ of~$[d]$.  Then:
		\[
		\frac{f(\lambda/\mu)}{F(\lambda/\mu)} \, \leq \, C_{d,\ep} \eG(\lambda,\mu)\ts.
		\]
		where \ts $C_{d,\ep}>0$ \ts is an absolute constant.  The same conclusion holds
		if condition \eqref{eq:lambda small gap} is replaced with~\eqref{eq:mu small gap}.
	\end{lemma}
	
	\begin{proof} For the first part, it follows from Lemma~\ref{l:G exact estimate}
		that $\eG(\lambda/\mu)$ is equal to $K_B(\lambda/\mu)$ up to a multiplicative constant.
		Therefore, it suffices to show that
		\begin{equation*}
		\frac{f(\lambda/\mu)}{F(\lambda/\mu)} \, \leq \, C_{d,\ep} \. K_\CB(\lambda/\mu)\ts,
		\end{equation*}
		for some absolute constant $C_{d,\ep}>0$.
		Let $N(\ell,B)$ be as in \eqref{eq:N(lambda,B)}.
		Then
		\begin{align*}
		N(\ell,\CB) \, = \, \max_{\substack{1 \leqslant i < j \leqslant d \\ i \eBnsim j}}
		\. \frac{\lambda_i+d-i}{\lambda_i-\lambda_j+j-i}  \,
		\leq_\text{\eqref{eq:lambda big gap},\.\eqref{eq:condition 1 min}} \, \frac{d}{\ep}\,.
		\end{align*}
		Substituting this into Corollary~\ref{c:interval}, we get
		\begin{align*}
		\frac{f(\lambda/\mu)}{F(\lambda/\mu)}   \ & \leq  \, C_d \prod_{\substack{1 \leqslant i < j \leqslant d \\ i \eBsim j}}
		\left(\mu_i-\mu_j+j-i +  \. \frac{d}{\ep} \right)  \, \prod_{\substack{1 \leqslant i < j \leqslant d \\ i \eBnsim j}} \frac{\lambda_i+d-i}{\lambda_i-\lambda_j+j-i} \\
		&  \leq   \, C_d \prod_{\substack{1 \leqslant i < j \leqslant d \\ i \eBsim j}} \left(1+\frac{d}{\ep}\right)
		\bigl(\mu_i-\mu_j+j-i\bigr)  \, \prod_{\substack{1 \leqslant i < j \leqslant d \\ i \eBnsim j}}
		\frac{\lambda_i+d-i}{\lambda_i-\lambda_j+j-i} \\
		& \leq   \, C_d \left(1+\frac{d}{\ep}\right)^{\frac{d(d-1)}{2}}
		\prod_{\substack{1 \leqslant i < j \leqslant d \\ i \eBsim j}} \bigl( \mu_i-\mu_j+j-i \bigr)  \, \prod_{\substack{1 \leqslant i < j \leqslant d \\ i \eBnsim j}} \frac{\lambda_i+d-i}{\lambda_i-\lambda_j+j-i}\\
		& \leq_{\eqref{eq:HB(lambda/mu)}} \  C_{d,\ep} \. K_\CB(\lambda/\mu)\ts,
		\end{align*}
		for some absolute constants \ts $C_{d}$, $C_{d,\ep} >0$. This finishes the
		proof of the first part.  The second part follows verbatim; we omit
		the details.
	\end{proof}
	
	\smallskip
	
	\subsection{Lower bounds}
	Our first ingredient is the following estimate on the hook-lengths.
	
	\smallskip
	
	\begin{lemma}\label{l:hook cells lower bound}
		Fix \ts $d\geq 2$ \ts and \ts $\ep >0$.
		Let $(\lambda, \mu)\in \La(n,d,\ep)$.
		Let $\CB$ be an interval decomposition of $[n]$ such that~\eqref{eq:lambda small gap} holds.
		Then, for all \ts $(i,j)\in\mu$, and all \ts $k \geq 0$ \ts such that \ts $i \Bsim (i+k)$, we have:
		\[
		\frac{h_\lambda(i+k,j+k)}{h_\lambda(i,j)}  \, \geq \, 1 \. - \. \frac{2\ts d}{\ep \ts |\lambda|}\..
		\]
	\end{lemma}
	
	\begin{proof}
		By the definition~\eqref{eq:hook-def} of the hook lengths, we have:
		\begin{equation*}
		\begin{split}
		\frac{h_{\la}(i+k,j+k)}{h_\la(i,j)} \,  & \geq  \, \frac{\la_{i+k} - j-k }{\la_i -i+ d-j+1}
		\, \geq_{[\text{since} \ i\ts\sim\ts(i+k)]} \, \frac{\la_{i} -1 - j-k }{\la_i -i+ d-j+1}  \, =  \, 1 \. - \. \frac{d+k+2-i}{\la_i - i +d -j +1}\\
		& \geq  \, 1 \. - \. \frac{2d}{\la_i - j + (d-i)+1} \, \geq_{[\text{since} \ (i,j) \ts\in\ts \mu]} \ 1 \. - \.
		\frac{2d}{\la_i - \mu_i } \, \geq_{\eqref{eq:separated-def}} \, 1 \ts - \. \frac{2d}{\ep |\la|}\,,
		\end{split}
		\end{equation*}
		as desired.
	\end{proof}
	
	\smallskip
	
	We apply Lemma~\ref{l:hook cells lower bound} to get a lower bound
	for the product of hooks of a flagged tableau,
	see~\eqref{eq:definition flagged tableau}.
	Let $\CB$ be an interval decomposition of~$[d]$.
	Denote by
	\[
	\Dc_\CB \. := \. \bigl\{ T \in \Ec(\lambda/\mu) \ \mid \  i \Bsim T(i,j)  \,
	\text{ for all } \, (i,j) \in \mu \bigr\},
	\]
	the set of flagged tableaux of $\lambda/\mu$,
	for which the entries for each row $i$ are drawn from the block  of  $\CB$ that contains~$i$.
	
	\begin{lemma}\label{l:hook products lower bound}
		Fix \ts $d\geq 2$ and \ts $\ep >0$.  Let \ts $(\lambda, \mu)\in \La(n,d,\ep)$,
		and let $\CB$ be an interval decomposition of~$[d]$, such that~\eqref{eq:lambda small gap} holds.
		Then, for all $T \in \Dc_B$, we have:
		\[
		\prod_{(i,j) \in \mu} \. \frac{h_{\lambda}\bigl(T(i,j), \ts j+T(i,j)-i\bigr)}{h_{\lambda}(i,j)}
		\, \geq  \,  C_{d,\ep}\,,
		\]
		for some absolute constant \ts $C_{d,\ep}>0$.
	\end{lemma}
	
	\begin{proof}
		We have:
		\begin{align*}
		& \prod_{(i,j) \in \mu} \frac{h_{\lambda}\bigl(T(i,j), \ts j+T(i,j)-i\bigr)}{h_{\lambda}(i,j)} \
		\ge_{\text{Lem~\ref{l:hook cells lower bound}}} \  \left( 1-\frac{2\ts d}{\ep \ts|\lambda|}\right)^{|\mu|}
		\ \geq \ \left( 1-\frac{2\ts d}{\ep\ts |\lambda|} \right)^{|\lambda|} \, \geq  \,
		\left(\frac{1}{3}\right)^{2d/\ep}\,, 
		\end{align*}
		for sufficiently large \ts $|\lambda|$.  This implies the result.
	\end{proof}
	
	Our second ingredient is the following lower bound on the cardinality of $\Dc_B$.

	\begin{lemma}\label{l:Dc lower bound}
		Fix \ts $d\geq 2$ \ts and \ts $\ep >0$.
		Let $(\lambda, \mu)\in \La(n,d,\ep)$.
		Let $\CB$ be an interval decomposition of $[n]$ such that~\eqref{eq:lambda small gap} holds.
		Then there exists an absolute constant $C_{d,\ep} >0$ such that
		\[
		\bigl|\Dc_\CB\bigr| \, \geq \, C_{d,\ep} \. \prod_{\substack{1 \leqslant i < j \leqslant d \\ i \Bsim j }} ({\mu_i-\mu_j+j-i}).\]
	\end{lemma}
	
	\begin{proof}
		Let \ts $\Dc_{\CB}':=\Dc'_{\CB}(\mu)$ \ts be the set of semistandard
		Young tableau of shape $\mu$  given by
		\[
		\Dc_{\CB}' \. := \. \bigl\{ T \in \text{SSYT}(\mu) \ \mid \  i \Bsim T(i,j) \, \text{ for all } \, (i,j) \in \mu \bigr\}.
		\]
		Note that \ts $\Dc_\CB = \Dc'_{\CB} \cap \Ec(\lambda/\mu)$.  We will estimate \ts $|\Dc_\CB|$ via $|\Dc'_{\CB}|$.
		
		Recall the definition~\eqref{eq:def-CB} of interval decompositions.
		For each \ts $k\in \{1,\ldots,r\}$, denote by \ts $\mu^{(k)}$ \ts
		the partition  obtained from $\mu$ by restricting to rows indexed by $B_k$:
		\begin{equation}\label{eq:def-mu-BC}
		\mu^{(k)} \, = \, \bigl(\mu^{(k)}_1, \mu^{(k)}_2, \ts \ldots \ts , \mu^{(k)}_{b_k-b_{k-1}}\bigr) \, := \,
		\bigl(\mu_{b_{k-1}+1}, \mu_{b_{k-1}+2}, \ts\ldots\ts , \mu_{b_k}\bigr).
		\end{equation}
		In this notation,
		\begin{equation}\label{eq:Dc lower bound 1}
		\Dc_{\CB}' \. = \. \bigl\{ T \in \text{SSYT}(\mu) \ \mid \  b_{k-1} \. < \.  T(i,j)
		\. \leq \. b_k \, \text{ for all } \, (i,j) \in \mu, \, i \in B_k \bigr\}.
		\end{equation}
		Therefore, that the following map is a bijection:
		\begin{equation}\label{eq:Dc lower bound 2}
		\aligned
		& \psi: \. \Dc_{\CB}' \, \to \, \SSYT\bigl(\mu^{(1)}\bigr) \times \ldots \times \SSYT\bigl(\mu^{(r)}\bigr), \quad
		\psi(T) \. :=  \. \bigl(T^{(1)}, \ldots, T^{(r)}\bigr), \\
		& \qquad \text{where} \ \ \, T^{(k)}(i,j) \. = \. T(i+b_{k-1},j) -b_{k-1} \quad \text{for all} \ \ \, (i,j) \in \mu^{(k)}\..
		\endaligned
		\end{equation}
		In other words, the semistandard Young tableaux $T^{(k)}$ is obtained
		by restricting~$T$ to rows indexed by~$B_k$ and normalizing the
		smallest entries to start from~$1$.
		It now follows from \eqref{eq:Dc lower bound 2} and~\eqref{eq:HCF}, that
		\begin{equation}\label{eq:Dc lower bound 3}
		\begin{split}
		|\Dc'_\CB| \ & =  \ \prod_{k=1}^r  \, \prod_{\substack{1 \leqslant i < j \leqslant d \\ i, j \in B_k }} \,
		\frac{\mu_i-\mu_j+j-i}{j-i} \ \geq  \
		\prod_{k=1}^r \, \prod_{\substack{1 \leqslant i < j \leqslant d \\ i, j \in B_k }} \, \frac{\mu_i-\mu_j+j-i}{d-1}\\
		\  & \geq  \  (d-1)^{-\frac{d(d-1)}{2}} \.  \prod_{\substack{1 \leqslant i < j \leqslant d \\ i \Bsim j }} \. \bigl({\mu_i-\mu_j+j-i}\bigr).
		\end{split}
		\end{equation}
		
		We claim that \ts $\Dc_\CB = \Dc_\CB'$ \ts for sufficiently large~$|\lambda|$.
		It suffices to show that each \ts $T\in\Dc_\CB'$ \ts is a flagged tableau
		of~$\lambda/\mu$, for sufficiently large $|\lambda|$.
		Let \ts $(i,j) \in \mu$, and let $k$ be the index such that $B_k$  is the block of $\CB$ that contains~$i$.
		We  have:
		\begin{align*}
		j+ T(u)-1   \ & \leq \ \mu_i +d-1
		\ = \  \lambda_i- (\lambda_i-\mu_i) + d-1 \ \leq_{\eqref{eq:separated-def}} \ \lambda_i - \ep \ts |\lambda| +(d-i) \\
		& \leq_\text{\eqref{eq:lambda small gap},\. $i \Bsim T(u)$} \ (\lambda_{T(u)} + 1) -\ep\ts |\lambda| +(d-i) \
		\leq \ \lambda_{T(u)}\.,
		\end{align*}
		for sufficiently large $\lambda_{T(u)}$.  This proves the claim.
		
		By~\eqref{eq:separated-def}, we have:
		\begin{equation}\label{eq:Dc-lower-sep}
		\lambda_1 \ \geq \ \ldots \ \geq \ \lambda_d \ \geq \ \lambda_d-\mu_d \ \geq \ \ep \ts |\lambda|,
		\end{equation}
		so the claim above assumes only that \ts $|\la|$ \ts is large enough.  We conclude:
		\begin{equation}\label{eq:Dc lower bound 3.5}
		|\Dc_\CB| \ =_{\eqref{eq:Dc-lower-sep}} \ |\Dc_\CB'| \ \geq_{\eqref{eq:Dc lower bound 3}} \
		(d-1)^{-\frac{d(d-1)}{2}}  \. \prod_{\substack{1 \leqslant i < j \leqslant d \\ i \Bsim j }} \. \bigl({\mu_i-\mu_j+j-i}\bigr),
		\end{equation}
		for all \ts   $|\lambda|$ sufficiently large.
		This completes the proof.
	\end{proof}
	
	The main result of this subsection is the following lower bound for $f(\lambda/\mu)$.
	
	\begin{lemma}\label{l:f -- TVK lower bound}
		Fix \ts $d\geq 2$ \ts and \ts $\ep >0$.
		Let $(\lambda, \mu)\in \La(n,d,\ep)$.
		Let $\CB$ be an interval decomposition of $[n]$ such that~\eqref{eq:lambda big gap}
		and~\eqref{eq:lambda small gap}  hold.
		Then there exists an absolute constant $C_{d,\ep}>0$ such that
		\[
		\frac{f(\lambda/\mu)}{F(\lambda/\mu)} \,\. \geq \,\. C_{d,\ep} \, \eG(\lambda/\mu)\ts.
		\]
	\end{lemma}

	\begin{proof} We have:
		\begin{align*}
		\frac{f(\lambda/\mu)}{F(\lambda/\mu) } \, &=_{\text{Thm~\ref{t:NHLF}}}  \  \sum_{T \in \Ec(\lambda/\mu)} \, \prod_{(i,j) \in \mu} \frac{h_{\lambda}(T(i,j), j+T(i,j)-i)}{h_{\lambda}(i,j)} \, \geq  \,  \sum_{T \in \Dc_B} \, \prod_{(i,j) \in \mu} \frac{h_{\lambda}(T(i,j), j+T(i,j)-i)}{h_{\lambda}(i,j)} \\
		& \geq_{\text{Lem~\ref{l:hook products lower bound}}}  \ \sum_{T \in \Dc_B} \. C_{d,\ep} \, = \,
		C_{d,\ep} \ts |\Dc_B| \,  \geq_{\text{Lem~\ref{l:Dc lower bound}}} \ C_{d,\ep}
		\prod_{\substack{1 \leqslant i < j \leqslant d \\ i \eBsim j }} \. ({\mu_i-\mu_j+j-i})\..
		\end{align*}
		This implies that
		\begin{align*}
		\frac{f(\lambda/\mu)}{F(\lambda/\mu)} \ &\geq_{\eqref{eq:condition 1 min}}  \ C_{d,\ep} \. \left(\frac{\ep}{d}\right)^{\frac{d(d-1)}{2}} \prod_{\substack{1 \leqslant i < j \leqslant d \\ i \eBsim j }} ({\mu_i-\mu_j+j-i})  \prod_{\substack{1 \leqslant i < j \leqslant d \\ i \eBnsim j }}  \frac{\lambda_i+d-i}{\lambda_i-\lambda_j+j-i} \\
		& =_{\eqref{eq:HB(lambda/mu)}}  \ C_{d,\ep} \. \left(\frac{\ep}{d}\right)^{\frac{d(d-1)}{2}} K_{\CB}(\lambda/\mu) \
		\geq_{\text{Lem~\ref{l:G exact estimate}}}  \ C_{d,\ep} \. \left(\frac{\ep}{d}\right)^{\frac{d(d-1)}{2}} \eG(\lambda/\mu)\.,
		\end{align*}
		as desired.
	\end{proof}

	\subsection{Proof of Lemma~\ref{l:asy-TVK}}
	Recall the definition of a Thoma pair~$(\al,\be)$ in~$\S$\ref{ss:intro-main}.
	Let \ts $\ep:=\ep(\alpha,\beta)$ \ts be given by
	\begin{equation}\label{eq:definition of epsilon TVK}
	\ep := \frac{1}{2(\alpha_1+\ldots+\alpha_d)} \, \min \Big\{  \min_{1 \leqslant i < j \leqslant d} \left \{{\alpha_i-\beta_i}{}  \right \}, \min_{\substack{1 \leq i <j \leq d \\ \alpha_i \neq \alpha_j}} \{\alpha_i-\alpha_j\} \Big \}.
	\end{equation}
	We have \ts $\ep >0$ \ts since \ts $\alpha_i > \beta_i$.
	
	Let \ts $\la\simeq \al\ts n$, \ts $\mu\simeq \be\ts n$  \ts be a TVK $(\al,\be)$-shape.
	Note that \ts $(\lambda,\mu)$ is $\ep$-admissible for sufficiently large \ts $n$.
	Indeed,
	\begin{equation}\label{eq:TVK is epsilon admissibile}
	\lambda_i \. - \. \mu_i \ =  \ \lfloor \alpha_i n \rfloor \. - \. \lfloor \beta_i n \rfloor
	\ \geq  \ (\alpha_i\ts -\ts \beta_i)\ts n \. - \. 1 \ \geq  \
	\frac{\alpha_i\ts -\ts \beta_i}{2\ts (\alpha_1\.+\. \ldots\. + \.\alpha_d)} \. |\lambda|
	\ \geq  \ \ep\ts |\lambda|\ts,
	\end{equation}
	for sufficiently large \ts $|\lambda|=|\al| \ts n+O(1)$,
	and for all \ts $1\le i \le d$.
	
	Let $\CB$ be the interval decomposition of~$[d]$ that puts two integers
	\ts $i, j \in [d]$ \ts in the same  block if and only if \ts $\alpha_i=\alpha_j$.
	Then~\eqref{eq:lambda big gap} holds for sufficiently large $|\lambda|$,
	since for all \ts $i \Bnsim j$, \ts $1 \leqslant i < j \leqslant d$, we have:
	\begin{equation}\label{eq:TVK condition 1 is met}
	\lambda_i \. - \. \la_j \ =  \ \lfloor \alpha_i n \rfloor \. - \. \lfloor \al_j n \rfloor
	\ \geq  \ (\alpha_i\ts - \ts \al_j)\ts n \. - \. 1 \ \geq  \
	\frac{\alpha_i\. - \. \al_j}{2\ts (\alpha_1\.+\. \ldots\. +\.\alpha_d)} \. |\lambda|\ts,
	\ \geq  \ \ep\ts |\lambda|\ts,
	\end{equation}
	Similarly, \eqref{eq:lambda small gap} holds, since for all
	\ts $i \Bsim j$, \ts $1 \leqslant i < j \leqslant d$, we have:
	\begin{equation}\label{eq:TVK condition 2 is met}
	\lambda_i \.- \. \lambda_j  \ =  \ \lfloor \alpha_i n \rfloor \. - \. \lfloor \alpha_j n \rfloor
	\, =  \, 0\ts.
	\end{equation}
	
	By Lemma~\ref{l:f -- TVK upper bound} and Lemma~\ref{l:f -- TVK lower bound}, this implies
	that there exists an absolute constant \ts $C_{\alpha,\beta}> 0$, such that
	\begin{align*}
	\frac{1}{C_{\alpha,\beta}} \, \eG(\lambda/\mu) \ \leq \ \frac{f(\lambda/\mu)}{F(\lambda/\mu)}   \ \leq \  {C_{\alpha,\beta}} \, \eG(\lambda/\mu)\,,
	\end{align*}
	for sufficiently large \ts $n$.  This implies the result. \ $\sq$

	\subsection{Proof of Theorem~\ref{t:TVK-skew}}
	Let $\ep:=\ep(\alpha,\beta)$ be as in \eqref{eq:definition of epsilon TVK}.
	By Lemma~\ref{l:main}, it suffices to check that for every $\ep$-admissible triplet
	$(\lambda,\gamma,\mu)$, we have:
	\begin{equation}\label{eq:TVK-123}
	\frac{f(\gamma/\mu)}{F(\gamma/\mu)}   \ \leq \  {C_{\alpha,\beta}} \,  \eG(\gamma/\mu)\.,
	\quad  \frac{f(\lambda/\gamma)}{F(\lambda/\gamma)}   \ \leq  \ {{C_{\alpha,\beta}}} \, \eG(\lambda/\gamma)\.,
	\quad \text{ and }  \quad
	\frac{f(\lambda/\mu)}{F(\lambda/\mu)}   \ \geq \  \frac{1}{{C_{\alpha,\beta}}} \, \eG(\lambda/\mu),
	\end{equation}
	for some absolute constant \ts ${C_{\alpha,\beta}} >0$.
	Note that the third inequality in \eqref{eq:TVK-123} is proved in Lemma~\ref{l:asy-TVK}.
	
	For the second inequality in~\eqref{eq:TVK-123},
	let $\CB$ be the interval decomposition of $[d]$ that puts two integers \ts
	$i, j \in [d]$ \ts in the same  block if and only if  $\alpha_i=\alpha_j$.
	By the same argument as in \eqref{eq:TVK condition 1 is met} and~\eqref{eq:TVK condition 2 is met},
	we have  \eqref{eq:lambda big gap} and \eqref{eq:lambda small gap} hold for the pair
	$(\lambda,\gamma)$ and $\CB$, and for sufficiently large~$n$.
	By Lemma~\ref{l:f -- TVK upper bound}, we get the second inequality
	for sufficiently large~$n$.
	
	For the first inequality in~\eqref{eq:TVK-123}, let $\CB'$ be the interval
	decomposition of $[d]$ that puts \ts $i, j \in [d]$ \ts
	in the same  block if and only if \ts $\beta_i=\beta_j$.
	Let \ts $\ep':=\ep'(\alpha,\beta)$ \ts be the constant defined by
	\[
	\ep' \, := \, \frac{d\ts \ep^3}{8\ts (\alpha_1+\ldots+\alpha_d)}
	\, \min_{\substack{1 \leqslant i <j \leqslant d \\ \beta_i \neq \beta_j}} \. \{\beta_i-\beta_j\}\ts.
	\]
	For all \ts $i  \overset{\CB'}{\nsim} j$, $1 \leqslant i < j \leqslant d$, we have:
	\begin{equation}\label{eq:TVK 4}
	\gamma_i\ts - \ts \gamma_j \, \geq_{\eqref{eq:progressive-def}} \,
	p\ts (\lambda_i-\lambda_{j}) \ts + \ts (1-p)\ts (\mu_i-\mu_{j})  \ts  -  \ts
	2\ts(|\lambda|-|\mu|)^{\frac{3}{4}} \, \geq  \, (1-p)\ts (\mu_i-\mu_{j}) \ts -  \ts 2\ts n^{\frac{3}{4}}.
	\end{equation}
	Note that
	\begin{equation}\label{eq:TVK 5}
	\mu_i- \mu_j \, = \,  \lfloor \beta_i n \rfloor \. - \.\lfloor \beta_j n \rfloor
	\, \geq \, (\beta_i-\beta_j)\ts |\la| \.  - \. 1 \, \geq  \,\.
	\frac{\beta_i-\beta_j}{2\ts (\alpha_1+\ldots+\alpha_d)} \. |\la|\ts,
	\end{equation}
	for sufficiently large~$n$.  Note also that
	\begin{equation}\label{eq:TVK 6}
	1- p \, =_{\eqref{eq:definition p}} \ \frac{|\lambda|-|\gamma|}{|\lambda|-|\mu|}
	\, =  \,  \sum_{i=1}^d \frac{\lambda_i-\gamma_i}{|\lambda|-|\mu|}
	\, \geq_{\eqref{eq:separated-def}} \,
	\frac{d\ep^3}{2} \. \frac{|\lambda|}{|\lambda|-|\mu|}
	\, \geq \,  \frac{d\ep^3}{2}\..
	\end{equation}
	Substituting \eqref{eq:TVK 5} and \eqref{eq:TVK 6} into \eqref{eq:TVK 4}, we get
	\begin{equation}\label{eq:TVK 7}
	\gamma_i-\gamma_j \, \geq \, \frac{d\ep^3}{4}\. \frac{\beta_i-\beta_j}{\alpha_1+\ldots+\alpha_d} \. |\la|
	\.  -  \. 2\ts n^{\frac{3}{4}} \, \geq \.
	\frac{d\ep^3}{8}\frac{\beta_i-\beta_j}{\alpha_1+\ldots+\alpha_d} \.|\la| \,
	\geq  \, \ep' \ts |\lambda| \, \geq  \, \ep'\ts |\gamma|,
	\end{equation}
	for \ts $|\la|=\Theta(n)$ \ts large enough.
	On the other hand, for all \ts  $i  \overset{B'}{\sim} j$, \ts
	$1 \leqslant i < j \leqslant d$, we have:
	\begin{equation}\label{eq:TVK 8}
	\mu_i \. - \. \mu_j  \ = \ \lfloor \beta_i n \rfloor \. - \. \lfloor \beta_j n \rfloor \, =  \, 0\ts.
	\end{equation}
	It follows from \eqref{eq:TVK 7} and \eqref{eq:TVK 8},
	that  \eqref{eq:lambda big gap} and \eqref{eq:mu small gap}
	hold for this case when~$n$ is sufficiently large.
	Thus, the first inequality in~\eqref{eq:TVK-123}
	follows by Lemma~\ref{l:f -- TVK upper bound}.
	This completes the proof of the theorem. \ $\sq$
	
	\bigskip
	
	\section{Conjectures and open problems}\label{s:conj}
	
	We believe our results can be further strengthened in several directions,
	and would like to mention a few possibilities.
	
	\subsection{Sorting probability}
	The bound \ts $\de\bigl(P_{\la/\mu}\bigr) = O\bigl(\frac{1}{\sqrt{n}}\bigr)$ \ts that we obtain in
	Theorems~\ref{t:TVK}--\ref{t:main} is likely not tight.  In fact, \ts
	$\Omega\bigl(\frac1n\bigr)$ \ts is the only lower bound that we know in some cases
	(see~$\S$\ref{ss:intro-examples}).  The results in Corollary~\ref{c:warmup-upper}
    and~\cite{CPP} also seem to suggest that \ts $O\left(\frac{1}{n}\right)$ \ts is
    perhaps the best one can aim for in full generality.  We believe the TVK shapes are likely
    the easiest case to make progress as they are most similar to the
    Catalan poset case:
	
\begin{conj}\label{conj:TVK-skew}  There is a universal constant \ts $C>0$,
such that for all~$d\ge 2$, and for every Thoma sequence \ts $\al\in \rr^d_{>0}$, we have:
$$
\de\bigl(P_{\la}\bigr) \. \le \. \frac{C}{n^{5/4}}\,,
$$
where \ts $\la\simeq \al \ts n$ \ts is a TVK \ts $\al$-shape.
\end{conj}
	
We believe the same bound holds for more general cases.  To understand
our reasoning, note that we take \ts $a=\lfloor\la_1/2\rfloor$ in this
case to minimize the sorting probability. 
Even if the bound we obtain is tight, by varying~$a$ one is
likely to obtain lower global minimum in the definition of
the sorting probability.  In fact, we believe the following general
claim with a weaker bound:
	
	\smallskip

\begin{conj}\label{conj:TVK-skew-more}
There is a universal constant \ts $C>0$, such that for every
$\la\vdash n$, \ts $\la\neq (n), (1^n)$, we have:
$$
		\de\bigl(P_{\la}\bigr) \. \le \.\frac{C}{\sqrt{n}}\..
$$
\end{conj}

\smallskip

This conjecture is suggesting that the constants \ts $C_{d,\ep}$ \ts in Theorem~\ref{t:thick} and Theorem~\ref{t:main}
can be made independent of parameters~$d$ and~$\ep$, even though
the proofs give dependence that is relatively wild. 
See, e.g.\ the last line of the proof of Theorem~\ref{t:interval}.   
At the moment, we cannot even prove that \ts $\de\bigl(P_{\la}\bigr)\to 0$ \ts 
for general partitions~$\la$, with \ts $n=|\la|\to \infty$.

\smallskip

In a different direction, suppose $\la$ is a $3$-dimensional diagram
defined as lower ideals in $\nn^3$.  The tools of this paper are
heavily based on the HLF~\eqref{eq:HLF}, NHLF~\eqref{eq:Naruse HLF},
asymptotics of Schur functions and other Algebraic Combinatorics
results.  None of these are available for $3$-dimensional diagrams,
even for the boxes (products of three chains).  Finding new tools 
to establish such bounds would be a major breakthrough.
	
	\begin{conj}\label{conj:3d-boxes}
		Fix~$d, r \ge 2$.  Denote by $P_{d,r,m}$ the $3$-dimensional poset
		given by a \ts $[d\times r \times m]\ssu \nn^3$ \ts box (product of chains 
on size \ts $d$, \ts and \ts $m$, respectively).  Then:
		$$
		\de\bigl(P_{d,r,m}\bigr) \. = \. O\Bigl(\frac{1}{m}\Bigr)\,,
		\ \ \text{as} \ \. m\to \infty.
		$$
	\end{conj}
	
	A more general problem would be to find conditions on the poset \ts
	$P=(X,\prec)$ \ts of \ts bounded width, which would guarantee
	that the sorting probability \ts $\de(P) \to 0$ \ts as the
	size \ts $|X|\to \infty$.
	
	\smallskip
	
	\subsection{Technical estimates}
	The tools of this paper are based on bounds for \ts $f(\la/\mu)=|\SYT(\la/\mu)|$,
	which are of independent interest.  Recall the definition of \ts
	$F(\la/\mu)$ \ts in~\eqref{eq:F-def} and the bound in
	Theorem~\ref{t:NHLF-asy}.  Recall also the balance function~$\eG(\lambda/\mu)$
	defined in~\eqref{eq:G-def} and the bounds in Lemmas~\ref{l:asy-smooth}
	and~\ref{l:asy-TVK}.  The following conjecture is a natural generalization.
	
	\begin{conj}\label{conj:number of skew-shaped diagrams asymptotic}
		Fix $d\ge 2$. Let $\la/\mu\vdash n$, $\ell(\la) \le d$. Then:
		\begin{equation}\label{eq:number of skew-shaped diagrams asymptotic}
		\frac{1}{C_{d}}\,\eG(\lambda/\mu)  \ \leq \  \frac{f(\lambda/\mu)}{F(\lambda/\mu)}
		\ \leq \ {C_{d}}\,\eG(\lambda/\mu)\ts,
		\end{equation}
		for an absolute constant \ts $C_{d}>0$.
	\end{conj}
	
	One can generalize the definition of \ts $\eG(\lambda/\mu)$ \ts to
	continuous setting:
	$$
	\eG(\bcx/\mu) \, := \,
	\prod_{1 \leqslant i < j \leqslant d} \. \min \left\{\mu_i-\mu_j+j-i, \frac{x_i}{x_i-x_j} \right\},
	$$
	where \ts $\bcx = (x_1,\ldots,x_d)\in \rr^d$, and \ts $\mu=(\mu_1,\ldots,\mu_d)$ \ts
	is an integer partition.
	
	\begin{conj}\label{conj:schur polynomial asymptotic}
		Fix $d\geq 2$ and $\ep >0$.
		Then, for every $\mu=(\mu_1,\ldots,\mu_d)$ \ts and
		\ts $\bcx = (x_1,\ldots,x_d)\in \rr^d$, such that \ts
		$x_1> \ldots > x_d > \ep \ts x_1>0$, we have:
		\begin{equation}\label{eq:conj-asy}
		\frac{s_\mu(x_1,\ldots, x_d)}{x_1^{\mu_1}\.\cdots \. x_d^{\mu_d}}
		\ \leq \ {C_{d,\ep}}\,\eG(\bcx/\mu)\ts,
		\end{equation}
		for an absolute constant \ts $C_{d,\ep} >0$.
	\end{conj}

	We obtain partial results in favor of this conjecture: a lower bound in
	Lemma~\ref{l:NHLF  upper bounded by schur polynomial} and an upper bound
	in Theorem~\ref{t:interval}.  Let us present the former with simplified
	notation, as it also gives connection between
	Conjectures~\ref{conj:number of skew-shaped diagrams asymptotic}
	and~\ref{conj:schur polynomial asymptotic}.
	
	\begin{thm}[{\rm $=$ Lemma~\ref{l:NHLF  upper bounded by schur polynomial}}]
		Let \ts $\la/\mu$ be a skew shape, where \ts
		$\lambda=(\la_1,\ldots,\la_d)$, $\mu=(\mu_1,\ldots,\mu_d)$.
		Then:
		\[
		1\ \leq \   \frac{f(\lambda/\mu)}{F(\lambda/\mu)}   \ \leq \ \frac{s_{\mu}(\lambda_1+d-1, \ts \lambda_2+d-2,\. \ldots \., \lambda_d)}{(\lambda_1+d-1)^{\mu_1} \ts (\lambda_2+d-2)^{\mu_2}\.\cdots \. \lambda_d^{\mu_d}}\,.
		\]
	\end{thm}
	
\smallskip

\begin{rem}\label{r:HCF} {\rm
Conjecture~\ref{conj:schur polynomial asymptotic} in the earlier version 
of the paper had a matching lower bound:
$$
\frac{1}{C_{d,\ep}} \ \eG(\bcx/\mu) \ \leq \ \frac{s_\mu(x_1,\ldots, x_d)}{x_1^{\mu_1}\.\cdots \. x_d^{\mu_d}}\,.
$$
Unfortunately, this bound fails for the substitution \. $x_i \gets q^i$ \. for \. 
$q \in \bigl[ \ep^{1/(d-1)}, 1\bigr)$, by the \emph{hook-content formula},
see \cite[Thm~7.21.2]{Sta99}.
On the other hand, the upper bound~\eqref{eq:conj-asy} is easy to check in this case. 
}
\end{rem}

	\bigskip
	
\section{Final remarks}\label{s:finrem}

\subsection{}\label{ss:finrem-1323}
Although much of the paper is motivated by the work surrounding the 
\. $\frac{1}{3}$--$\frac{1}{3}$ \ts Conjecture~\ref{c:1323}, we do not 
resolve the conjecture in any new cases.  As mentioned in the introduction,
for all skew Young diagrams the conjecture was already established in~\cite{OS}.  
In fact, when compared with the Kahn--Saks Conjecture~\ref{c:width-KS}, our 
results are counterintuitive since we obtain the conclusion of the conjecture
in a strong form, while the width of our posets remains bounded.  Clearly, 
much of the subject remains misunderstood and open to further exploration.   

\subsection{}\label{ss:finrem-thick}
The technical assumption in Theorem~\ref{t:thick},
that $\la$ is \ts $\ep$-thick is likely unnecessary,
but at the moment we do not know how to avoid it.
The same applies for the $\ep$-smooth assumption, and
the Main Theorem~\ref{t:main} most likely holds
under much weaker assumptions.  Let us remark though,
that in some formal sense these two assumptions are
equivalent.
Indeed, let \ts $\la=(\la_1,\ldots,\la_d)$ \ts and
$\mu=(\la_{i+1},\ldots, \la_{i+1}, \la_{i+2},\ldots,\la_d)$.
The skew shape \ts $\nu:=\la/\mu= (\la_1-\la_{i+1},\ldots,\la_i-\la_{i+1})$ \ts
is then the straight shape, so the $\ep$-smooth condition
\ts $\la_{i}-\la_{i+1} \ge \ep \ts n$ \ts becomes the
$\ep$-thick condition for~$\nu$.

\subsection{}\label{ss:finrem-thick-2}
For a fixed number of rows $d=\ell(\la)$,
Corollary~\ref{c:warmup-upper} shows that \ts
$\de(P_\la)=O\big(\frac1n\big)$ \ts for all \ts
$\la\vdash n$, such that  \ts $\la_2= O(1)$.
This is the opposite extreme of $\ep$-thick diagrams~$\la$,
suggesting that the $\ep$-thick assumption in
Theorem~\ref{t:thick} might be unnecessary indeed.

\subsection{}\label{ss:finrem-TVK-const}
The upper bound in~\eqref{eq:lambda small gap} and~\eqref{eq:mu small gap} can be
	replaced with an arbitrary constant~$K$ at the cost of changing the positive
	constant $C_{d,\ep}$ in our results into the positive constant $C_{d,\ep,K}$,
	which now also depends on $K$.  The rest of the proof follows verbatim and
	gives a slight extension of Theorem~\ref{t:TVK-skew} under weaker conditions \ts
	$\bigl|\la_i - \al_i \ts n\bigr|\le K$, and the same for the~$\mu$.
	We omit the details.
	
	\subsection{}\label{ss:finrem-OO}
	The Naruse's hook-length formula~\eqref{eq:Naruse HLF} works well
	when \ts $|\la/\mu|$ \ts is relatively small compared to~$|\la|$.
    On the other hand, when \ts $|\mu|$
	\ts is very small, there is another positive formula due to Okounkov
	and Olshanski \cite{OO}, which was observed in~\cite{OO,Sta99} to give
	sharp estimates in that regime.  In \cite[$\S$9.4]{MPP1}, the authors
	suggested that this rule is equivalent to the Knutson--Tao
    ``equivariant puzzles'' rule.
	This was proved in~\cite{MZ}, which reworked the Okounkov--Olshanski
	formula in the NHLF-style.  It would be interesting to see if this
	formula can be used in place of NHLF to obtain sharper bounds on the
	sorting probability of skew Young diagrams, at least in some cases.

\subsection{}\label{ss:finrem-rect}
When \ts $\la=(m^d)$ \ts is a rectangle, one can estimate \ts $\de(P_\la)$ \ts
without the NHLF, since \ts $f(\la/\mu)$ \ts can be computed by the
hook-length formula~\eqref{eq:HLF}.  This greatly simplifies the
calculations, and is an approach take in~\cite{CPP} for the
Catalan numbers example \ts $\la=\bigl(\frac{n}2,\frac{n}2\bigr)$, see~$\S$\ref{ss:intro-examples}.

	\subsection{}\label{ss:finrem-sqrt}
	As we mentioned in the previous section, there are several places
	where our bounds are likely not sharp.  First, the argument
	in~$\S$\ref{ss:paths-anti}, is a quantitative version of
	Linial's pigeonhole principle argument, which we also employ
    in~$\S$\ref{ss:warmup-proof}.  But the real obstacle to improving
	the \ts $O\bigl(\frac{1}{\sqrt{n}}\bigr)$ \ts bound is not apparent until
	Section~\ref{s:upper-bounds-SYTs}, where the interval decompositions
	are introduced and a different pigeonhole argument is used.

\subsection{}\label{ss:finrem-GG}
Most recently, Conjecture~\ref{c:1323} was generalized to all Coxeter
groups~\cite{GG}.  It would be interesting to see if our results extend
to this setting. 
	
	\vskip.6cm
	
	\subsection*{Acknowledgements}
	We are grateful to Han Lyu, Alejandro Morales and Fedya Petrov
	for many interesting discussions on the subject.  We are thankful
    to Vadim Gorin, Jeff Kahn, Martin Kassabov, Richard Stanley and Tom Trotter for useful
    comments.  Dan Romik kindly provided us with Figure~\ref{f:romik}.
We thank the anonymous referees for the careful reading of the paper, 
especially the suggestion which led us to Remark~\ref{r:HCF}. 
The first author was partially supported by the Simons Foundation. 
The last two authors were partially supported by the NSF.


\begin{dajauthors}
	\begin{authorinfo}[swee]
		Swee Hong Chan\\
		University of California Los Angeles\\
		Los Angeles, USA\\
		sweehong\imageat{}math\imagedot{}ucla\imagedot{}edu \\
		\url{https://www.math.ucla.edu/~sweehong/}
	\end{authorinfo}
	\begin{authorinfo}[igor]
		Igor Pak\\
		University of California Los Angeles\\
Los Angeles, USA\\
pak\imageat{}math\imagedot{}ucla\imagedot{}edu \\
\url{https://www.math.ucla.edu/~pak/}
	\end{authorinfo}
	\begin{authorinfo}[greta]
		Greta Panova\\
		University of Southern California\\
		Los Angeles, USA\\
		gpanova\imageat{}usc\imagedot{}edu\\
		\url{https://sites.google.com/usc.edu/gpanova/home}
	\end{authorinfo}
\end{dajauthors}
	
%
%

\end{document}